\documentclass[11pt,letterpaper]{amsart}
\usepackage[utf8]{inputenc}
\pdfoutput=1

\usepackage{amsmath,amssymb,amsthm,amsfonts}
\usepackage{mathtools}%

\usepackage{comment}%
\usepackage[bookmarksdepth=4]{hyperref}%
\usepackage{bbm}%
\usepackage{mathrsfs}%

\usepackage{tikz,graphicx,color}
\usepackage{tikz-cd}%
\usepackage{tikz-3dplot}%
\usetikzlibrary{calc}%
\usetikzlibrary{arrows}%
\usetikzlibrary{shapes}%
\usetikzlibrary{patterns}%
\usetikzlibrary{positioning}%
\usetikzlibrary{arrows.meta}
\usetikzlibrary{decorations.markings}
\usepackage{epstopdf}%
\pdfpageattr{/Group <</S /Transparency /I true /CS /DeviceRGB>>}

\usepackage{array}
\newcommand{\PreserveBackslash}[1]{\let\temp=\\#1\let\\=\temp}
\newcolumntype{C}[1]{>{\PreserveBackslash\centering}p{#1}}
\newcolumntype{R}[1]{>{\PreserveBackslash\raggedleft}p{#1}}
\newcolumntype{L}[1]{>{\PreserveBackslash\raggedright}p{#1}}

\usepackage[all]{xy}%
\usepackage{diagbox}%
\usepackage{subfig}%
\usepackage{arcs}%
\usepackage{xcolor}%
\usepackage{tabu}
\usepackage{booktabs}%

\usepackage[margin=1in]{geometry}

\usepackage{soul}
\usepackage{accents}

\usepackage{xspace}

\usetikzlibrary{knots}

\usepackage{enumitem}
\usepackage{letltxmacro}
\usepackage{thmtools,etoolbox}

\def\myarabic#1{\normalfont(\roman{#1})}
\newlist{theoremlist}{enumerate}{1}
\setlist[theoremlist]{label=\myarabic{theoremlisti},ref={\myarabic{theoremlisti}},itemindent=0pt,labelindent=0pt,
leftmargin=*,noitemsep}

\makeatletter
\renewcommand{\p@theoremlisti}{\perh@ps{\thetheorem}}
\protected\def\perh@ps#1#2{\textup{#1#2}}
\newcommand{\itemrefperh@ps}[2]{\textup{#2}}
\newcommand{\itemref}[1]{\begingroup\let\perh@ps\itemrefperh@ps\ref{#1}\endgroup}
\makeatother

\usepackage{nameref,hyperref}
\usepackage[capitalize]{cleveref}

\protected\def\ignorethis#1\endignorethis{}
\let\endignorethis\relax

\newtheorem{theorem}{Theorem}[section]

\newtheorem{lemma}[theorem]{Lemma}
\newtheorem{proposition}[theorem]{Proposition}
\newtheorem{corollary}[theorem]{Corollary}
\newtheorem{conjecture}[theorem]{Conjecture}
\theoremstyle{definition}
\newtheorem{remark}[theorem]{Remark}
\theoremstyle{definition}
\newtheorem{definition}[theorem]{Definition}

\theoremstyle{definition}
\newtheorem{problem}[theorem]{Problem}
\theoremstyle{definition}
\newtheorem{example}[theorem]{Example}

\crefname{figure}{Figure}{Figures}

\addtotheorempostheadhook[theorem]{\crefalias{theoremlisti}{theorem}}
\addtotheorempostheadhook[lemma]{\crefalias{theoremlisti}{lemma}}
\addtotheorempostheadhook[proposition]{\crefalias{theoremlisti}{proposition}}
\addtotheorempostheadhook[corollary]{\crefalias{theoremlisti}{corollary}}

\def\Acal{\mathcal{A}}\def\Fcal{\mathcal{F}}\def\Gcal{\mathcal{G}}\def\Ocal{\mathcal{O}}\def\Pcal{\mathcal{P}}\def\Rcal{\mathcal{R}}\def\Tcal{\mathcal{T}}\def\Xcal{\mathcal{X}}

\def\C{\mathbb{C}}
\def\R{\mathbb{R}}
\def\N{\mathbb{N}}
\def\Z{\mathbb{Z}}

\def\<{{\langle}}
\def\>{{\rangle}}

\def\la{{\lambda}}

\def\op{{ \operatorname{op}}}

\def\Span{ \operatorname{Span}}
\def\proj{ \operatorname{proj}}

\def\wt{\operatorname{wt}}

\def\id{{\operatorname{id}}}

\def\Lie{\operatorname{Lie}}

\def\j{{\mathbf{j}}}

\def\dv{\d{v}}
\def\da{\d{a}}

\def\Racc{R^\circ}
\def\Rich_#1^#2{\Racc_{#1,#2}}
\def\Richcl_#1^#2{R_{#1,#2}}

\def\Hom{\operatorname{Hom}}

\def\xrasim{\xrightarrow{\sim}}

\def\O{{\Ocal}}
\def\g{{\mathfrak{g}}}
\def\h{{\mathfrak{h}}}

\def\mod{{\operatorname{mod}}}
\def\Spec{\operatorname{Spec}}

\def\Hom{\operatorname{Hom}}
\def\Ext{\operatorname{Ext}}

\def\Semis{\operatorname{Semis}}

\def\SBim{{\operatorname{\mathbb{S}Bim}}}
\def\SMod{\operatorname{\mathbb{S}Mod}}
\def\IC{\mathcal{IC}}
\def\HH{H\!\!H}
\def\HHH{H\!\!H\!\!H}
\def\HHHC{\HHH_\C}
\def\H{{\mathbb{H}}}

\def\pt{{\rm pt}}
\def\mix{{\operatorname{mix}}}

\def\F{{\mathbb{F}}}
\def\Q{{\mathbb{Q}}}

\def\Mod{{\operatorname{Mod}}}
\def\Fr{{\rm Fr}}
\def\RHom{{\operatorname{RHom}}}

\def\ext{{\operatorname{ext}}}
\def\Fl{{\operatorname{Fl}}}

\def\For{{\operatorname{For}}}

\def\Mod{{\operatorname{Mod}}}
\def\Fr{{\rm Fr}}
\def\RHom{{\operatorname{RHom}}}

\def\ext{{\operatorname{ext}}}
\def\Fl{{\operatorname{Fl}}}

\def\For{{\operatorname{For}}}

\def\Pio{\Pi^\circ}

\def\DA{{\Dscr_1}}
\def\DB{{\Dscr_2}}

\def\BS{\operatorname{BS}}

\numberwithin{equation}{section}

\def\FLY{HOMFLY\xspace}

\def\Richafftp_#1^#2{{\Rcal_{#1,#2}^{>0}}}

\def\Richaff_#1^#2{\accentset{\circ}{\mathcal{R}}_{#1,#2}}

\def\Cat{C}
\def\Cab{\Cat_{a,b}}

\def\Ca{\Cat_n}
\newcommand{\qbinom}{\genfrac{[}{]}{0pt}{}}
\def\qint#1{[#1]_q}
\def\qbin#1#2{\qbinom{#1}{#2}_q}
\def\area{\operatorname{area}}
\def\dinv{\operatorname{dinv}}
\def\Dyck{\operatorname{Dyck}}
\def\Dyckxx(#1,#2){\Dyck_{#1,#2}}
\def\Dyckab{\Dyckxx(a,b)}

\def\Poinc{\mathcal{P}}

\def\df{d_f}

\def\Piokn{\Pio_{k,n}}

\def\Poincf{\Poinc(\Pio_f;}
\def\Poinckn{\Poinc(\Pio_{k,n};}

\def\Potf{\Poinc(\Pit_f;}

\def\Fvw{\FR_{v,w}}

\def\dkn{d_{k,n}}

\def\Tn{T}
\def\ncyc{\operatorname{ncyc}}
\def\ncyc{c}
\def\Bknc{\Bcal^{\ncyc=1}_{k,n}}

\def\qtn{\left(q^{\frac12}+t^{\frac12}\right)^{n-1}}

\def\betah{\hat\beta}

\def\fkn{{f_{k,n}}}
\def\Pit_#1{\Xcal^\circ_{#1}}

\def\Braid{\mathcal{B}}

\def\k{\mathbbm{k}}

\def\Ext{\operatorname{Ext}}

\def\mix{{\operatorname{mix}}}

\def\BS{\operatorname{BS}}

\def\KR{{\operatorname{KR}}}
\def\FR{F^\bullet}

\def\top{{\operatorname{top}}}

\def\PKR{\Pcal_\KR}
\def\PKRtop{\Pcal_\KR^\top}
\def\PK{\PKRtop}

\def\SL{\operatorname{SL}}

\def\Povar_#1{\accentset{\circ}{\Pi}_{#1}}
\def\Povarcl_#1{\Pi_{#1}}
\def\RPovar_#1{\accentset{\circ}{\Pi}^\R_{#1}}
\def\RPovarcl_#1{\Pi^\R_{#1}}
\def\Povtp_#1{\Pi_{#1}^{>0}}
\def\Povtnn_#1{\Pi_{#1}^{\geq0}}

\def\Fl{\operatorname{Fl}}

\def\Gr{\operatorname{Gr}}

\def\Cast{\C^\ast}

\def\KLR_#1^#2{R_{#1,#2}(q)}

\def\sclbx{1}

\def\op{1}

\def\gridop{0.4}

\def\drawelbow#1#2{
\def\a{#1}
\def\b{#2}
  \draw[line width=2pt,opacity=\op] (\a.5,\b) to[out=90,in=0] (\a,\b.5);
  \draw[line width=2pt,opacity=\op] (\a.5,1+\b) to[out=-90,in=180] (1+\a,\b.5);
}

\def\drawelbowfill#1#2#3{
\def\a{#1}
\def\b{#2}
  \fill[opacity=\gridop,fill=#3] (\a,\b) rectangle (1+\a,1+\b);
  \draw[line width=2pt,opacity=\op] (\a.5,\b) to[out=90,in=0] (\a,\b.5);
  \draw[line width=2pt,opacity=\op] (\a.5,1+\b) to[out=-90,in=180] (1+\a,\b.5);
}

\def\drawcrossing#1#2{
\def\a{#1}
\def\b{#2}
  \draw[line width=2pt,opacity=\op] (\a.5,\b) -- (\a.5,1+\b);
  \draw[line width=2pt,opacity=\op] (\a,\b.5) -- (1+\a,\b.5);
}

\def\godiagx#1#2#3{
\foreach \x/\y in {#1}{
\drawelbow{\x}{\y}
}
\foreach \x/\y in {#2}{
\drawcrossing{\x}{\y}
}
\foreach \x/\y in {#3}{
\drawcrossing{\x}{\y}
}
}

\def\drawtree#1{
\foreach \x/\y/\z/\bnd in {#1}{
  \node[black,fill=black,circle,scale=0.3] (A\x) at (\x.5,\z.7) {};
  \node[black,fill=black,circle,scale=0.3] (A\y) at (\y.5,\z.7) {};
  \draw[line width=1pt] (A\x) to[bend right=-\bnd] (A\y);
}
}

\def\dkn{d_{k,n}}
\def\df{d_f}

\def\KSBim{\Kb\SBim}
\def\HOMP{P}
\def\unkn{%
\begin{tikzpicture}[baseline=(ZUZU.base),scale=0.2]\coordinate(ZUZU) at (0,-0.6);
\draw[line width=0.7pt] (0,0) circle (1cm);
  \end{tikzpicture}
}

\def\unlk{
\scalebox{0.2}{%
\begin{tikzpicture}[baseline=(ZUZU.base)]\coordinate(ZUZU) at (0,-0.6);
\begin{knot}[
    clip width=6,
    ]
    \strand [line width=3.5pt] (0,0) circle (1.0cm);
    \strand [line width=3.5pt] (1,0) circle (1.0cm);
\end{knot}
\end{tikzpicture} 
}
}

\def\hopflink{
\scalebox{0.2}{%
\begin{tikzpicture}[baseline=(ZUZU.base)]\coordinate(ZUZU) at (0,-0.6);
\begin{knot}[
    clip width=6,
    flip crossing={1},
    ]
    \strand [line width=3.5pt] (0,0) circle (1.0cm);
    \strand [line width=3.5pt] (1,0) circle (1.0cm);
\end{knot}
\end{tikzpicture} 
}
}

\def\trefoil{\!
\scalebox{0.22}{%
\begin{tikzpicture}[baseline=(ZUZU.base)]\coordinate(ZUZU) at (0.5,-0.6);
\begin{knot}[
    clip width=6,
    consider self intersections
    ]
\strand[line width=3.5pt] (90:1)
            \foreach \x in {1,2,3} {
                to [bend left=117,looseness=1.9] ({90+120*\x}:1)
            }
        ;

        \flipcrossings{1,3}
\end{knot}
\end{tikzpicture} 
}\!
}

\def\top{{\operatorname{top}}}
\def\Ptop_#1{P^\top_{#1}(q)}
\def\Ptopnoq_#1{P^\top_{#1}}

\def\Symh{\C[\h]}

\def\O{\Ocal}
\def\g{\mathfrak{g}}

\def\mda(#1){\deg^\top_a(#1)}
\def\Pvw{\HOMP_{v,w}}

\def\Ptopvw{\HOMP^\top_{v,w}(q)}

\def\uu{{\underline{u}}}
\def\uv{{\underline{v}}}
\def\uw{{\underline{w}}}
\def\BS_#1{B_{#1}}

\def\bul{\bullet}
\def\und#1{\underline{#1}}
\def\BraidW{\Braid_W}
\def\BraidSn{\Braid_{S_n}}

\def\HHB#1#2(#3){H^{#1,(#2)}(\HH^0(#3))}
\def\HHBC#1#2(#3){H^{#1,(#2)}(\HHC^0(#3))}
\def\HHX#1#2_#3{H^{#1,(#2)}(\HH^0(\FR_{#3}))}
\def\HHXC#1#2_#3{H^{#1,(#2)}(\HHC^0(\FR_{#3}))}
\def\HHXk#1#2_#3{H^{#1,(#2)}(\HHk^0(\FR_{#3}))}

\def\ksa[#1]{[#1]}
\def\psa#1{\{#1\}}

\def\FR{F^\bul}

\def\Qlb{\bQ_\ell}
\def\etale{\'etale\xspace}

\def\Poincare{Poincar\'e\xspace}
\def\lvw{\ell_{v,w}}
\def\eps{\epsilon}

\def\da{\kappa}
\def\davw{{\da_{v,w}}}

\def\chib{{\chi(\beta)}}
\def\chibvw{{\chi(\bvw)}}
\def\cb{\ncyc(\beta)}
\def\cbvw{\ncyc(\bvw)}

\def\tt{t^{\frac12}}
\def\qq{q^{\frac12}}
\def\tti{t^{-\frac12}}
\def\qqi{q^{-\frac12}}
\def\bvw{\beta_{v,w}}
\def\bhvw{\betah_{v,w}}

\def\k{{\mathbbm{k}}}
\def\ku{{\underline{\k}}}

\def\Fqb{\overline{\F}_q}
\def\Qlb{\overline{\Q}_{\ell}}

\def\Dop{{\operatorname{D}}}
\def\Kop{{\operatorname{K}}}
\def\bop{{\operatorname{b}}}

\def\mix{{\operatorname{mix}}}
\def\real{{\bf {\operatorname{real}}}}
\def\perf{{\operatorname{perf}}}
\def\Dperf{\Dop_{\perf}}

\def\DbcH{\Dop^\bop_{(H)}}
\def\Db{\Dop^\bop}
\def\Dmix{\Dop^\mix}

\def\Fq{\F_q}
\def\Xq{X}

\def\Xqb{\Xq_{\Fqb}}

\def\Xqw{\Xw}
\def\Xw{{\accentset{\circ}{X}_w}}
\def\Xv{{\accentset{\circ}{X}_v}}
\def\Xu{{\accentset{\circ}{X}_u}}

\def\Xwcl{X_w}

\def\Xvm{\accentset{\circ}{X}_v^-}

\def\DB{{\Db_B(\Xq,\k)}}
\def\DBH{{\Db_H(Y,\k)}}

\def\DCH{{\Db_H(\Yqb,\k)}}

\def\Nabla{\nabla}

\def\proj{\pi}

\def\Rsemi_#1^#2{\Racc_{#1,\overline{#2}}}
\def\projvw{\proj}

\def\ibar{\overline{i}}

\def\qmap{r}
\def\kuw{\ku_{\Xw}}
\def\kuv{\ku_{\Xv}}
\def\kuqw{\ku_{\Xqw}}

\def\BA{\phi_v}
\def\amap{a}

\def\PKRC{\Pcal_{\KR;\C}^\top}

\def\Res{\operatorname{Res}}

\def\Hom{\operatorname{Hom}}
\def\RHom{\operatorname{RHom}}
\def\sRHom{R\mathcal{H}om}
\def\Ext{\operatorname{Ext}}
\def\Spec{\operatorname{Spec}}
\def\ext{\operatorname{ext}}
\def\Free{\operatorname{Free}}
\def\HHC{\HH_\C}
\def\HHk{\HH_\k}

\def\ExtDC{\Ext}

\def\For{\operatorname{For}}
\def\DeltaC{\Delta^{(B)}}

\def\Kb{\Kop^\bop}
\def\VERD{\mathbb{D}}
\def\dv{\dot v}
\def\Bkn{\mathbf{B}_{k,n}}
\def\Bknc{\mathbf{B}^{\ncyc=1}_{k,n}}
\def\ft{f}

\def\taukn{\tau_{k,n}}
\def\wtl{\tilde w}
\def\ut{\tilde u}
\def\vt{\tilde v}

\def\RowSpan{\operatorname{RowSpan}}

\def\PGL{\operatorname{PGL}}
\def\wt{r}

\def\strvw{N^\circ_{v,w}}
\def\strvwi{N^\circ_{v^{-1},w^{-1}}}

\def\Rgauge_#1^#2{R^{\circ,\Delta=1}_{#1,#2}}
\def\Ngauge{N^{\circ,\Delta=1}_{v,w}}

\def\drawbox(#1,#2){
\draw[black!50,dashed] (#1-0.5,#2-0.5) rectangle (#1+0.5,#2+0.5);
}

\def\drawgrid#1#2#3#4{
\foreach\i in {#1,...,#3}{
\foreach \j in {#2,...,#4}{
\drawbox(\i,\j)
}
}
}

\def\crossing{\scalebox{0.4}{\begin{tikzpicture}[baseline=(ZUZU.base)]
\coordinate(ZUZU) at (0,-0.5);\drawgrid{0}{0}{0}{0}\draw[line width=3pt, rounded corners=12] (0.00,-0.50)--(0.00,0.50);\draw[line width=3pt, rounded corners=12] (-0.50,0.00)--(0.50,0.00);
\end{tikzpicture}}\xspace} 
\def\elbow{\scalebox{0.4}{\begin{tikzpicture}[baseline=(ZUZU.base),xscale=-1]
\coordinate(ZUZU) at (0,-0.5);\drawgrid{0}{0}{0}{0}\draw[line width=3pt, rounded corners=12] (0.00,-0.50)--(0.00,0.00)--(0.50,0.00);\draw[line width=3pt, rounded corners=12] (-0.50,0.00)--(0.00,0.00)--(0.00,0.50);
\end{tikzpicture}}\xspace}
\def\fil{D}

\def\maxx{\operatorname{max}}
\def\Go{\operatorname{Deo}}
\def\Gom{\Go^{\maxx}}

\def\elb{\operatorname{elb}}
\def\xing{\operatorname{xing}}

\def\Hecke{\mathcal{H}}
\def\Htrace{\epsilon}
\def\Tor{\operatorname{Tor}}

\def\dimT{d}
\def\fmod{\bar f}

\let\simeq\cong

\def\fvw{{f_{v,w}}}

\def\Qkn{\mathbf{Q}_{k,n}}

\def\figref#1(#2){Figure~\hyperref[#1]{\ref*{#1}(#2)}}
\def\tabref#1(#2){Table~\hyperref[#1]{\ref*{#1}(#2)}}

\makeatletter
\def\subsubsection{\@startsection{subsubsection}{3}%
  \z@{.5\linespacing\@plus.7\linespacing}{-.5em}%
  {\normalfont\bfseries}}
\makeatother

\def\parag#1{\subsubsection{#1}}

\def\abxcup{\smile}

\def\Gsc{\dot G}
\def\Tsc{\dot T}
\def\Bsc{\dot B}
\def\Usc{\dot U}
\def\minn{\operatorname{min}}
\def\Xmin{X^{\minn}}

\def\lline#1{\overline{\overline{#1}}}
\def\sv{v}
\def\laction{$\bigwedge$-action\xspace}
\def\Kbul{K^\bul}

\begin{document}

\title{Positroids, knots, and $q,t$-Catalan numbers}
\author{Pavel Galashin}
\address{Department of Mathematics, University of California, Los Angeles, 520 Portola Plaza,
Los Angeles, CA 90025, USA}
\email{\href{mailto:galashin@math.ucla.edu}{galashin@math.ucla.edu}}

\author{Thomas Lam}
\address{Department of Mathematics, University of Michigan, 2074 East Hall, 530 Church Street, Ann Arbor, MI 48109-1043, USA}
\email{\href{mailto:tfylam@umich.edu}{tfylam@umich.edu}}
\thanks{P.G.\ was supported by an Alfred P. Sloan Research Fellowship and by the National Science Foundation under Grants No.~DMS-1954121 and No.~DMS-2046915. T.L.\ was supported by a von Neumann Fellowship from the Institute for Advanced Study and by grants DMS-1464693 and DMS-1953852 from the National Science Foundation.}

\begin{abstract}
We relate the mixed Hodge structure on the cohomology of open positroid varieties (in particular, their Betti numbers over $\mathbb{C}$ and point counts over $\mathbb{F}_q$) to Khovanov--Rozansky homology of associated links. We deduce that the mixed Hodge polynomials of top-dimensional open positroid varieties are given by rational $q,t$-Catalan numbers.  Via the curious Lefschetz property of cluster varieties, this implies the $q,t$-symmetry and unimodality properties of rational $q,t$-Catalan numbers.  We show that the $q,t$-symmetry phenomenon is a manifestation of Koszul duality for category $\mathcal{O}$, and discuss relations with open Richardson varieties and extension groups of Verma modules.
\end{abstract}

\subjclass[2010]{
  Primary:
  14M15. %
  Secondary:
  14F05, %
  05A15, %
  57K18. %
}

\keywords{Positroid varieties, $q,t$-Catalan numbers, HOMFLY polynomial, Khovanov--Rozansky homology, mixed Hodge structure, equivariant cohomology, Verma modules, Koszul duality.
}

\date{\today}

\maketitle

\setcounter{tocdepth}{1}

\vspace{-0.09in}

\tableofcontents

\section*{Introduction}
The binomial coefficients $n\choose k$ have natural $q$-analogs $\qbin{n}k$, known as Gaussian polynomials. On the other hand, the rational Catalan numbers $\Cat_{k,n-k}:=\frac1n {n\choose k}$ (defined for $\gcd(k,n)=1$) have two different well-studied $q$-analogs: the area generating function $\sum_{P\in\Dyck_{k,n-k}} q^{\area(P)}$ of rational Dyck paths~\cite{CaRi}, and the polynomial $\frac{1}{\qint{n}}\qbin{n}k$ going back to~\cite{MacMahon}.

The \Poincare polynomial of the complex Grassmannian $\Gr(k,n)$, and the number of points of $\Gr(k,n)$ over a finite field $\F_q$, are both well known to be given by $\qbin{n}k$.  We give a Catalan analog of this statement by considering the top-dimensional positroid variety $\Piokn\subset\Gr(k,n)$, introduced in~\cite{KLS} building on the results of~\cite{Pos}. The space $\Piokn$ is the subspace of $\Gr(k,n)$ where all cyclically consecutive Pl\"ucker coordinates are non-vanishing.  We show that, up to a simple factor, the mixed Hodge polynomial $\Poinc(\Piokn;q,t)$ coincides with the rational $q,t$-Catalan number $\Cat_{k,n-k}(q,t)$ introduced in~\cite{LoWa} in the study of Macdonald polynomials~\cite{GaHa,HagBook}.  It follows that the \Poincare polynomial of $\Piokn$ equals $\sum_{P\in\Dyck_{k,n-k}} q^{\area(P)}$, while the point count $\#\Piokn(\F_q)$ equals $\frac{1}{\qint{n}}\qbin{n}k$, both up to a simple factor. 

The coincidence of the \Poincare polynomial and the point count of $\Gr(k,n)$ is reflected in the purity of the mixed Hodge structure on the cohomology of $\Gr(k,n)$.  Purity holds for many spaces of interest in combinatorics, e.g. for complements of hyperplane arrangements.  By contrast, the mixed Hodge structure on $H^\bul(\Piokn)$ is not pure, and simultaneously yields both of the natural $q$-analogs of rational Catalan numbers discussed above.

Our proof proceeds via relating both sides to Khovanov--Rozansky knot homology~\cite{KR1,KR2,KhoSoe}.  Our main result connects the cohomology of arbitrary open positroid varieties, and more generally open Richardson varieties in generalized flag varieties, to knot homology.

Connections between knot invariants and Macdonald theory have received an enormous amount of attention in recent years; see e.g.~\cite{Cherednik,GORS,GN,Haglund,Mellit_torus,ORS,EH,MeHog}.  In particular, Khovanov--Rozansky homology of torus knots and links was computed in~\cite{Mellit_torus,HogCat,EH,MeHog}. For torus knots, the answer is given by the rational $q,t$-Catalan numbers.

Our main results are described in detail in the next section. We start by highlighting some consequences of our approach from several points of view.

{\bf Combinatorics.} The coefficients of the Gaussian polynomial $\qbin nk$ are well known to form a unimodal palindromic sequence. A geometric explanation for this phenomenon is the hard Lefschetz theorem for the cohomology of $\Gr(k,n)$. It follows from the results of~\cite{LS,GL} that the cohomology of $\Pio_{k,n}$ satisfies the curious Lefschetz property which, combined with our main result, yields a geometric proof that $\Cat_{k,n-k}(q,t)$ is $q,t$-symmetric and unimodal.  Furthermore, our work produces a whole family of $q,t$-symmetric and unimodal polynomials, which includes $\Cat_{k,n-k}(q,t)$ as a special case.  We discuss their $q=t=1$ specialization in \cref{sec:Deogram} where we obtain a new combinatorial interpretation for rational Catalan numbers in terms of certain kinds of pipe dreams; see \cref{fig:Go_3_8}.

{\bf Knot theory.} 
We introduce a class of \emph{Richardson links}, which are closures of braids of the form $\beta(w)\cdot \beta(v)^{-1}$ for pairs of permutations $v,w\in S_n$ such that $v\leq w$ in the Bruhat order. We give a geometric interpretation of the top $a$-degree coefficient\footnote{The top $a$-degree coefficient encodes the zeroth Hochschild cohomology~\eqref{eq:HH0} which sometimes corresponds to the \emph{bottom} $a$-degree in the literature. Our conventions are chosen so that the $a$-degree in KR homology matches the $a$-degree in \FLY polynomial.} of Khovanov--Rozansky (KR) homology and of the \FLY polynomial~\cite{HOMFLY,PT} for such links. When a Richardson link is a knot, we show that the associated $q,t$-polynomial is $q,t$-symmetric.  Our investigations suggest that KR homology may have hitherto unstudied unimodality and Lefschetz-type properties. 
 Our results generalize equally well to other Dynkin types.%

{\bf Representation theory.} 
We show that the $q,t$-symmetry property is a consequence of the Koszul duality phenomenon~\cite{BGS,BY} in the derived category of the flag variety.   

The computation of the extension groups $\Ext^\bul(M_v,M_w)$ between Verma modules in the principal block $\O_0$ of the Bernstein–Gelfand–Gelfand category~$\O$ (see e.g.~\cite{Humphreys}) is a classical, still open problem.  We show that these extension groups are isomorphic to knot-homology groups.  Along the same vein, we show that the $R$-polynomials of Kazhdan and Lusztig~\cite{KL1,KL2} are certain coefficients of the HOMFLY polynomial.

{\bf Algebraic geometry.} Our results provide evidence for a $P=W$ conjecture relating the weight filtration of $\Pio_{k,n}$ with the perverse filtration of the compactified Jacobian $J_{k,n-k}$; see \cref{sec:STZ}.

\label{sec:main-results}
\section{Main results}
We give a detailed description of our main results.  The historical context and motivation for our work is delayed to \cref{sec:notes}. We give the full background on the below objects in the main body of the paper.

\subsection{Rational $q,t$-Catalan numbers}

Let $a$ and $b$ be coprime positive integers. The \emph{rational $q,t$-Catalan number} $\Cab(q,t)\in\N[q,t]$ was introduced by Loehr--Warrington~\cite{LoWa} (see also~\cite{GoMa1,GoMa2}), generalizing the work of Garsia--Haiman~\cite{GaHa}. It is defined as 
\begin{equation}\label{eq:qt_cat}
  \Cab(q,t):=\sum_{P\in\Dyckab} q^{\area(P)} t^{\dinv(P)},
\end{equation}
where $\Dyckab$ is the set of lattice paths $P$ inside a rectangle of height $a$ and width $b$ that stay above the diagonal, $\area(P)$ is the number of unit squares fully contained between $P$ and the diagonal, and $\dinv(P)$ is the number of pairs $(h,v)$ satisfying the following conditions: $h$ is a horizontal step of $P$, $v$ is a vertical step of $P$ that appears to the right of $h$, and there exists a line of slope $a/b$ (parallel to the diagonal) intersecting both $h$ and $v$.
 For example, 
\begin{equation}\label{eq:Cat_5_3}
\Cat_{3,5}(q,t)=q^{4} + q^{3} t + q^{2} t^{2} + q^{2} t + q t^{3} +q t^{2} +  t^{4},
\end{equation}
as shown in \cref{fig:Cab_ex}. 

\begin{figure}

\def\dgr{
\draw[line width=0.3pt,dashed] (0,0) grid (5,3);
\draw[line width=0.5pt,red] (0,0)--(5,3);
}
\def\dyckscl{0.33}
\def\resc#1{\scalebox{2.5}{#1}}
\def\plw{3pt}
\def\pcol{blue}
\scalebox{\dyckscl}{
\begin{tabular}{ccccccc}
\begin{tikzpicture}
\dgr
\draw[line width=\plw,draw=\pcol] (0,0)--(0,3)--(5,3);
\end{tikzpicture}
&
\begin{tikzpicture}
\dgr
\draw[line width=\plw,draw=\pcol] (0,0)--(0,2)--(1,2)--(1,3)--(5,3);
\end{tikzpicture}
&
\begin{tikzpicture}
\dgr
\draw[line width=\plw,draw=\pcol] (0,0)--(0,2)--(2,2)--(2,3)--(5,3);
\end{tikzpicture}
&
\begin{tikzpicture}
\dgr
\draw[line width=\plw,draw=\pcol] (0,0)--(0,1)--(1,1)--(1,3)--(5,3);
\end{tikzpicture}
&
\begin{tikzpicture}
\dgr
\draw[line width=\plw,draw=\pcol] (0,0)--(0,1)--(1,1)--(1,2)--(2,2)--(2,3)--(5,3);
\end{tikzpicture}
&
\begin{tikzpicture}
\dgr
\draw[line width=\plw,draw=\pcol] (0,0)--(0,2)--(3,2)--(3,3)--(5,3);
\end{tikzpicture}
&
\begin{tikzpicture}
\dgr
\draw[line width=\plw,draw=\pcol] (0,0)--(0,1)--(1,1)--(1,2)--(3,2)--(3,3)--(5,3);
\end{tikzpicture}\\

\resc{$q^4t^0$} &\resc{$q^3t^1$}&\resc{$q^2t^2$}&\resc{$q^2t^1$}&\resc{$q^1t^3$}&\resc{$q^1t^2$}&\resc{$q^0t^4$}

\end{tabular}
}
  \caption{\label{fig:Cab_ex} Computing the rational $q,t$-Catalan number $\Cat_{3,5}(q,t)$.}
\end{figure}
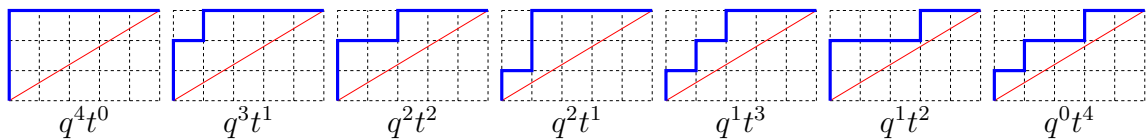

\subsection{Positroid varieties in the Grassmannian}\label{sec:positr-vari-grassm}
The \emph{Grassmannian} $\Gr(k,n)$ is the space of $k$-dimensional linear subspaces of $\C^n$.  Building on Postnikov's cell decomposition~\cite{Pos} of its totally nonnegative part, Knutson--Lam--Speyer~\cite{KLS} 
constructed a stratification
\begin{equation}\label{eq:Gr_strat}
  \Gr(k,n)=\bigsqcup_{f\in\Bkn} \Pio_f,
\end{equation}
where the \emph{(open) positroid varieties} $\Pio_f$ are defined as the non-empty intersections of cyclic rotations of $n$ Schubert cells.  These varieties also arise in Poisson geometry~\cite{BGY} and in the study of scattering amplitudes~\cite{abcgpt}. Open positroid varieties are indexed by a finite set $\Bkn$ of \emph{bounded affine permutations}, and for $f \in \Bkn$ the reduction of $f$ modulo $n$ is a permutation $\fmod \in S_n$. See \cref{sec:positroid-varieties} for further background.
 
 Let $\fkn \in \Bkn$ be the bounded affine permutation given by $\fkn(i) = i+k$.  The positroid stratification~\eqref{eq:Gr_strat} contains a unique open stratum, the \emph{top-dimensional positroid variety} $\Piokn:=\Pio_{\fkn}$, which can be described explicitly as
\begin{equation}\label{eq:Piokn_dfn}
  \Piokn:=\{V \in \Gr(k,n) \mid \Delta_{1,2,\dots,k}(V),\Delta_{2,3,\dots,k+1}(V),\ldots,\Delta_{n,1,\dots,k-1}(V)\neq0\},
\end{equation}
consisting of subspaces whose cyclically consecutive Pl\"ucker coordinates are non-vanishing.

 For each $f\in \Bkn$, the space $\Pio_f$ is a smooth algebraic variety. Two basic questions one can ask are:

\begin{enumerate}
\item\label{item:Q1} What is the number of points in $\Pio_f(\F_q)$ over a finite field $\F_q$ with $q$ elements?
\item\label{item:Q2} What are the Betti numbers of $\Pio_f$ considered as a complex manifold?
\end{enumerate}

These two questions are related by the \emph{mixed Hodge structure}~\cite{Del71} on cohomology.   The cohomology ring $H^\bul(\Pio_f)=H^\bul(\Pio_f,\C)$ of an open positroid variety is of Hodge--Tate type, and we have a \emph{Deligne splitting} 
\begin{equation}\label{eq:DS}
H^k(\Pio_f, \C) \simeq \bigoplus_{p\in\Z} H^{k,(p,p)}(\Pio_f,\C).
\end{equation}
Since $\Pio_f$ is smooth, we have that $H^{k,(p,p)}$ vanishes unless $k/2\leq p \leq k$.  We view \eqref{eq:DS} as a bigrading on $H^\bul(\Pio_f)$ and let $\Poincf q,t)$ be the suitably renormalized (see~\eqref{eq:Poinc_MHT_dfn}) \Poincare polynomial of this bigraded vector space, called the \emph{mixed Hodge polynomial}.

 We are ready to state the most important special case of our main result.
\begin{theorem}\label{thm:master_kn}
Assume that $\gcd(k,n)=1$. Then 
\begin{equation}\label{eq:Poinc=Cat}
  \Poinckn q,t)=\qtn \Cat_{k,n-k}(q,t).
\end{equation}
\end{theorem}
The equality \eqref{eq:Poinc=Cat} arises as a conjecture from the works \cite{STZ,STWZ} and we thank Vivek Shende for drawing our attention to the conjecture; see \cref{sec:STZ} for further discussion. We generalize \cref{thm:master_kn} to all positroid varieties in \cref{thm:main} below.

Let us discuss the specializations of \cref{thm:master_kn} that give answers to Questions~\eqref{item:Q1} and~\eqref{item:Q2} above. 
 Denote 
\begin{equation*}
  \qint{n}:=1+q+\dots+q^{n-1},\quad \qint{n}!:=\qint{1}\qint{2}\cdots\qint{n},\quad \qbin{n}k:=\frac{\qint{n}!}{\qint{k}!\qint{n-k}!}.
\end{equation*}

\begin{corollary}\label{cor:master}
The \Poincare polynomial and point count of $\Pio_{k,n}$ are
\begin{align}
 \label{eq:cor:master:Poinc}
 \Pcal(\Pio_{k,n}; q) &= (q+1)^{n-1}\cdot  \Cat_{k,n-k}(q^2,1);\\
 \label{eq:point_count}
  \#\Piokn(\F_q)&=(q-1)^{n-1}\cdot \frac{1}{\qint{n}}\qbin{n}k. 
  \end{align}
\end{corollary}

Our proof of \cref{thm:master_kn} involves a number of ingredients, including Khovanov--Rozansky knot homology and  derived categories of flag varieties. The point count specialization~\eqref{eq:point_count} requires less advanced machinery and we give a quicker elementary proof in \cref{sec:point-count-fly}.   Associating a \emph{link} $\betah_f$ to each positroid variety $\Pio_f$ (\cref{sec:knots-assoc-positr}), we compare the point count $\#\Pio_f(\F_q)$ to the \FLY polynomial of $\betah_f$ (\cref{sec:HOMFLY}). The \FLY polynomial is categorified by Khovanov--Rozansky knot homology, and our proof of \cref{thm:master_kn} may be considered a ``categorification'' of the point count computation.

We have the following elegant but baffling corollary.
\begin{corollary}
Let $\gcd(k,n)=1$. Then the probability that a uniformly random $k$-dimensional subspace of $(\F_q)^n$ 
 belongs to $\Piokn(\F_q)$ is given by
\begin{equation*}
{\rm Prob}(V \in  \Piokn(\F_q)) =  \frac{(q-1)^n}{q^n-1}.
\end{equation*}
\end{corollary}
\noindent The probability $\frac{(q-1)^n}{q^n-1}$ does not depend on $k$. We do not have a direct explanation for this phenomenon.

\subsection{Cluster structure and the curious Lefschetz theorem}
Since the work of Scott~\cite{Scott}, positroid varieties have been expected to admit a natural \emph{cluster algebra}~\cite{FZ} structure arising from Postnikov diagrams. We recently proved this conjecture building on the results of~\cite{Lec,MuSp,SSBW}.
\begin{theorem}[{\cite{GL}}]\label{thm:GL}
  The coordinate ring of each positroid variety $\Pio_f$ is isomorphic to the associated cluster algebra.
\end{theorem}

This result allows one to study $\Pio_f$ as a \emph{cluster variety}, and for such spaces the mixed Hodge structure can be explored using the machinery developed by Lam--Speyer~\cite{LS}, whose work implies the following properties of the mixed Hodge polynomials $\Poincf q,t)$.

\begin{theorem}[\cite{LS},\cite{GL}]\label{thm:LS}
  For each $f\in\Bkn$, the \emph{mixed Hodge polynomial} $\Poincf q,t)\in\N[q^{\frac12},t^{\frac12}]$ has the following properties:
\begin{theoremlist}
\item\label{thm:LS:qt_symm} $q,t$-symmetry: $\Poincf q,t)=\Poincf t,q)$;
\item\label{thm:LS:qt_unim} $q,t$-unimodality: for each $d$, the coefficients of $\Poincf q,t)$ at $q^dt^0$, $q^{d-1}t^1$, \dots, $q^0t^d$ form a unimodal sequence;
\item\label{thm:LS:poincare}  $\Poincf 1,q^2)$ equals the \Poincare polynomial of $\Pio_f$ (considered as a variety over $\C$);
\item\label{thm:LS:point_cnt} $q^{\frac12\dim\Pio_f}\cdot \Poincf q,t) |_{t^{\frac12}=-q^{-\frac12}}$ equals the point count $\#\Pio_f(\F_q)$.
\end{theoremlist}
\end{theorem}
\noindent See \cref{ex:E8} below. 

Parts~\itemref{thm:LS:qt_symm} and~\itemref{thm:LS:qt_unim} are consequences of the \emph{curious Lefschetz property}, formalized in~\cite{HR} and proven to hold for certain cluster varieties in~\cite{LS}; see \cref{sec:mixed-hodge-struct}.

\subsection{The Catalan variety} \label{sec:catalan-variety}
Let $\Tn\cong(\Cast)^{n-1}$ be the group of diagonal matrices in $\PGL_n(\C)$: it is the quotient of the group of diagonal $n\times n$ matrices by the group of scalar matrices. The group $\Tn$ acts on $\Gr(k,n)$ preserving the positroid stratification. 
For $u\in S_n$, let
\begin{equation}\label{eq:ncyc_dfn}
  \ncyc(u):= \mbox{ the number of cycles of $u$},
\end{equation}
and let $\Bknc:=\{f\in\Bkn\mid \ncyc(\fmod)=1\}$.
The following observation is proved in \cref{sec:torus-acti-rich}.
\begin{proposition}\label{prop:free}
The action of $\Tn$ on $\Pio_f$ is free if and only if the permutation $\fmod$ is a single cycle.
\end{proposition}
For $f\in\Bknc$, the quotient $\Pit_f:= \Pio_f/\Tn$ is again a smooth affine variety that we call a {\it positroid configuration space}; see also \cite{AHLS}. It is a cluster variety (with no frozen variables, since the $\Tn$-action on the frozen variables of $\Pio_f$ is free), and thus \cref{thm:LS} applies to it.  
\begin{proposition}\label{prop:mod_T_qtn}
For $f\in\Bknc$, the mixed Hodge polynomials of $\Pio_f$ and $\Pit_f$ are related by:
\begin{equation}\label{eq:mod_T_qtn}
  \Poincf q,t)=\qtn \cdot \Pcal(\Pit_{f};q,t).
\end{equation}
\end{proposition}
When $\gcd(k,n) = 1$, we have $\fkn \in \Bknc$.  The quotient $\Pit_{k,n}:= \Pio_{k,n}/\Tn$ satisfies
\begin{equation}\label{eq:Poinc=Cat2}
\Pcal(\Pit_{k,n};q,t)=\Cat_{k,n-k}(q,t),
\end{equation}
and we refer to $\Pit_{k,n}$ as \emph{the Catalan variety}. 
Let $\dkn:=(k-1)(n-k-1)=\dim(\Pit_{k,n})$.  We obtain the following as a consequence of \cref{thm:LS}.
\begin{corollary}\label{cor:master_kn}
Assume that $\gcd(k,n)=1$. We have:
\begin{theoremlist}
\item\label{cor:master:qt_symm} $q,t$-symmetry: $\Cat_{k,n-k}(q,t)=\Cat_{k,n-k}(t,q)$;
\item\label{cor:master:qt_unim} $q,t$-unimodality: for each $d$, the coefficients of $\Cat_{k,n-k}(q,t)$ at $q^dt^0$, $q^{d-1}t^1$, \dots, $q^0t^d$ form a unimodal sequence;
\item\label{cor:master:poincare}  the \Poincare polynomial of $\Pit_{k,n}$ is given by 
\begin{equation}\label{eq:Piokn_Poinc}
\sum_{d} q^{\frac d2} \dim H^{\dkn-d}(\Pit_{k,n})=\Cat_{k,n-k}(q,1)=\sum_{P\in\Dyckxx(k,n-k)} q^{\area(P)};
\end{equation}
\item\label{cor:master:point_cnt} the number of $\F_q$-points of $\Pit_{k,n}$ is given by
\begin{equation}\label{eq:Pit_kn_Fq}
  \#\Pit_{k,n}(\F_q)=\frac{1}{\qint{n}}\qbin{n}k=q^{\frac12\dkn}\cdot  \Cat_{k,n-k}(q,1/q).
\end{equation}
\end{theoremlist}
\end{corollary}
\noindent  While part~\itemref{cor:master:qt_symm} is known, the remaining parts of \cref{cor:master_kn} appear to be new; see \cref{sec:notes:catalan}. Note also that the odd Betti numbers of $\Pit_{k,n}$ vanish, a phenomenon that we do not have a direct explanation for. Parts~\itemref{cor:master:poincare}--\itemref{cor:master:point_cnt} may be deduced directly from \cref{cor:master} using \cref{prop:free}.

\begin{example}\label{ex:E8}
\begin{table}
\begin{tabular}{|c|ccccccccc|}\hline
$H^k$ & $H^0$ & $H^1$ & $H^2$ & $H^3$ & $H^4$ & $H^5$ & $H^6$ &$H^7$ & $H^8$ \\\hline
$k-p=0$&  $1$ &  $0$ &  $1$ &  $0$ &  $1$ &  $0$ &  $1$ &  $0$ &  $1$ \\
$k-p=1$&      &      &      &      &  $1$ &  $0$ &  $1$ & & \\\hline
\end{tabular}
  \caption{\label{tab:E8} The mixed Hodge table recording the dimensions of $H^{k,(p,p)}(\Pit_{3,8})$ for the cluster algebra of type $E_8$; see \cite[Table~5]{LS}. The dimensions agree with the coefficients of $\Cat_{3,5}(q,t)$; see \cref{ex:E8}.}
\end{table}
Let $k=3$ and $n=8$.  The coordinate ring of $\Pit_{3,8}$ is a cluster algebra of type $E_8$ (with no frozen variables). The associated \emph{mixed Hodge table} is given in \cref{tab:E8}; see~\cite[Table~5]{LS}. The grading conventions~\eqref{eq:Poinc_MHT_dfn} are chosen so that the first row contributes $q^{4} + q^{3} t + q^{2} t^{2} + q t^{3} +  t^{4}$ while the second row contributes $q^{2} t + q t^{2}$ to $\Pcal(\Pit_{3,8}; q,t)$. Note that all odd cohomology groups vanish, which is why all monomials have integer powers of $q$ and $t$. Comparing the result with~\eqref{eq:Cat_5_3}, we find $\Pcal(\Pit_{3,8}; q,t)=\Cat_{3,5}(q,t)$, in agreement with \cref{thm:master_kn}.

The polynomial $\Cat_{3,5}(q,t)$ given in~\eqref{eq:Cat_5_3} is indeed $q,t$-symmetric. It is also $q,t$-unimodal: fixing the total degree of $q$ and $t$, it splits into polynomials $q^{4} + q^{3} t + q^{2} t^{2} + q t^{3} +  t^{4}$ and $q^{2} t + q t^{2}$, both of which have unimodal coefficient sequences, corresponding to the rows of \cref{tab:E8}. We also have $\Cat_{3,5}(q,1)=q^4+q^3+2q^2+2q+1$; the coefficient of $q^{d/2}$ is equal to $\dim H^{\dkn-d}(\Pio_{3,8})$ for each $d$ (these coefficients are column sums in \cref{tab:E8}).
 This agrees with~\eqref{eq:Piokn_Poinc}.
\end{example}

\subsection{Links associated to positroid varieties}\label{sec:knots-assoc-positr}
Let us say that a permutation $w\in S_n$ is \emph{$k$-Grassmannian} if $w^{-1}(1)<w^{-1}(2)<\dots<w^{-1}(k)$ and $w^{-1}(k+1)<\dots<w^{-1}(n)$. We denote by $\leq$ the (strong) Bruhat order on $S_n$. Let $\Qkn$ denote the set of pairs $(v,w)$ of permutations such that $v\leq w$ and $w$ is $k$-Grassmannian. The following result is well known; see \cref{prop:v_w_aff}.

\begin{proposition}[\cite{KLS}] \label{prop:v_w}
There exists a bijection $(v,w)\mapsto \fvw$ between $\Qkn$ and $\Bkn$ such that for every $f=\fvw\in\Bkn$, we have $\fmod=wv^{-1}$.
\end{proposition}

\noindent For example, when $f=\fkn$, we have $v=\id$ and the permutation $w=\fmod$ sends $i\mapsto i+k$ modulo $n$ for all $i\in[n]$. The dimension of $\Pio_f$ equals $\lvw:=\ell(w)-\ell(v)$, where $\ell(u)$ is the number of inversions of $u\in S_n$. 

The group $S_n$ is generated by simple transpositions $s_i=(i,i+1)$ for $1\leq i\leq n-1$. Similarly, let $\Braid_n$ be the \emph{braid group} on $n$ strands, generated by $\sigma_1,\dots,\sigma_{n-1}$ with relations $\sigma_i\sigma_{i+1}\sigma_i=\sigma_{i+1}\sigma_i\sigma_{i+1}$ and $\sigma_i\sigma_j=\sigma_j\sigma_i$ for $|i-j|>1$. Connecting the corresponding endpoints of a braid $\beta$ gives rise to a \emph{link} called the \emph{closure} $\betah$ of $\beta$; see \cref{fig:braids}.

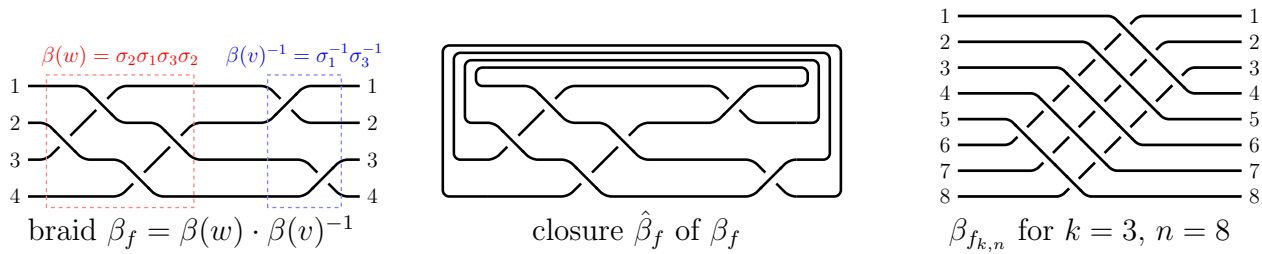
\begin{figure}

\def\rc{5}
\def\lw{2pt}

\def\whdiam(#1,#2){
\fill[white] (#1-\whlw,#2)--(#1,#2-\whlw)--(#1+\whlw,#2)--(#1,#2+\whlw)--cycle;
}
\def\whdiamr(#1,#2){
\whdiam(#1,#2)
\draw[line width=\lw] (#1-\whlww,#2-\whlww)--(#1+\whlww,#2+\whlww);
}
\def\whdiaml(#1,#2){
\whdiam(#1,#2)
\draw[line width=\lw] (#1+\whlww,#2-\whlww)--(#1-\whlww,#2+\whlww);
}

\def\whlw{0.5}
\def\whlww{0.25}
\def\yscl{0.7}
\def\xscl{1}
\def\sclbx{0.5}

\def\hmar{0.5}
\def\vmar{0.5}

\def\hmarr{0.8}
\def\vmarr{0.7}

\def\hmarrr{1.1}
\def\vmarrr{0.9}

\def\hmarrrr{1.4}
\def\vmarrrr{1.1}

\def\recmar{0.3}
\def\ndscl{1.4}
\centering

\setlength{\tabcolsep}{6pt}
\scalebox{0.75}{
\makebox[1.0\textwidth]{
\begin{tabular}{ccc}

\def\mar{0.5}

\scalebox{\sclbx}{
\begin{tikzpicture}[yscale=1, baseline=(ZUZU.base)]
\coordinate(ZUZU) at (0,0);
\begin{knot}[
    clip width=6,
    ]
\strand[rounded corners=\rc, line width=\lw] (-\mar,3)--(0,3)--(1,3)--(2,2)--(3,2)--(4,1)--(7,1)--(8,0)--(8+\mar,0);
\strand[rounded corners=\rc, line width=\lw] (-\mar,2)--(0,2)--(1,1)--(2,1)--(3,0)--(7,0)--(8,1)--(8+\mar,1);
\strand[rounded corners=\rc, line width=\lw] (-\mar,1)--(0,1)--(2,3)--(6,3)--(7,2)--(8,2)--(8+\mar,2);
\strand[rounded corners=\rc, line width=\lw] (-\mar,0)--(2,0)--(4,2)--(6,2)--(7,3)--(8+\mar,3);
\flipcrossings{1,6}
\end{knot}
\draw[line width=1pt,dashed,red!50] (0,-\recmar) rectangle (4,3+\recmar);
\draw[line width=1pt,dashed,blue!50] (6,-\recmar) rectangle (8,3+\recmar);
\node[scale=1.4,anchor=south,red](A) at (2,3+\recmar) {$\beta(w)=\sigma_2\sigma_1\sigma_3\sigma_2$};
\node[scale=1.4,anchor=south,blue](A) at (7,3+\recmar) {$\beta(v)^{-1}=\sigma_1^{-1}\sigma_3^{-1}$};
\node[scale=\ndscl,anchor=east] (A) at (-\mar,0) {$4$};
\node[scale=\ndscl,anchor=east] (A) at (-\mar,1) {$3$};
\node[scale=\ndscl,anchor=east] (A) at (-\mar,2) {$2$};
\node[scale=\ndscl,anchor=east] (A) at (-\mar,3) {$1$};

\node[scale=\ndscl,anchor=west] (A) at (8+\mar,0) {$4$};
\node[scale=\ndscl,anchor=west] (A) at (8+\mar,1) {$3$};
\node[scale=\ndscl,anchor=west] (A) at (8+\mar,2) {$2$};
\node[scale=\ndscl,anchor=west] (A) at (8+\mar,3) {$1$};
\end{tikzpicture}
}

&

\def\mar{0.2}

\scalebox{\sclbx}{
\begin{tikzpicture}[yscale=1, baseline=(ZUZU.base)]
\coordinate(ZUZU) at (0,0);
\begin{knot}[
    clip width=6,
    ]
\strand[rounded corners=\rc, line width=\lw] (-\mar,3)--(0,3)--(1,3)--(2,2)--(3,2)--(4,1)--(7,1)--(8,0)--(8+\mar,0);
\strand[rounded corners=\rc, line width=\lw] (-\mar,2)--(0,2)--(1,1)--(2,1)--(3,0)--(7,0)--(8,1)--(8+\mar,1);
\strand[rounded corners=\rc, line width=\lw] (-\mar,1)--(0,1)--(2,3)--(6,3)--(7,2)--(8,2)--(8+\mar,2);
\strand[rounded corners=\rc, line width=\lw] (-\mar,0)--(2,0)--(4,2)--(6,2)--(7,3)--(8+\mar,3);
\flipcrossings{1,6}
\end{knot}
\draw[rounded corners=\rc, line width=\lw] 
(8+\mar,3)--(8+\hmar,3)--(8+\hmar,3+\vmar)--(-\hmar,3+\vmar)--(-\hmar,3)--(-\mar,3);
\draw[rounded corners=\rc, line width=\lw] (8+\mar,2)--(8+\hmarr,2)--(8+\hmarr,3+\vmarr)--(-\hmarr,3+\vmarr)--(-\hmarr,2)--(-\mar,2);
\draw[rounded corners=\rc, line width=\lw] (8+\mar,1)--(8+\hmarrr,1)--(8+\hmarrr,3+\vmarrr)--(-\hmarrr,3+\vmarrr)--(-\hmarrr,1)--(-\mar,1);
\draw[rounded corners=\rc, line width=\lw] (8+\mar,0)--(8+\hmarrrr,0)--(8+\hmarrrr,3+\vmarrrr)--(-\hmarrrr,3+\vmarrrr)--(-\hmarrrr,0)--(-\mar,0);
\end{tikzpicture}
}

&
\def\hbr{5}
\def\lw{3pt}
\def\lww{10pt}
\def\mar{2}
\def\ndscl{2}
\scalebox{0.35}{
\begin{tikzpicture}[baseline=(ZUZU.base),yscale=1,xscale=1]
\coordinate(ZUZU) at (0,0);
\begin{knot}[%
    clip width=5,
    ]
  \strand[line width=\lw,rounded corners=\rc] (-\mar,0)--(2,0)--(7,5)--(7+\mar,5);
  \strand[line width=\lw,rounded corners=\rc] (-\mar,1)--(1,1)--(6,6)--(7+\mar,6);
  \strand[line width=\lw,rounded corners=\rc] (-\mar,2)--(0,2)--(5,7)--(7+\mar,7);
  \strand[line width=\lw,rounded corners=\rc] (-\mar,3)--(0,3)--(3,0)--(7+\mar,0);
  \strand[line width=\lw,rounded corners=\rc] (-\mar,4)--(1,4)--(4,1)--(7+\mar,1);
  \strand[line width=\lw,rounded corners=\rc] (-\mar,5)--(2,5)--(5,2)--(7+\mar,2);
  \strand[line width=\lw,rounded corners=\rc] (-\mar,6)--(3,6)--(6,3)--(7+\mar,3);
  \strand[line width=\lw,rounded corners=\rc] (-\mar,7)--(4,7)--(7,4)--(7+\mar,4);
\flipcrossings{1,2,3,4,5,6,7,8,9,10,11,12,13,14,15}
\foreach[count=\j]\i in {0,1,...,7}{
\node[scale=\ndscl,anchor=east] (A) at (-\mar,7-\i) {$\j$};
\node[scale=\ndscl,anchor=west] (A) at (7+\mar,7-\i) {$\j$};
}
\end{knot}

\end{tikzpicture}
}

\\
braid $\beta_f=\beta(w)\cdot \beta(v)^{-1}$ & closure $\betah_f$ of $\beta_f$ & \quad$\beta_{\fkn}$ for $k=3$, $n=8$

\end{tabular}
}
}

  \caption{\label{fig:braids} Braids and links associated to positroid varieties.}
\end{figure}
For each element $u\in S_n$, let $\beta(u)$ denote the corresponding braid, obtained by choosing a reduced word $u=s_{i_1}s_{i_2}\cdots s_{i_{\ell(u)}}$ for $u$ and then replacing each $s_i$ with $\sigma_i$. 
\begin{definition}\label{dfn:knot}
For $f=\fvw\in\Bkn$, we set 
\begin{equation}\label{eq:beta_dfn}
\beta_f:=\beta(w)\cdot \beta(v)^{-1}.
\end{equation}
 We refer to the closure $\betah_f$ as \emph{the link associated to $f$}. See \cref{fig:braids} for an example.
\end{definition}
The link $\betah_f$ is a \emph{knot} (i.e., has one connected component) if and only if $f\in\Bknc$; see \cref{prop:free}.

\subsection{\FLY polynomial}\label{sec:HOMFLY}
The \emph{\FLY polynomial} $\HOMP(L)=\HOMP(L;a,z)$ of an (oriented) link $L$ is defined by the skein relation
\begin{equation}\label{eq:HOMFLY_dfn}
  a \HOMP(L_+) - a^{-1}\HOMP(L_-)=z \HOMP(L_0)\quad\text{and}\quad \HOMP(\unkn)=1,
\end{equation}
where $\unkn$ denotes the unknot and  $L_+$, $L_-$, $L_0$ are three links whose planar diagrams locally differ as  follows.

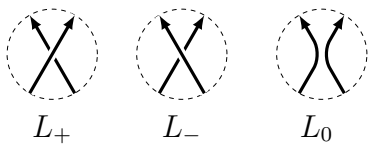
\def\lw{2pt}
\def\lww{6pt}
\def\drcirc{
\draw[line width=0.3pt, dashed] (0,0) circle (1cm);
}
\def\scl{0.6}

\def\DA{30}
\begin{center}
\begin{tabular}{ccc}
\scalebox{\scl}{
\begin{tikzpicture}
\drcirc
\draw[line width=\lw,->,>=latex] (-90+\DA:1)--(90+\DA:1);
\draw[line width=\lww,white] (-90-\DA:1)--(90-\DA:1);
\draw[line width=\lw,->,>=latex] (-90-\DA:1)--(90-\DA:1);
\end{tikzpicture}
}
&

\scalebox{\scl}{
\begin{tikzpicture}
\drcirc
\draw[line width=\lw,->,>=latex] (-90-\DA:1)--(90-\DA:1);
\draw[line width=\lww,white] (-90+\DA:1)--(90+\DA:1);
\draw[line width=\lw,->,>=latex] (-90+\DA:1)--(90+\DA:1);
\end{tikzpicture}
}
&

\def\rc{10}
\scalebox{\scl}{
\begin{tikzpicture}
\drcirc
\draw[line width=\lw,->,>=latex,rounded corners=\rc] (-90-\DA:1)--(0,0)--(90+\DA:1);
\draw[line width=\lw,->,>=latex,rounded corners=\rc] (-90+\DA:1)--(0,0)--(90-\DA:1);
\end{tikzpicture}
}
\\
$L_+$ & $L_-$ & $L_0$

\end{tabular}
\end{center}

\begin{example}\label{ex:S_2_homfly}
For $n=2$, we may take $L_+$ to be the closure of $\sigma_1$, in which case $L_-$ is the closure of $\sigma_1^{-1}$ and $L_0=\unlk$ is 
 the $2$-component unlink. Applying~\eqref{eq:HOMFLY_dfn}, we find $\HOMP(L_0)=\frac{a-a^{-1}}z$.
\end{example}

Surprisingly, the \FLY polynomial computes the number of $\F_q$-points of \emph{any} positroid variety.

\begin{theorem}\label{thm:homfly}
  For all $f\in \Bkn$, let $\Ptop_f$ be obtained from the top $a$-degree term of $\HOMP(\betah_f;a,z)$ by substituting $a:=q^{-\frac12}$ and $z:=q^{\frac12}-q^{-\frac12}$. Then
\begin{equation}\label{eq:homfly_f}
  \#\Pio_f(\F_q)=(q-1)^{n-1}\cdot \Ptop_f.
\end{equation}
\end{theorem}

\begin{remark}\label{rmk:torus_knot}
When $\gcd(k,n)=1$, we have $\fkn\in\Bknc$, and the associated knot $\betah_{\fkn}$ is the \emph{$(k,n-k)$-torus knot}; see \figref{fig:braids}(right). The value of $\HOMP(\betah_{\fkn};a,z)$ was computed in~\cite{Jones}, and its relationship with Catalan numbers was clarified in~\cite{Gorsky_cat}. Thus,~\eqref{eq:point_count} follows from \cref{thm:homfly} as a direct corollary. 
\end{remark}

\begin{example}
For $k=3$, $n=8$, one calculates (for instance, using  Sage\footnote{\url{https://doc.sagemath.org/html/en/reference/knots/sage/knots/link.html}}) that
\[\HOMP(\betah_{\fkn};a,z)=\frac{z^{8} + 8 \, z^{6} + 21 \, z^{4} + 21 \, z^{2} + 7}{a^{8}} - \frac{z^{6} + 7 \, z^{4} + 14 \, z^{2} + 8}{a^{10}} + \frac{z^{2} + 2}{a^{12}}.\] 
Substituting $a:=q^{-\frac12}$ and $z:=q^{\frac12}-q^{-\frac12}$ into $\frac{z^{8} + 8 \, z^{6} + 21 \, z^{4} + 21 \, z^{2} + 7}{a^{8}}$, we get 
\[\Ptop_f=q^{8} + q^{6} + q^{5} + q^{4} + q^{3} + q^{2} + 1=q^4\cdot \Cat_{3,5}(q,1/q).\]
This agrees with~\eqref{eq:Pit_kn_Fq} and~\eqref{eq:homfly_f}.
\end{example}

\subsection{Richardson varieties}\label{sec:richardson-varieties}
Let $G$ be a complex semisimple algebraic group of adjoint type, and choose a pair $B,B_-\subset G$ of opposite Borel subgroups.  Let $T:=B\cap B_-$ be the maximal torus and $W:=N_G(T)/T$ the associated Weyl group.  We have Bruhat decompositions $G=\bigsqcup_{w\in W} BwB=\bigsqcup_{v\in W} B_-vB$, and the intersection $BwB\cap B_-vB$ is nonempty if and only if $v\leq w$ in the Bruhat order on $W$. For $v\leq w$, we denote by $\Rich_v^w:=(BwB\cap B_-vB)/B$ an \emph{open Richardson variety} inside the \emph{(generalized) complete flag variety} $G/B$.  The varieties $\Rich_v^w$ form a stratification of $G/B$.

Now suppose $G = \PGL_n(\C)$.  We have $W=S_n$, the subgroups $B,B_-\subset G$ consist of upper and lower triangular matrices, and $T\cong (\Cast)^{n-1}$ is the group of diagonal matrices modulo scalar matrices.  In this case we denote the generalized flag variety $G/B$ by $\Fl(n)$.  By \cref{prop:v_w}, positroid varieties correspond to pairs ${v\leq w}$ of permutations such that $w$ is $k$-Grassmannian.  The projection map $\Fl(n)\to\Gr(k,n)$ restricts to an isomorphism $ \Rich_v^w\cong \Pio_f$ for every permutation $f=\fvw\in\Bkn$ (see \cref{prop:KLS}). Thus, positroid varieties are special cases of Richardson varieties.  One can similarly associate a braid $\bvw:=\beta(w)\cdot \beta(v)^{-1}$ to any pair $v\leq w$ in $S_n$ and consider its closure $\bhvw$. We refer to links of the form $\bhvw$ as \emph{Richardson links}. 

The point count $\#\Rich_v^w(\F_q)$ is given by the \emph{Kazhdan--Lusztig $R$-polynomial} \cite{KL1,KL2}, and both the statement and the proof of \cref{thm:homfly} generalize to this setting (for $G$ of arbitrary type); see \cref{thm:homfly2,thm:traces}.

\subsection{Main result: ordinary cohomology}\label{sec:main-result-ordinary}
Our results for the positroid variety $\Pio_{k,n}$ are special cases of a statement which applies to open Richardson varieties in arbitrary Dynkin type. This includes all positroid varieties $\Pio_f$ for $f\in\Bkn$, where the number of cycles $\ncyc(\fmod)$ can be arbitrary.  We start with the non-equivariant version of our result.

Let $\h:=\Lie(T)$ be the Cartan subalgebra of $\Lie(G)$ corresponding to $T$, and denote $R := \Symh = {\rm Sym}_\C \h^\ast$. The ring $R$ is graded so that the elements of $\h^\ast\subset R$ have polynomial degree $2$. For $G=\PGL_n(\C)$, $R=\C[y_1,\dots,y_{n-1}]$ is the polynomial ring. Since $W$ is a Coxeter group, we can consider the category $\SBim$ of \emph{Soergel bimodules}~\cite{SoeHC,EMTW}. Each object $B\in\SBim$ is a graded $R$-bimodule, and we will be interested in its $R$-invariants, which by definition form the \emph{zeroth Hochschild cohomology of $B$}:
\begin{equation}\label{eq:HH0}
\HH^0(B):=\{b\in B\mid rb=br\text{ for all $r\in R$}\}.
\end{equation}
Thus, $\HH^0(B)$ is a graded $R$-module. Denote
\begin{equation}\label{eq:HHC}
  \HHC^0(B):=\HH^0(B)\otimes_R\C,
\end{equation}
where $\C=R/(\h^\ast)$ is the $R$-module on which $\h^\ast$ acts by $0$. While the functor $\HH^0$ involves Soergel bimodules, the functor $\HHC^0$ involves \emph{Soergel modules} instead; see \cref{cor:SMod}.

 To any element $u\in W$, Rouquier~\cite{Rou} associates two cochain complexes $\FR(u)$ and $\FR(u)^{-1}$ of Soergel bimodules. For a braid $\bvw=\beta(w)\cdot \beta(v)^{-1}$, we set $\Fvw:=\FR(w)\otimes_R \FR(v)^{-1}$. Applying the functor $\HHC^0$ to each term of this complex yields a complex $\HHC^0(\Fvw)$ of graded $R$-modules. Taking its cohomology 
\begin{equation}\label{eq:HHH0}
\HHHC^0(\Fvw):=H^\bul(\HHC^0(\Fvw)),
\end{equation}
we get a bigraded vector space. Explicitly, we have  
\begin{equation*}
  \HHHC^0(\Fvw)=\bigoplus_{k,p\in\Z} \HHXC kp_{v,w},
\end{equation*}
where $\HHXC kp_{v,w}$ is the polynomial degree $2p$ part of $H^{k}(\HHC^0(\Fvw))$. 

\begin{table}
  \setlength{\tabcolsep}{1.5pt}
\renewcommand{\arraystretch}{1.2}
\makebox[1.0\textwidth]{
\scalebox{0.68}{
\begin{tabular}{|c|c|c|c|c|c|c|c|c|c|c|}\hline
Title & $n$ & $v$ & $w$ & $\lvw$ & $\bhvw$ & $\Rich_v^w$ & $H^{k,(p,p)}(\Rich_v^w)$ & $H^{k,(p,p)}_{T,c}(\Rich_v^w)$ & $\HHXC kp_{v,w}$ & $\HHX kp_{v,w}$\\\hline\hline
Unknot-I & $1$ & $\id$ & $\id$ & $0$ & $\unkn$ & $\pt$ & $k=0,\,p=0$ & $k=0,\,p=0$ & $k=0,\,p=0$ & $k=0,\,p=0$\\\hline
Unknot-II & $2$ & $\id$ & $s_1$ & $1$ & $\unkn$ & $\Pio_{1,2}$ & \begin{tabular}{c} $k=0,\,p=0$\\ $k=1,\,p=1$\end{tabular} & $k=1,\,p=0$ & \begin{tabular}{c}$k=0,\,p=0$\\ $k=-1,\,p=1$\end{tabular} & $k=0,\,p=0$\\\hline
$2$-cpt. unlink & $2$ & $s_1$ & $s_1$ & $0$ & $\unlk$ & $\pt$ & $k=0,\,p=0$ & $k=2p,\,p\in\Z_{\geq0}$ & $k=0,\,p=0$ & $k=0,\,p\in\Z_{\geq0}$\\\hline
Hopf link
& $4$ & $\id$ & $f_{2,4}$ & $4$ & $\hopflink$ & $\Pio_{2,4}$ & \cref{tab:MHT_vs_HHHC} & \begin{tabular}{c} $k=4,\,p=0$\\ $k=2+2p,\,p\in\Z_{\geq2}$\end{tabular}  & \cref{tab:MHT_vs_HHHC} & \begin{tabular}{c} $k=0,\,p=0$\\ $k=-2,\,p\in\Z_{\geq2}$\end{tabular}\\\hline
Trefoil knot
 & $5$ & $\id$ & $f_{2,5}$ & $6$ & $\trefoil$ & $\Pio_{2,5}$ & \cref{tab:MHT_vs_HHHC} & \begin{tabular}{c} $k=6,\,p=0$\\ $k=8,\,p=2$\end{tabular}  & \cref{tab:MHT_vs_HHHC} & \begin{tabular}{c} $k=0,\,p=0$\\ $k=-2,\,p=2$\end{tabular}\\\hline
\end{tabular}
}
}
  \caption{\label{tab:examples} Summary of examples computed in \cref{sec:Soergel_examples,sec:Htc_examples}, illustrating \cref{thm:ordinary,thm:main}. The last four columns list all values of $k,p$, for which the corresponding bigraded component is nonzero (in each case, it is $1$-dimensional). We have $f_{2,4}=s_2s_1s_3s_2$ and $f_{2,5}=s_3s_2s_1s_4s_3s_2$.}
\end{table}

Recall from~\eqref{eq:DS} that we have a bigrading on $H^\bul(\Rich_v^w)$ coming from the Deligne splitting. 
\begin{theorem}\label{thm:ordinary}
 For all $v\leq w\in W$ and $k,p\in\Z$, we have
\begin{equation}\label{eq:ordinary}
 H^{k,(p,p)}(\Rich_v^w) \cong \HHXC {-k}p_{v,w}.
\end{equation}
\end{theorem}
\noindent See \cref{tab:examples,tab:MHT_vs_HHHC} for examples. %

\subsection{Main result: equivariant cohomology}\label{sec:main-result-equivariant}
The spaces $\HHH^0(\Fvw)$ and $\HHHC^0(\Fvw)$ are closely related. By \cref{thm:ordinary}, $\HHHC^0(\Fvw)$ yields the cohomology of $\Rich_v^w$. It turns out that $\HHH^0(\Fvw)$ yields the \emph{torus-equivariant} cohomology of $\Rich_v^w$.

The algebraic torus $T$ acts on each Richardson variety $\Rich_v^w$, and thus we can consider its \emph{$T$-equivariant cohomology with compact support}, denoted $H^\bul_{T,c}(\Rich_v^w)$. It is equipped with an action of the ring $H^\bul_{T}(\pt)\cong R$. Similarly to~\eqref{eq:DS}, $H^\bul_{T,c}(\Rich_v^w)$ admits a second grading via the mixed Hodge structure and is therefore a bigraded $R$-module. 

\begin{theorem}\label{thm:main}
 For all $v\leq w\in W$, we have an isomorphism of bigraded $R$-modules
\begin{equation}\label{eq:main}
H^\bul_{T,c}(\Rich_v^w)\cong \HHH^0(\Fvw).
\end{equation}
For each $k,p\in\Z$, it restricts to a vector space isomorphism 
\begin{equation}\label{eq:grading:intro}
  H^{\lvw+2p+k,(p,p)}_{T,c}(\Rich_v^w)\cong \HHX {k}p_{v,w},
\end{equation}
where $\lvw=\ell(w)-\ell(v)=\dim\Rich_v^w$.
\end{theorem}
\noindent See \cref{tab:examples} for examples. We explain how \cref{thm:master_kn,thm:ordinary} follow from \cref{thm:main} in \cref{ssec:introproof,sec:ordinary-cohomology}, respectively.

Observe that the grading conventions in~\eqref{eq:ordinary} and~\eqref{eq:grading:intro} are quite different. In fact, the two statements are related by an application of the $q,t$-symmetry~\eqref{eq:Koszul}, as we now explain.

\subsection{Koszul duality and $q,t$-symmetry}\label{ssec:Koszul}
For any $f\in\Bknc$, the positroid variety $\Pit_f$ is a cluster variety~\cite{GL}, so the polynomial $\Poincf q,t)$ is $q,t$-symmetric by \cref{cor:master:qt_symm}. Richardson varieties are not yet known to admit cluster structures (see~\cite{Lec}), and in particular, the curious Lefschetz theorem of \cite{LS} cannot yet be applied to conclude that $\Poinc(\Rich_v^w;q,t)$ is $q,t$-symmetric for arbitrary $v\leq w\in S_n$. In \cref{sec:Koszul}, we show that the $q,t$-symmetry phenomenon for positroid and Richardson varieties is a manifestation of  \emph{Koszul duality} for the derived category of Schubert-constructible sheaves on the flag variety~\cite{BGS,BY,AR2}.

\begin{theorem}\label{thm:Koszul_intro:cohom}\label{thm:Koszul}
  For all $v\leq w\in W$ and $k,p\in\Z$, we have an isomorphism
\begin{equation}\label{eq:Koszul}
  H^{k,(p,p)}(\Rich_v^w,\C) \cong H^{\lvw+k-2p,(\lvw-p,\lvw-p)}(\Rich_{v}^{w},\C)
\end{equation}
of vector spaces. In other words, the polynomial $\Poinc(\Rich_v^w;q,t)$ is $q,t$-symmetric.
\end{theorem}
\noindent This gives a new proof of the $q,t$-symmetry of $\Cat_{k,n-k}(q,t)$ for $\gcd(k,n)=1$.

We now explain the connection to link invariants. Given a Richardson link $\bhvw$, one can consider the bigraded vector spaces $\HHH^0(\FR(\bvw))$ and $\HHHC^0(\FR(\bvw))$, and their suitably renormalized bigraded Hilbert series, denoted $\PK(\bvw;q,t)$ and $\PKRC(\bvw;q,t)$, respectively; see~\eqref{eq:PKRtop}--\eqref{eq:PKRC}.

The polynomial $\PK(\bvw;q,t)$ is the top $a$-degree coefficient (see the footnote in the introduction) of the celebrated \emph{Khovanov--Rozansky homology}~\cite{KR1,KR2,KhoSoe} of $\bhvw$, which is a link invariant, i.e., depends only on the closure $\bhvw$ of $\bvw$; see \cref{sec:link-invariant}. 

Our results imply (see \cref{sec:ordinary-cohomology}) that when $\bhvw$ is a knot, we have
\begin{align}
\label{eq:PKRtop=PKRC}
  \PKRtop(\bhvw;q,t)&=\PKRC(\bvw;q,t) \quad\text{and}\\ 
\Poinc(\Rich_v^w/T; q,t)&=   \left(\qq\tt\right)^{\chi(\bvw)} \PK(\bhvw;q,t),
\end{align}
where $\chi(\bvw):=\frac{\lvw-n+\ncyc(f)}2$; see~\eqref{eq:chi_dfn}. More generally, \cref{thm:ordinary} implies that for any $v\leq w\in S_n$, one can relate $\PKRC(\bvw;q,t)$ to $\Poinc(\Rich_v^w;q,t)$ by a simple transformation. 
As we show in \cref{cor:PKRtop=PKRC}, \eqref{eq:PKRtop=PKRC} holds more generally for arbitrary knots $\betah$. 

For a general link $\betah$, the question of whether $\PK(\betah;q,t)$ is $q,t$-symmetric has been a major open problem\footnote{At the final stages of the preparation of this manuscript, we learnt that the $q,t$-symmetry of $\PK(\betah;q,t)$ for the case when $\betah$ is a knot has been established in a very recent preprint~\cite{ObRoz}.  (Note added in 2023: see also \cite{GHM}.)} going back to~\cite{DGR}. For Richardson links, we show that it also amounts to applying Koszul duality.

\begin{corollary}\label{cor:Koszul_intro}
  For any $v\leq w\in S_n$, we have 
\begin{equation*}
  \PKRC(\bvw;q,t) =\PKRC(\bvw;t,q).
\end{equation*}
Consequently, by~\eqref{eq:PKRtop=PKRC}, if $\bhvw$ is a knot then 
\begin{equation*}
  \PK(\bhvw;q,t) =\PK(\bhvw;t,q).
\end{equation*}
\end{corollary}

\subsection{Extensions of Verma modules}\label{sec:Verma_intro}
The above theorems have represen\-tation-theoretic applications to \emph{Verma modules}, which are certain infinite-dimensional modules over the Lie algebra $\g$ of $G$. Consider the principal block $\O_0$ of the Bernstein–Gelfand–Gelfand category~$\O$, and let $M_w$ be the Verma module with highest weight $w(\rho)-\rho$, where $\rho$ is half the sum of positive roots of the root system of $\g$. We also denote by $L_w$ the corresponding simple module. The graded dimensions of $\Ext^\bul(M_v,L_w)$ famously coincide with the coefficients of the \emph{Kazhdan--Lusztig $P$-polynomials} $P_{v,w}(q)$; see e.g.~\cite[Theorem~3.11.4]{BGS}. On the other hand, computing extension groups $\Ext^\bul(M_v,M_w)$ is an important open problem; see e.g.~\cite{Mazorchuk1,DhMa}.

A \emph{graded} version of $\O_0$ was introduced by Beilinson--Ginzburg--Soergel \cite{BGS}. They constructed the (essentially unique) graded lifts of Verma modules $M_w$ (see also~\cite{Stroppel}), thus endowing the space $\Ext^\bul(M_v,M_w)$ with a second grading:
\begin{equation*}
  \Ext^\bul(M_v,M_w)=\bigoplus_{k,\wt\in\Z} \Ext^{k,(\wt/2)}(M_v,M_w).
\end{equation*}
These $\Ext$-groups can be related to the cohomology of open Richardson varieties using the \emph{localization theorem} of~\cite{BB,BK}.
In the case of Kazhdan--Lusztig polynomials, the groups $\Ext^\bul(M_v,L_w)$ are ``pure": the two gradings agree.  On the other hand, the bigrading on $\Ext^\bul(M_v,M_w)$ turns out to be quite non-trivial. As a corollary to our approach, we obtain the following result.\footnote{We remark that Soergel's original work \cite{SoeGarben} directly relates $\Ext$-groups in category~$\O$ and $\Hom$-groups in $\SBim$; see also~\eqref{eq:HHH=Hom}.}
\begin{theorem}\label{thm:Verma}
  For all $v\leq w\in W$ and $k,\wt\in\Z$, we have 
\begin{equation*}
  \dim \Ext^{k,((\wt-\lvw)/2)}(M_v,M_w) = \dim  \HHXC {k-r}{\wt/2}_{v,w}.
\end{equation*}
(In particular, both sides are zero for odd $\wt$.)
\end{theorem}

Thus, while $\Ext^\bul(M_v,L_w)$ gives the Kazhdan--Lusztig polynomials, $\Ext^\bul(M_v,M_w)$ gives the rational $q,t$-Catalan numbers for $v=\id$ and $w=\fkn$.

\subsection{Notes}\label{sec:notes}

We collect the historical background and several remarks on the above results.

\parag{Symmetry and unimodality}\label{sec:notes:catalan}
The symmetry and unimodality of the Gaussian polynomial $\qbin{n}k$ are consequences of the hard Lefschetz theorem for $H^\bul(\Gr(k,n))$.  Whereas symmetry is apparent from the combinatorial definition of $\qbin{n}k$ (\cite[Proposition~1.7.3]{EC1}), unimodality is notoriously difficult to see combinatorially.  Unimodality was first proven by Sylvester~\cite{Sylvester}, the relation to hard Lefschetz observed by Stanley~\cite{Stanley_HL}, and a combinatorial proof given by O'Hara~\cite{OHara}.
When $a=n$ and $b=n+1$, $\Cab(q,t)$ recovers the famous \emph{$q,t$-Catalan numbers} $\Ca(q,t)$ of Garsia and Haiman~\cite{GaHa} studied extensively in e.g.~\cite{GaHag,HagBook,HagConj,GoMa1,GoMa2}. The fact that $\Ca(q,t)$ is $q,t$-symmetric and $q,t$-unimodal follows from the results of Haiman~\cite{Haiman94,Haiman02}.  For arbitrary $a,b$, an explanation for the $q,t$-symmetry property was given by the \emph{Rational Shuffle Conjecture} of~\cite{GN}, proved recently in~\cite{Mellit_rat,CaMe}. The specialization $q^{\frac12\dkn} \Cat_{k,n-k}(q,1/q)=\frac{1}{\qint{n}}\qbin{n}k$ in~\eqref{eq:Pit_kn_Fq} is also a consequence of the Rational Shuffle Conjecture. To our knowledge, the $q,t$-unimodality of $\Cat_{k,n-k}(q,t)$ in \cref{cor:master:qt_unim} is a new result.  See also~\cite[Section~2.2]{Stucky}, which includes a specialization of our unimodality result.%

\parag{Compactified Jacobians and the $P=W$ conjecture}\label{sec:STZ}
We explain the original motivation coming from the results of~\cite{STZ,STWZ} that led to the statement of \cref{thm:master_kn}. The compactified Jacobian $J_{a,b}$ of the plane curve singularity $x^a = y^b$ (with ${\gcd(a,b) = 1}$) is a compact, singular variety with a long history of connections to Catalan theory.  Beauville~\cite{Beauville} showed that the Euler characteristic of $J_{a,b}$ is the rational Catalan number $C_{a,b}$ and Piontkowski~\cite{Piontkowski} (see also Lusztig--Smelt~\cite{LuSm}) showed that the \Poincare polynomial and point count are given by the $q$-analog $\sum_{P\in\Dyck_{k,n-k}} q^{\dinv(P)}$.  Gorsky and Mazin~\cite{GoMa1,GoMa2} first suggested the relation between $J_{a,b}$ and $q,t$-Catalan numbers and since then there has been an explosion of developments relating compactified Jacobians and knot invariants; see e.g.~\cite{GORS,CheDan,GN}.  Our work provides evidence for the following conjecture, arising from the works \cite{STZ,STWZ}; see~\cite{dCHM} for the original $P=W$ conjecture.
\begin{conjecture}\label{conj:Jac}
There is a deformation retraction from the torus quotient $\Pit_{k,n}$ to the compactified Jacobian $J_{k,n-k}$ sending the weight filtration of $H^\bul(\Pit_{k,n})$ to the perverse filtration of $H^\bul(J_{k,n-k})$ \cite{MS,MY}.
\end{conjecture}
Conjecture \ref{conj:Jac} is motivated by the isomorphism, discovered in \cite{STWZ}, between open positroid varieties and moduli spaces of constructible sheaves associated to Legendrian knots \cite{STZ}.  We thank Vivek Shende for explaining a conjectural \emph{wild} non-abelian Hodge correspondence in that setting.  

More generally, when a Richardson link is algebraic (that is, arising as the link of a singularity), one may expect a statement similar to Conjecture~\ref{conj:Jac} for the compactified Jacobian of the singularity. See \cref{rmk:alg} for related discussion.

\begin{remark}\label{rmk:STZ_point_count}
After discovering the proof of \eqref{eq:point_count} via the \FLY polynomial, we found that it can also be deduced from the results of~\cite{STZ,STWZ}. Our proof is new and yields a generalization (\cref{thm:homfly}) of \eqref{eq:point_count} to arbitrary open positroid varieties, and more generally to open Richardson varieties in generalized flag varieties. 
\end{remark}

\parag{Plabic graph links}
In \cref{sec:knots-assoc-positr}, we associated a link $\betah_f$ to each positroid variety $\Pio_f$. Two other (more complicated) ways of assigning a Legendrian/transverse link to a positroid variety have appeared recently in~\cite{STWZ,FPST}, stated in the language of Postnikov's \emph{plabic graphs}~\cite{Pos}. Conjecturally, the links of~\cite{STWZ,FPST} are isotopic to our links $\betah_f$. We hope to return to this question in future work~\cite{GL_plabic_links}; see also~\cite{CGGS_positroid}.

\parag{Geometric interpretations and other Dynkin types}
A geometric interpretation of the full triply-graded KR homology was given by Webster--Williamson~\cite{WW17}. Our approach yields a different geometric interpretation of the (doubly-graded) top $a$-degree part of KR homology.  Our geometric interpretation in addition holds for Dynkin types outside type~A. The analog of the \FLY polynomial in other Dynkin types (as a trace on the Hecke algebra; cf. \cref{sec:FLY_arb_type}) was introduced in~\cite{Gomi}; see also~\cite{Rou_2braid}.  For related discussion of knot invariants in other types, see for example~\cite{WW08,WW11};  see also~\cite{BT,CGGS,ChPh2,ChPh1,Mellit_Cell} for related results. 

\parag{Odd cohomology vanishing}\label{sec:notes:odd_cohom_van}
It follows from the results of~\cite{Mellit_torus,MeHog} that Khovanov--Rozansky homology of any (positive) torus knot or link is concentrated in even degrees. It is tempting to conjecture that the same property holds for all Richardson knots or links. However, this is not the case: see Examples~\ref{ex:odd_cohom1}, \ref{ex:odd_cohom2}, and~\ref{ex:odd_cohom3}. See \cref{rmk:alg} for a discussion of the closely related class of \emph{algebraic knots}.

\parag{Complements of hyperplane arrangements}\label{sec:hyperplane}
The top positroid variety $\Piokn$ may be considered ``the complement of a hyperplane arrangement in the Grassmannian'': it is obtained from $\Gr(k,n)$ by removing $n$ hypersurfaces, each given by a linear equation in the Pl\"ucker coordinates on $\Gr(k,n)$.  More general ``Grassmannian hyperplane arrangements" appear naturally in the study of amplituhedra and Grassmann polytopes \cite{LamCDM,GL_parity}.

The cohomology of complements of hyperplane arrangements in projective space is very well studied: both the \Poincare polynomial and the point count are simple specializations of the characteristic polynomial.  The coincidence is a manifestation of the purity of the mixed Hodge structure~\cite{Shapiro}, a property that also holds for the Grassmannian $\Gr(k,n)$.

\parag{Recurrence relations}
Our results associate a $q,t$-polynomial to each positroid variety. One possible advantage of this approach is a recurrence for these polynomials, arising from the recurrence for positroid varieties developed by Muller--Speyer~\cite{MuSp}. For instance, their results allow to compute the point counts recursively; cf.~\cite{GL_plabic_links}. To compute the \Poincare or the mixed Hodge polynomials, the recurrence of~\cite{MuSp} yields a long exact sequence for the cohomology. It seems plausible that in favorable cases (e.g., when the odd cohomology vanishes), this sequence may be used to calculate the mixed Hodge polynomials of special families of links as was done in~\cite{Mellit_torus,EH,MeHog}. We remark that the latter recurrences pass through complexes of Soergel bimodules which do not come from any braids; an interesting open problem is to understand the positroid/Richardson interpretation of such complexes.
%
%
%

\subsubsection*{Structure of the paper}
In \cref{sec:point-count-fly}, we study the relationship between the point count and the \FLY polynomial, and prove \cref{thm:homfly} and its generalization (\cref{thm:traces}) to open Richardson varieties. In \cref{sec:KR_hom,sec:cohom-pos-rich}, we discuss background on KR homology and cohomology of positroid varieties, respectively. We deduce Theorem~\ref{thm:master_kn} from Theorem~\ref{thm:main} in \cref{ssec:introproof}. 
 In \cref{sec:sheaf-cohomology}, we recast our results in the language of equivariant derived categories, and split the main result (\cref{thm:main}) into two statements, \cref{prop:RSW,prop:degr}. These statements are proved in \cref{sec:proof-RSW,sec:sheafproof}, respectively, thereby completing the proof of \cref{thm:main}.
In \cref{sec:ord_cohom}, we deduce the rest of our results (\cref{thm:ordinary,thm:Koszul,thm:Verma}) from \cref{thm:main}. Finally, in \cref{sec:Deogram}, we study analogs of Catalan numbers associated to arbitrary positroid varieties.

\subsubsection*{Acknowledgments}
We are grateful to David Speyer for his contributions to this project at the early stages. We also thank Pramod Achar, Sergey Fomin, Alexei Oblomkov, Vic Reiner, and Vivek Shende for stimulating discussions. In addition, we are grateful to Ivan Cherednik, Eugene Gorsky, Wolfgang Soergel, and Minh-Tam Trinh for their comments on the first version of this manuscript. Finally, we thank the anonymous referees for their extremely careful reading of our manuscript and for suggesting numerous improvements.

  \addtocontents{toc}{\protect\setcounter{tocdepth}{1}}

\section{Point count and the \FLY polynomial}\label{sec:point-count-fly}
\subsection{Type $A$}
Let $W=S_n$ and $G=\PGL_n(\C)$. Recall from~\cite[Lemmas~A3 and~A4]{KL1} that the number of $\F_q$-points of a Richardson variety $\Rich_v^w$ is given by the \emph{Kazhdan--Lusztig $R$-polynomial} $\KLR_v^w$. When $v\not\leq w$, we have $\KLR_v^w=0$ and $\Rich_v^w=\emptyset$, and for $v=w$, we have $\KLR_v^w=1$ and $\Rich_v^w=\pt$. For $v\leq w\in W$, $\KLR_v^w$ can then be computed by a recurrence relation~\cite[Section~2]{KL1}:
\begin{equation}\label{eq:KLR_dfn}
\KLR_v^w=
  \begin{cases}
    \KLR_{sv}^{sw}, &\text{if $sv<v$ and $sw<w$,}\\
    (q-1)\KLR_{sv}^w+q\KLR_{sv}^{sw}, &\text{if $sv>v$ and $sw<w$.}\\
  \end{cases}\\
\end{equation}
Here, $s=s_i$ for some $1\leq i\leq n-1$ is a simple transposition satisfying $sw<w$. 

Recall from \cref{sec:richardson-varieties} that we associate an $n$-strand braid $\bvw:=\beta(w)\cdot \beta(v)^{-1}$ to any pair $v,w\in S_n$ of permutations. For a Laurent polynomial $P=P(a,z)$, we denote by $\mda(P)\in\Z$ the maximal degree of $a$ in $P$, and for $\da\in \Z$, we let $[a^{\da}]P\in \C[z^{\pm1}]$ be the coefficient of $a^{\da}$ in $P$.   For $v,w\in S_n$, recall that we set $\lvw:=\ell(w)-\ell(v)$. Denote
\begin{equation}\label{eq:davw}
  \davw:=n-1-\lvw\quad\text{and}\quad \Pvw=\Pvw(a,z):=\HOMP(\bhvw;a,z),
\end{equation}
where $\HOMP(\bhvw;a,z)$ is the \FLY polynomial defined in \cref{sec:HOMFLY}.  The goal of this section is to show the following strengthening of \cref{thm:homfly}.
\begin{theorem}\label{thm:homfly2}  Let $v,w\in S_n$.
\begin{theoremlist}
\item\label{thm:homfly2:1} If $v\not\leq w$ then $\mda(\Pvw)<\davw$.
\item\label{thm:homfly2:2} If $v\leq w$ then $\mda(\Pvw)=\davw$.
\item\label{thm:homfly2:3} For any $v,w\in S_n$, let $\Ptopvw$ be obtained from $a^\davw\cdot ([a^\davw]\Pvw)$ by substituting $a:=q^{-\frac12}$ and $z:=q^{\frac12}-q^{-\frac12}$. Then 
\begin{equation}\label{eq:homfly_vw}
\KLR_v^w=(q-1)^{n-1}\cdot \Ptopvw.
\end{equation}
\end{theoremlist}
\end{theorem}

\begin{proof}

We start by recalling the following result, which states that the \emph{(lower) Morton--Franks--Williams inequality}~\cite{MFW2,MFW1} is not sharp for negative braids. It may be alternatively deduced from~\eqref{eq:trace(T_u*T_v)} below.
\begin{lemma}[{\cite[Proposition~2.1]{GMM}}]
Let $v\in S_n$ be a non-identity permutation and let $\beta:=\beta(v)^{-1}$ be the associated negative braid. Then\footnote{Our conventions for $\HOMP(\betah;a,z)$ differ from those of~\cite{GMM} by changing $a\mapsto a^{-1}$.}
\begin{equation}\label{eq:homfly_neg_ineq}
\pushQED{\qed} 
\mda(\HOMP(\betah;a,z))<n-1+\ell(v).
\qedhere
\popQED
\end{equation}
\end{lemma}

\def\itr#1{\itemref{thm:homfly2:#1}}

We now prove all parts of \cref{thm:homfly2} by induction on $\ell(w)$. Consider the base case $\ell(w)=0$. Then~\itr{1} is the content of~\eqref{eq:homfly_neg_ineq}. For~\itr{2}, we observe that $\ell(w)=0$ implies $v=w=\id$,
and iterating \cref{ex:S_2_homfly}, we get $\Pvw=\left(\frac{a-a^{-1}}z\right)^{n-1}$. Thus, $\mda(\Pvw)=n-1=\davw$ for $v=w=\id$. For~\itr{3}, if $v\not\leq w$ then by~\itr1, we get $[a^\davw]\Pvw=0$, so $\Ptopvw=0$, in agreement with $\KLR_v^w=0$. If $v\leq w$ then $v=w=\id$, $a^\davw\cdot [a^\davw]\Pvw=(a/z)^{n-1}$, so $\Ptopvw=(q-1)^{-(n-1)}$, in agreement with~\eqref{eq:homfly_vw}. We have shown the base case for each part.

For the induction step, suppose that $\ell(w)>0$. Choose some $1\leq i\leq n-1$ such that $s_iw<w$ and let $s:=s_i$ and $\sigma:=\sigma_i$. If $sv<v$ then the links $\bhvw$ and $\betah_{sv,sw}$ are isotopic since $\beta_{v,w}=\sigma\beta_{sv,sw}\sigma^{-1}$, and thus $\Pvw=\HOMP_{sv,sw}$. We also have $\davw=\da_{sv,sw}$, and thus $\Ptopvw=\Ptop_{sv,sw}$. By~\eqref{eq:KLR_dfn}, $\KLR_v^w=\KLR_{sv}^{sw}$. So in the case $sv<v$, the induction step holds trivially for each of the three parts. 

Assume now that we have $sw<w$ and $sv>v$. In this case, we have $\beta_{v,w}=\sigma \beta(sw)\beta(sv)^{-1}\sigma\sim \beta(sw)\beta(sv)^{-1}\sigma^2$, $\beta_{v,sw}=\beta(sw)\beta(sv)^{-1}\sigma$, and $\beta_{sv,sw}=\beta(sw)\beta(sv)^{-1}$, where $\sim$ relates conjugate braids. Applying~\eqref{eq:HOMFLY_dfn} with
\begin{equation*}
L_+:=\bhvw,\quad L_0:=\betah_{v,sw},\quad L_-:=\betah_{sv,sw},
\end{equation*}
we get $  a\Pvw-a^{-1}\HOMP_{sv,sw}=z\HOMP_{v,sw}$, and thus 
\begin{equation}\label{eq:Pvw_ind_step}
\Pvw=\frac za \HOMP_{v,sw}+a^{-2}\HOMP_{sv,sw}.
\end{equation}
Note that $\da_{v,sw}=\davw+1$ and $\da_{sv,sw}=\davw+2$. Let us show~\itr1. We have $v\not\leq w$, $sw<w$, and $sv>v$, and thus clearly $v\not\leq sw$ and $sv\not\leq sw$. By the induction hypothesis, we have $\mda(\HOMP_{v,sw})<\da_{v,sw}$ and $\mda(\HOMP_{sv,sw})<\da_{sv,sw}$. By~\eqref{eq:Pvw_ind_step}, we get $\mda(\Pvw)<\davw$, finishing the proof of~\itr1. In particular, we have shown that~\eqref{eq:homfly_vw} holds for all $v\not\leq w$.

Now assume $v\leq w$. We show~\itr2 and~\itr3 simultaneously. By the induction hypothesis, we have $\mda(\HOMP_{v,sw})\leq \da_{v,sw}$, $\mda(\HOMP_{sv,sw})\leq \da_{sv,sw}$ (whether the equality holds depends on whether $v\leq sw$ and $sv\leq sw$). Thus, by~\eqref{eq:Pvw_ind_step}, $\mda(\Pvw)\leq \davw$. The links $L_0=\betah_{v,sw}$ and $\betah_{sv,w}$ are isotopic since $\beta_{sv,w}=\sigma\beta(sw)\beta(sv)^{-1}\sim \beta(sw)\beta(sv)^{-1}\sigma=\beta_{v,sw}$, so $\HOMP_{v,sw}=\HOMP_{sv,w}$. Applying the map $P\mapsto a^\davw\cdot ([a^\davw] P)$ to both sides of~\eqref{eq:Pvw_ind_step} and substituting $a:=q^{-\frac12}$ and $z:=q^{\frac12}-q^{-\frac12}$, we get 
\begin{equation*}
\Ptop_{v,w}=(q-1) \Ptop_{sv,w}+q\Ptop_{sv,sw}.
\end{equation*}
Combining this with the induction hypothesis and~\eqref{eq:KLR_dfn}, we get $\KLR_v^w=(q-1)^{n-1}\cdot \Ptopvw$. In particular, the coefficient of $a^\davw$ in $\Pvw$ is nonzero, so $\mda(\Pvw)=\davw$. Thus, we have completed the induction step for both~\itr2 and~\itr3.
\end{proof}

\subsection{Arbitrary type}\label{sec:FLY_arb_type}
The above connection between point counts and the \FLY polynomial can be generalized to arbitrary Weyl groups as follows. Let $\Hecke$ be the \emph{Hecke algebra} of $W$: it is generated over $\C[q^{\pm1}]$ by the elements $\{T_s\}_{s\in S}$ satisfying the braid relations as well as the Hecke relation 
\begin{equation}\label{eq:T_i^2}
  (T_s+q)(T_s-1)=0 \quad\text{for $s\in S$.}
\end{equation}
The algebra $\Hecke$ admits a linear basis $\{T_w\}_{w\in W}$ indexed by the elements of $W$: we set $T_w:=T_{s_1}T_{s_2}\cdots T_{s_{\ell(w)}}$ for any reduced word $w=s_{1}s_{2}\cdots s_{\ell(w)}$. The \emph{standard trace} $\Htrace:\Hecke\to \C[q^{\pm1}]$ is the $\C[q^{\pm1}]$-linear map defined by
\begin{equation}\label{eq:Htrace_dfn}
  \Htrace(T_w):=
  \begin{cases}
    1, &\text{if $w=\id$;}\\
    0,&\text{otherwise.}
  \end{cases}
\end{equation}
\begin{theorem}\label{thm:traces}
  For any $v,w\in W$, we have 
\begin{equation}\label{eq:traces}
  \KLR_v^w=q^{\lvw} \Htrace(T_w^{-1}T_v).
\end{equation}
\end{theorem}
\noindent For $W=S_n$, in view of the well-known relation between traces and the \FLY polynomial (going back to~\cite{Jones}), \cref{thm:traces} specializes to \cref{thm:homfly2}.

\begin{proof}
First, we state a simple consequence of~\eqref{eq:T_i^2}: for any $v\in W$ and $s\in S$, we have
\begin{equation}\label{eq:T_s*T_v}
  T_sT_v=
  \begin{cases}
    T_{sv}, &\text{if $sv>v$,}\\
    (1-q)T_{v}+qT_{sv}&\text{if $sv<v$.}
  \end{cases}
\end{equation}
Next, we claim that for any $u,v\in W$, we have
\begin{equation}\label{eq:trace(T_u*T_v)}
  \Htrace(T_uT_v)=
  \begin{cases}
    q^{\ell(v)}, &\text{if $u=v^{-1}$,}\\
    0,&\text{otherwise.}
  \end{cases}
\end{equation}
We prove~\eqref{eq:trace(T_u*T_v)} by induction on $\ell(u)$. The base case $\ell(u)=0$ is clear. Otherwise, choose $s\in S$ such that $u=xs$ with $x<xs$. If $v<sv$ then $T_uT_v=T_xT_{sv}$ and we are done by induction. Thus, assume  $v>sv$. By~\eqref{eq:T_s*T_v}, we get
\begin{equation}\label{eq:Htrace_ind}
  \Htrace(T_uT_v)=(1-q)\Htrace(T_xT_v)+q\Htrace(T_xT_{sv}).
\end{equation}
We have $u=v^{-1}$ if and only if $x=(sv)^{-1}$, in which case by induction we find $\Htrace(T_uT_v)=q^{\ell(v)}$. If $u\neq v^{-1}$ then the right-hand side of~\eqref{eq:Htrace_ind} is zero unless $x=v^{-1}$. But $x=v^{-1}$ contradicts our assumptions $x<xs$ and $v>sv$. This completes the proof of~\eqref{eq:trace(T_u*T_v)}.

It is well known that $T_w^{-1}\in \Span\{T_u\}_{u\leq w}$. Thus, by~\eqref{eq:trace(T_u*T_v)}, $\Htrace(T_w^{-1}T_v)=0$ unless $v\leq w$. We now proceed to prove~\eqref{eq:traces} by induction on $\lvw$. For $v=w$, the result again follows from~\eqref{eq:trace(T_u*T_v)}. For $v<w$, we choose $s\in S$ such that $sw<w$ and then calculate using $T_s^{-1}=q^{-1}T_s+(1-q^{-1})$ that
\begin{equation*}
  T_w^{-1}T_v=
  \begin{cases}
    T_{sw}^{-1}T_{sv}, &\text{if $sv<v$ and $sw<w$,}\\
    q^{-1}T_{sw}^{-1}T_{sv}+(1-q^{-1})T_{sw}^{-1}T_v, &\text{if $sv>v$ and $sw<w$.}\\
  \end{cases}
\end{equation*}
Applying $\Htrace$ and multiplying both sides by $q^{\lvw}$, the result matches perfectly with~\eqref{eq:KLR_dfn}.
\end{proof}

\begin{remark}
\Cref{thm:traces} may also be deduced from \cref{thm:main} by taking the Euler characteristic: Soergel bimodules categorify the Hecke algebra, with Rouquier complexes $\FR(w)$ corresponding to the elements $T_w$, and the zeroth Hochschild cohomology functor $\HH^0$ categorifies the trace $\Htrace$. 
\end{remark}

\section{\for{toc}{Soergel bimodules, Rouquier complexes, and Khovanov--Rozansky homology}\except{toc}{Soergel bimodules, Rouquier complexes, \\ and Khovanov--Rozansky homology}}\label{sec:KR_hom}

In this section, we review Khovanov--Rozansky (KR) link homology. A friendly introduction to most of this material can be found in the excellent recent book~\cite{EMTW}. 
\subsection{Soergel bimodules}\label{sec:soergel-bimodules}
Let  $R := \Symh$ be as in \cref{sec:main-result-equivariant}. It is a graded ring where we set $\deg(y)=2$ for $y\in\h^\ast$. We refer to $\deg(y)$ as the \emph{polynomial degree} (as opposed to the cohomological degree introduced later on). The Weyl group $W$ acts naturally on $R$. Denote by $I$ the indexing set of simple roots of $R$, and thus $W$ is generated by the simple reflections $S=\{s_i\}_{i\in I}$. When $G=\PGL_n(\C)$, recall that  $R=\C[y_1,\dots,y_{n-1}]$ is the polynomial ring and $S=\{s_1,\dots,s_{n-1}\}$ is the set of simple transpositions in $W=S_n$. The action of $S_n$ on $R$ is obtained by restricting the permutation action on $\C[x_1,x_2,\dots,x_n]$ to $R\subset \C[x_1,x_2,\dots,x_n]$, where we identify $y_i=x_i-x_{i+1}$ for $1\leq i\leq n-1$. 
 For example,
\begin{equation}\label{eq:s_i_y_vs_x}
  s_1(y_1)=-y_1,\quad s_{2}(y_1)=y_1+y_{2},\quad s_{3}(y_1)=\dots=s_{n-1}(y_1)=y_1.
\end{equation}

Soergel bimodules are special kinds of \emph{graded $R$-bimodules}, that is, graded $\C$-vector spaces equipped with a  left and a right graded action of $R$. For a graded $R$-bimodule $B=\bigoplus_i B^i$ and $m\in\Z$, we denote by $B\psa{m/2}:=\bigoplus_i B^{i-m}$ the polynomial grading shift by $m$ on $B$. Thus, $y\in\h^\ast$ has degree $2$ as an element of $R$ but has degree $0$ as an element of $R\psa{-1}$.

 Let us introduce the ``building blocks'' of Soergel bimodules.

\begin{definition}
For $s\in S$, let $R^s:=\{r\in R\mid sr=r\}$.
 Define 
\begin{equation}\label{eq:B_s}
B_s:=R\otimes_{R^s} R.
\end{equation}
For a sequence $\uu=(s_{i_1},s_{i_2},\dots,s_{i_m})$ of elements of $S$, let 
\begin{equation}\label{eq:BS_bimod}
\BS_\uu:=B_{s_{i_1}}\otimes_R B_{s_{i_2}}\otimes_R\cdots\otimes_R B_{s_{i_m}}=R\otimes_{R^{s_{i_1}}} R \otimes_{R^{s_{i_2}}} \cdots \otimes_{R^{s_{i_m}}} R.
\end{equation}
\end{definition}
\noindent Both $B_s$ and $\BS_\uu$ are naturally graded $R$-bimodules, called \emph{Bott--Samelson bimodules}, where $R$ acts on the leftmost and the rightmost terms of the tensor product by multiplication.

We let $\SBim$ denote the \emph{category of Soergel bimodules}. By definition, its objects are graded shifts of direct summands of Bott--Samelson bimodules $\BS_\uu$, where $\uu$ runs over all finite sequences of elements in $S$. The morphisms in $\SBim$ are given by degree $0$ maps of $R$-bimodules. The indecomposable objects $\{S_w\}_{w\in W}$ of $\SBim$ are indexed by the elements of $W$: for each $w\in W$ and any reduced word $\uw$ for $w$, $\BS_\uw$ contains a unique indecomposable summand $S_w$ that is not contained in $\BS_\uv$ for any $v<u$ and any reduced word $\uv$ for $v$. Up to isomorphism, the bimodule $S_w$ depends only on $w$ and not on the choice of $\uw$.

\subsection{Rouquier complexes}\label{sec:rouquier-complexes}
We let $\KSBim$ denote the homotopy category of $\SBim$. Its objects are bounded cochain complexes $C^\bul=(\dots \to C^{-1}\to \und{C^0}\to C^1\to\dots)$ of Soergel bimodules, and morphisms are homotopy classes of maps of complexes. When depicting a complex, we usually omit some of the zeroes and underline the object that is in cohomological degree $0$. For example, a complex with only two non-zero entries may be written as $(\und{C^0}\to C^1)$ or as $(0\to\und{C^0}\to C^1)$. We denote by $\ksa[m]$ the \emph{cohomological shift} on $\KSBim$. It shifts each cochain complex $m$ steps to the left: $C^\bul[1]=(\dots\to C^{-1}\to C^0\to \und{C^1}\to\dots)$.

The tensor product $C^\bul\otimes_R D^\bul$ of two cochain complexes is the total complex of a double complex whose entries are $C^i\otimes_R D^j$ for $i,j\in\Z$. The sign of the map $C^i\otimes_R D^j\to C^i\otimes_R D^{j+1}$ is the negation of the obvious one for all even $i$; the differential of the resulting total complex squares to zero.

Since $(W,S)$ is a Coxeter system, we can consider the associated Artin braid group $\BraidW$ generated by $\{\sigma_i\}_{i\in I}$. The following construction is due to Rouquier~\cite{Rou}.
\begin{definition}
For $s=s_i\in S$ and $\sigma=\sigma_i$, define the \emph{Rouquier complexes}
\begin{equation}\label{eq:Rouquier_s}
\FR(\sigma):=(B_s\to \und{R}),\quad \FR(\sigma^{-1}):=(\und{R}\to B_s\psa{-1}),
\end{equation}
where the first map sends $f\otimes g\mapsto fg$ and the second map sends $1\mapsto (\alpha_s\otimes 1+1\otimes \alpha_s)$. Here $\alpha_s\in\h^\ast$ is the simple root corresponding to $s$. Note that both $1\in R$ and $(\alpha_s\otimes 1+1\otimes \alpha_s)\in B_s\psa{-1}$ have polynomial degree $0$. 
For a braid $\beta=\sigma_{i_1}\sigma_{i_2}\cdots\sigma_{i_m}\in \BraidW$, we set
\begin{align*}
\FR(\beta)&:=\FR(\sigma_{i_1})\otimes_R \FR(\sigma_{i_2})\otimes_R\cdots\otimes_R \FR(\sigma_{i_m}),\\ \FR(\beta^{-1})=\FR(\beta)^{-1}&:=\FR(\sigma_{i_m}^{-1})\otimes_R\cdots\otimes_R \FR(\sigma_{i_{2}}^{-1})\otimes_R \FR(\sigma_{i_1}^{-1}).
\end{align*}
We also let $\FR(\id):=(0\to\und{R}\to0)$.
\end{definition}
\noindent A priori, the complex $\FR(\beta)$ depends on the choice the word $(\sigma_{i_1},\sigma_{i_2},\dots,\sigma_{i_m})$. However, modulo homotopy, it does not.
\begin{proposition}[{\cite[Section~3]{Rou}}]\label{prop:Rou}
If $\sigma_{i_1}\sigma_{i_2}\cdots\sigma_{i_m}=\sigma_{i'_1}\sigma_{i'_2}\cdots\sigma_{i'_m}$ in $\BraidW$ then 
\begin{equation*}
  \FR(\sigma_{i_1})\otimes_R \FR(\sigma_{i_2})\otimes_R\cdots\otimes_R \FR(\sigma_{i_m})\cong   \FR(\sigma_{i'_1})\otimes_R \FR(\sigma_{i'_2})\otimes_R\cdots\otimes_R \FR(\sigma_{i'_m})%
\end{equation*}
in $\KSBim$.
\end{proposition}
\noindent For example, one can check that $\FR(\sigma_i)\otimes_R \FR(\sigma_i^{-1})\cong \FR(\id)$. It follows that the functors $(-)\otimes_R\FR(\sigma_i)$ and $(-)\otimes_R\FR(\sigma_i^{-1})$ are mutually inverse biadjoint equivalences of categories: for complexes $C^\bul,D^\bul\in\KSBim$, we have
\begin{align}
\label{eq:biadj1}
  \Hom_{\KSBim}(C^\bul,D^\bul \otimes_R \FR(\sigma_i))&\cong \Hom_{\KSBim}(C^\bul\otimes_R \FR(\sigma_i^{-1}),D^\bul ),\\
\label{eq:biadj2}
\Hom_{\KSBim}(C^\bul,D^\bul \otimes_R \FR(\sigma_i^{-1}))&\cong \Hom_{\KSBim}(C^\bul\otimes_R \FR(\sigma_i),D^\bul ).
\end{align}

\Cref{prop:Rou} allows one to define $\FR(\beta)\in \KSBim$ unambiguously for any braid $\beta\in\BraidW$. Recall that we are interested in the braid $\bvw=\beta(w)\cdot \beta(v)^{-1}$, which corresponds to the complex
\begin{equation}\label{eq:Fvw_dfn}
\Fvw:=\FR(\bvw)=\FR(\beta(w))\otimes_R \FR(\beta(v)^{-1}).
\end{equation}
 
\subsection{KR homology}
Recall from~\eqref{eq:HH0} that the functor $\HH^0$ sends a graded $R$-bimodule $B$ to the graded $R$-module $\HH^0(B)$ of its $R$-invariants. Alternatively, it can be expressed as
\begin{equation}\label{eq:HH0_Hom}
\HH^0(B)=\bigoplus_{\wt\in\Z}\Hom_\SBim(R,B\psa{-\wt/2}).
\end{equation}
\begin{remark}\label{rmk:diag}
The $R$-module $\HH^0(B)$ is free for any Bott-Samelson bimodule $B$. One can make explicit combinatorial computations with this $R$-module (including finding a basis and computing the maps in the Rouquier complexes) using the \emph{diagrammatic calculus} developed by Elias--Williamson \cite{EW}.
\end{remark}
\begin{remark}
Higher Hochschild cohomology functors $\HH^h$, which give the full (triply-graded) KR homology, are the right derived functors of $\HH^0$. They can be computed using a Koszul resolution of $R$; see e.g.~\cite{KhoSoe,Mellit_torus}.
\end{remark}

Applying the functor $\HH^0$ to a complex $C^\bul$ of Soergel bimodules yields a complex $\HH^0(C^\bul)$ of graded $R$-modules. In particular, for each $k\in\Z$, the cohomology $H^k(\HH^0(C^\bul))$ of this complex is a graded $R$-module. For $\wt\in\Z$, we denote by $H^{k,(\wt/2)}(\HH^0(C^\bul))$ its graded piece of polynomial degree $\wt$. It is not hard to check that we have
\begin{equation}\label{eq:HHH=Hom}
H^{k,(\wt/2)}(\HH^0(C^\bul))\cong\Hom_{\KSBim} (R,C^\bul[k]\psa{-\wt/2}).
\end{equation}
For any $\beta\in\BraidW$, $\FR(\beta)$ is concentrated in even polynomial degrees, and thus $H^{k,(\wt/2)}(\HH^0(\FR(\beta)))$ vanishes when $\wt$ is odd. Similarly to~\eqref{eq:HHH0}, we denote
\begin{align}
\label{eq:HHH0_dfn}
\HHH^0(\FR(\beta))&:=\bigoplus_{k,p\in\Z} \HHB kp (\FR(\beta)).
\end{align}
The complex $\Fvw$ is concentrated in cohomological degrees $-\ell(w),-\ell(w)+1,\dots,\ell(v)$, and thus $\HHX k{p}_{v,w}=0$ unless $-\ell(w)\leq k\leq \ell(v)$, and the index $p\in\Z$ is bounded from below.

The functor $\HHC^0$ admits a similar description. Recall that $\C$ is considered an $R$-bimodule on which $\h^\ast$ acts by zero on both sides. Any Soergel bimodule $B\in\SBim$ gives rise to a \emph{Soergel module} $B\otimes_R \C$, which is a graded $R$-module. We let $\SMod$ denote the category of Soergel modules (with morphisms being maps of polynomial degree $0$). By a result of Soergel~\cite{SoeHC} (see~\cite[Proposition~15.27]{EMTW}), for any $B,B'\in\SBim$, we have a natural isomorphism
\begin{equation*}
  \Hom_{\SBim}(B,B')\otimes_R\C \xrasim \Hom_{\SMod}(B\otimes_R\C,B'\otimes_R\C).
\end{equation*}
Applying this to the case $B=R$, we get the following result.

\begin{corollary}\label{cor:SMod}
For any Soergel bimodule $B\in\SBim$, we have
\begin{equation*}
  \HHC^0(B)\cong \Hom_{\SMod}(\C,B\otimes_R\C).
\end{equation*}
\end{corollary}

\subsection{Link components and $R$-module structure}\label{sec:link_cpts_vs_Koszul}
\def\tR{{\tilde R}}
\def\tF{{\tilde F}}
\def\tFR{{\tilde F^\bullet}}
For this section, we assume $W=S_n$.  Let $\beta$ be a braid and $u\in S_n$ the image of $\beta$.  Let $\tR$ denote the polynomial ring $\C[x_1,x_2,\ldots,x_n]$, and let $\tilde F^\bullet(\beta)$ denote the Rouquier complex using $\tR$ instead of $R$; cf.~\eqref{eq:s_i_y_vs_x}.  Thus, $\FR(\beta)$ is a reduced version of $\tilde F^\bullet(\beta)$.  The complex $\tilde F^\bullet(\beta)$ is an $\tR \otimes \tR$ module, and it is known \cite{Ras,GH} that the action of $x_i \otimes 1$ is homotopic to the action of $1 \otimes x_{u(i)}$.  Indeed, it follows from \cite[Proposition 2.11 and Theorem 2.18]{GH}, that there exist cochain maps $\xi_i$ (of polynomial degree $2$ and cohomological degree $-1$) called \emph{dot-sliding homotopies} such that \begin{enumerate}
\item
$d \xi_i + \xi_i d = x_i \otimes 1 - 1 \otimes x_{u(i)}$ for $i=1,2,\dots,n$, and
\item
$\xi_i \xi_j +\xi_j \xi_i = 0$ for $i, j=1,2,\ldots,n$.
\end{enumerate}

In $\HH^0(\tFR(\beta))$, the two $R$-actions are equalized, so $d \xi_i + \xi_i d = x_i - x_{u(i)}$.  Thus, the actions of $x_i$ and $x_j$ on $\HHH^0(\tFR(\beta))$ agree when $i$ and $j$ belong to the same component of the link $\betah$.  Working instead with the smaller polynomial ring $R = \C[x_1-x_2,\ldots,x_{n-1}-x_n] \subset \tR$, we deduce that $x_i -x_j$ acts as $0$ on $\HHH^0(\FR(\beta))$ when $i, j$ belong to the same component of the link $\betah$.  In particular, if $\betah$ is a knot, then the action of $R$ on $\HHH^0(\FR(\beta))$ factors through the natural map $R \to \C$.

Suppose now that $\betah$ is a knot. Denote $z_i:=x_i-x_{u(i)}$ for $i=1,2,\dots,n$. Thus, $R=\C[z_1,\dots,z_{n-1}]$.  Recall that $\HH^0(\FR(\beta))$ is a complex of graded, free (cf. \cref{rmk:diag}) $R$-modules.  Let $a \in \HH^0(\FR(\beta))$ be a nonzero element satisfying $d(a) = 0$. Using the relation $d\xi_i+\xi_i d=z_i$, one can show by induction on $k=0,1,\dots, n-1$ that for all $1\leq i_1<i_2<\dots<i_k\leq n-1$, we have
\begin{equation}\label{eq:Koszul_z}
  d\xi_{i_1i_2\cdots i_k} (a) = \sum_{j=1}^k (-1)^{j-1} z_{i_j} \xi_{i_1\cdots \hat i_j\cdots i_k} (a),
\end{equation}
where $\xi_{i_1i_2\cdots i_k}:=\xi_{i_1}\xi_{i_2}\cdots \xi_{i_k}$, and $\hat i_j$ denotes omission of $\xi_{i_j}$. We therefore obtain a subcomplex $\Kbul(a)$ of $\HH^0(\FR(\beta))$ given by
\begin{align*}
R\cdot \xi_{1\,2\,\cdots\,n-1}(a)\to \cdots &\to \sum_{1 \leq i_1 < \cdots < i_k \leq n-1} R \cdot \xi_{i_1\cdots i_k} (a) \to \cdots\\ 
& \to \sum_{1 \leq i \leq n-1} R \cdot \xi_{i} (a) \to R \cdot a.
\end{align*}
\begin{definition}
Let $a \in \HH^0(\FR(\beta))$ be such that $d(a)=0$. We say that $\Kbul(a)$ is a \emph{Koszul subcomplex} if the set $\{\xi_{i_1\cdots i_k}(a)\mid 0\leq k\leq n-1,\ 1 \leq i_1 < \cdots < i_k \leq n-1\}$ can be extended to a free $R$-module basis of $\HH^0(\FR(\beta))$.
\end{definition}
\noindent It follows from~\eqref{eq:Koszul_z} that we have a natural cochain map $\bigotimes _{i=1}^{n-1} (R \stackrel{z_i}{\longrightarrow} R) \to \HH^0(\FR(\beta))$ with image $\Kbul(a)$, and this map is an isomorphism when $\Kbul(a)$ is a Koszul subcomplex.

Our next goal is to show that $\HH^0(\FR(\beta))$ admits a filtration by Koszul subcomplexes and contractible subcomplexes of the form $R\xrasim R$.

\begin{definition}\label{dfn:xi}
We say that a complex $(C^\bullet,d)$ of finite rank, free, graded $R$-modules \emph{admits a \laction} if there exist endomorphisms $\xi_1,\xi_2,\ldots,\xi_{n-1}$ of cohomological degree $-1$ and polynomial degree $2$ satisfying $d \xi_i + \xi_i d = z_i$ and $\xi_i \xi_j +\xi_j \xi_i = 0$, for all $i,j=1,2,\dots,n-1$.
\end{definition}

We thank the anonymous referee for suggesting to us that the following statement may be deduced from the results of~\cite{GH}.
\begin{proposition}\label{prop:Koszul}
Suppose that $(C^\bullet, d)$ admits a \laction.
Then $C^\bullet$ has a filtration by Koszul complexes and trivial complexes $R \simeq R$. That is, there exists a family of subcomplexes $0=F^\bul_0\subset F^\bul_1\subset\dots\subset F^\bul_t=C^\bul$ such that for all $j=1,2,\dots,t$, $C^\bul/F^\bul_{j-1}$ is free, admits a \laction, and $F^\bul_j/F^\bul_{j-1}$ is either a Koszul subcomplex of $C^\bul/F^\bul_{j-1}$ or a trivial subcomplex isomorphic to $R\xrasim R$.
\end{proposition}
\begin{proof}
Let $C^{k,(p/2)}$ denote the subspace of $C^\bul$ of cohomological degree $k$ and polynomial degree $p$.  Assuming $C^\bullet$ is nonzero, let $D^\bul \subset C^\bullet$ be sum of those nonzero pieces $C^{k,(p/2)}$ where $k+p/2$ is minimal.

Suppose $a \in D^\bul$ satisfies $d(a) =0$.  Then for any $i_1,\ldots,i_k$ we have that $\xi_{i_1\cdots i_k}(a) \in D^\bul$.  Any linearly independent (over $\C$) elements in $D^\bul$ can be extended to a free $R$-module basis of $C^\bul$. Thus, using~\eqref{eq:Koszul_z}, one can show by induction on $k=0,1,\dots,n-1$ that the elements $\{\xi_{i_1\cdots i_k}(a)\mid 1 \leq i_1 < \cdots < i_k \leq n-1\}$ are linearly independent. It follows that $\Kbul(a)$ is a Koszul subcomplex of $C^\bul$, and furthermore that the quotient by this complex is again free and admits a \laction.

Repeating this, we may assume that $d|_{D^\bul}$ is injective. We claim that for any nonzero element $b\in D^\bul$, $d(b)$ may be completed to a free basis of $C^\bul$. Suppose that $b\in D^{k-1}$. We proceed by inverse induction on $k$. For the base case, if $D^{k}=0$ then $d(b)$ is a $\C$-linear combination of free basis elements, so the statement follows. For the induction step, suppose that $b\in D^{k-1}$ satisfies $d(b)\in R\cdot D^k$. Write $d(b)=\sum_{i=1}^{n-1}y_i e_i$ and let $i$ be such that $e_i\neq0$. By the induction hypothesis, $f:=d(e_i)$ may be extended to a free basis of $C^{k+1}$, and we find that the coefficient of $y_if$ in $d^2(b)$ is nonzero, a contradiction. Thus, $d(b)\notin R\cdot D^k$. Comparing the polynomial degree of $d(b)$ to that of $D^k$, we get that $d(b)$ can be extended to an $R$-basis of $C^k$.

 A subcomplex $R \xrasim R$ in $C^\bullet$ is called \emph{splittable} if the quotient by this subcomplex again consists of free $R$-modules. Let $k$ be the smallest index such that $D^k\neq0$, and let $a_1 \in D^k$ be a nonzero element. 
We have shown above that $a_0:=d(a_1)$ may be completed to a free basis of $C^\bul$, and thus $a_1$ and $a_0$ generate a splittable subcomplex $S^\bul\cong (R \xrasim R)$. For any $1\leq i\leq n-1$, we have $\xi_i(a_1)\in D^{k-1}=0$. Using $d\xi_i+\xi_id=z_i$, we get $z_ia_1=\xi_ia_0$. Thus, the subcomplex $S^\bul$ is closed under the action of $\xi_1,\dots,\xi_{n-1}$. Using Gaussian elimination~\cite[Exercise~19.12]{EMTW}, one can check that in this case, the quotient complex $C^\bul/S^\bul$ admits a \laction.

Repeating the above procedure, we obtain the desired filtration.
\end{proof}

Recall that the Koszul resolution of the $R$-module $\C$ by free $R$-modules yields a $2^{n-1}$-dimensional complex $\Tor^R_\bul(\C,\C)\cong (\C \xrightarrow{0} \underline{\C})^{\otimes(n-1)}$. 
\begin{corollary}\label{cor:Koszul_laction}
Suppose that $(C^\bullet, d)$ admits a \laction. Let $(C^\bul_\C,d_\C)$ be the complex of $\C$-vector spaces obtained by setting $y_1=y_2=\dots=y_{n-1}=0$. Thus, $C^\bul_\C:=C^\bul\otimes_R\C$ as in~\eqref{eq:HHC}. Then 
\begin{equation*}%
  H^\bul(C^\bul_\C) \cong  \Tor^R_\bul(\C,H^\bul(C^\bul)) \cong  H^\bul(C^\bul) \otimes (\C \xrightarrow{0} \underline{\C})^{\otimes(n-1)}
\end{equation*}
(as complexes of graded $\C$-vector spaces with zero differentials).
\end{corollary}
\begin{proof}
  It is clear that the filtration constructed in \cref{prop:Koszul} induces an injection $H^\bul(F^\bul_j/F^\bul_{j-1})\hookrightarrow H^\bul(C^\bul/F^\bul_{j-1})$, and a similar statement holds after setting $y_1=y_2=\dots=y_{n-1}=0$. Thus, each Koszul complex $\Kbul(a)$ appearing in the filtration contributes a $1$-dimensional subcomplex to $H^\bul(C^\bul)$. In view of~\eqref{eq:Koszul_z}, $\Kbul(a)$ contributes to $H^\bul(C^\bul_\C)$ a $2^{n-1}$-dimensional subcomplex isomorphic to $(\C \xrightarrow{0} \underline{\C})^{\otimes(n-1)}$.
\end{proof}

\subsection{Link invariant}\label{sec:link-invariant}
For this section, we continue to assume $W=S_n$. The above construction may be turned into a link invariant as we now explain. We follow the conventions of~\cite{MeHog}.

For a braid $\beta\in \BraidSn = \Braid_n$, let $e(\beta)$ denote the \emph{exponent sum} of $\beta$: 
\begin{equation*}
  \beta=\sigma_{i_1}^{\eps_1}\sigma_{i_2}^{\eps_2}\cdots \sigma_{i_m}^{\eps_m} \quad\Longrightarrow\quad e(\beta):=\eps_1+\eps_2+\dots+\eps_m.
\end{equation*}
Thus, $e(\bvw)=\lvw=\ell(w)-\ell(v)$. Next, define
\begin{equation}\label{eq:chi_dfn}
  \chi(\beta):=\frac{e(\beta)-n+\ncyc(\beta)}2
\end{equation}
where $\ncyc(\beta)$ is the number of components of the link $\bhvw$, which equals the number of cycles of the corresponding permutation (obtained from the group homomorphism $\BraidSn\to S_n$ sending $\sigma_i\mapsto s_i$ for each $1\leq i\leq n-1$). It is easy to check that $\chi(\beta)$ is always an integer. Define %
\begin{align*}%
  \PKR(\beta;a,q,t):=&(1-t)^{\cb-1} \left(\qq\tti a^{-2}\right)^{\chib} \\
&\times\sum_{k,p,h\in\Z} (-1)^h q^{\frac{k}2} t^{p+\frac{k}2+h} a^{-2h} \dim H^{k,(p)}(\HH^h (\FR(\beta))).
\end{align*}

Let $\PKRtop(\beta;q,t)$ be its top $a$-degree coefficient:
\begin{equation}\label{eq:PKRtop_dfn}
  \PKRtop(\beta;q,t):=[a^{\mda(\PKR(\beta))}] \PKR(\beta).
\end{equation}
\begin{theorem}[\cite{KhoSoe}]
$\PKR$ and $\PKRtop$ are link invariants: if $\beta\in\BraidSn,\beta'\in\Braid_{S_{n'}}$ are two braids such that the corresponding links $\betah\cong\betah'$ are isotopic then 
\begin{equation*}
  \PKR(\beta;a,q,t)=\PKR(\beta';a,q,t) \quad\text{and}\quad\PKRtop(\beta;q,t)=\PKRtop(\beta';q,t)
\end{equation*}
\end{theorem}
\noindent Thus, it makes sense to write $\PKR(\betah;a,q,t):=\PKR(\beta;a,q,t)$ and $\PK(\betah;q,t):=\PK(\beta;q,t)$. Both $\PKR(\betah;a,q,t)$ and $\PK(\betah;q,t)$ have been recently shown to be $q,t$-symmetric when $\betah$ is a knot~\cite{ObRoz}.

It is well known that $\PKR(\betah;a,q,t)$ specializes to the \FLY polynomial of $\betah$:
\begin{equation}\label{eq:PKR_HOMFLY}
  \PKR(\betah) |_{\tt=-\qqi}= (-1)^\chib\left(z/a\right)^{\cb-1} \HOMP(\betah;a,z) \big|_{z=\qq-\qqi}.
\end{equation}
Note that $\chib$ is \emph{not} a link invariant, but $(-1)^\chib$ and $\cb$ are link invariants.

Clearly, for any braid $\beta$, we have $\mda(\PKR(\beta))\leq -2\chib$. Let $v\leq w\in S_n$. Comparing~\eqref{eq:chi_dfn} with~\eqref{eq:davw}, we find $\davw=-2\chi(\bvw)+\ncyc(\bvw)-1$, and thus the coefficient of $a^{-2\chibvw}$ in $\PKR(\bvw)$ is nonzero, by~\eqref{eq:PKR_HOMFLY} combined with \cref{thm:homfly2:2}. Therefore $\mda(\PKR(\bvw))=-2\chibvw$, and we get the following result.
\begin{proposition}
  For $v\leq w\in S_n$, $\PKRtop(\bhvw)=[a^{-2\chibvw}]\PKR(\bhvw)$ is given by
\begin{equation}\label{eq:PKRtop}
\begin{aligned}
  \PKRtop(\bhvw;q,t)=&(1-t)^{\cbvw-1} \left(\qq\tti\right)^\chibvw \\ &\times \sum_{k,p\in\Z} q^{\frac{k}2} t^{p+\frac{k}2} \dim \HHB kp (\FR(\bvw)).
\end{aligned}
\end{equation}
\end{proposition}

Let us also define the analogous polynomial in the non-equivariant case (cf. \cref{sec:main-result-ordinary}). For $\beta\in\BraidSn$, set
\begin{equation}\label{eq:PKRC}
  \PKRC(\beta;q,t)=\frac{\left(\qq\tti\right)^{\chib}}{\left(1+\qqi\tt\right)^{n-\cb}}\sum_{k,p\in\Z} q^{\frac{k}2} t^{p+\frac{k}2} \dim \HHBC kp (\FR(\beta)).
\end{equation}
\noindent The denominator $\left(1+\qqi\tt\right)^{n-\cb}$ in~\eqref{eq:PKRC} is chosen in view of the discussion in \cref{sec:link_cpts_vs_Koszul}: when $\beta$ is a link with $\cb$ components, a filtration analogous to the one in \cref{prop:Koszul} would involve complexes with $2^{n-\cb}$ terms.

The following result is a consequence of \cref{cor:Koszul_laction}.
\begin{corollary}\label{cor:PKRtop=PKRC}
Assume that $\betah$ is a knot such that $\PKRC(\beta;q,t)\neq0$. Then
\begin{equation*}
  \PKRtop(\beta;q,t)=\PKRC(\beta;q,t).
\end{equation*}
\end{corollary}
\noindent In \cref{sec:ordinary-cohomology}, we give an alternative proof for knots of the form $\bhvw$ for $v\leq w\in S_n$.

\subsection{Examples}\label{sec:Soergel_examples}
We compute  $\HHH^0(\FR(\beta))$ and  $\PK(\betah;q,t)$, as well as $\HHHC^0(\FR(\beta))$ and $\PKRC(\beta;q,t)$, for a few braids $\beta$. Throughout, we assume $G=\PGL_n(\C)$, in which case recall that $R=\C[y_1,\dots,y_{n-1}]$ is a polynomial ring. These examples are summarized in \cref{tab:examples}. We abbreviate $\otimes_R$ by $\otimes$.

\begin{example}[Unknot I]\label{ex:soe_unkn_1}
Let $n=1$, $v=\id$, $w=\id$, and thus $\cbvw=1$ and $\chibvw=0$. We have $\FR(\bvw)=(0\to \underline{R}\to 0)$ and $R=\C$. Thus, the only nonzero term is $\HHX 00_{v,w}\cong \HHXC 00_{v,w}\cong \C$. We have
\begin{equation*}
\PKRtop(\bhvw;q,t)=\PKRC(\bhvw;q,t)=1.
\end{equation*}
Note that any $(1,b)$-torus knot is isotopic to the unknot, and we have $\Cat_{1,b}(q,t)=1$.

\end{example}

\begin{example}[Unknot II]\label{ex:soe_unkn_2}
  Let $n=2$, $v=\id$, $w=s_1$, and thus $\cbvw=1$ and $\chibvw=0$. We have $\Fvw=(B_{s_1}\to \underline{R})$. It is easy to see that $\HH^0(B_{s_1})$ is a free $R$-module spanned by $(y_1\otimes 1+1\otimes y_1)$, and thus $\HH^0(B_{s_1})\cong R\psa{1}$, and $\HH^0(\Fvw)=(R\psa{1}\to \underline{R})$, with the map sending $1\mapsto 2y_1$. The only nonzero term is $\HHX 00_{v,w}\cong\C$. Tensoring with $\C$, we get $\HHC^0(\Fvw)=(\C\psa{1}\xrightarrow{0} \underline{\C})$, so there are two nonzero terms: $\HHXC 00_{v,w}\cong \HHXC {-1}1_{v,w} \cong\C$. Therefore
\begin{equation*}
  \PKRtop(\bhvw;q,t)=1 \quad\text{and}\quad \PKRC(\bhvw;q,t)=\frac1{1+\qqi\tt} \left(1+\qqi\tt\right)=1.
\end{equation*}
\end{example}

Let us also consider an example of $\bhvw$ for $v\not\leq w$. 
\begin{example}[Unknot III]\label{ex:unknot_III}
  Let $n=2$, $v=s_1$, $w=\id$, and thus $\ncyc(wv^{-1})=1$ and ${\chi(\bvw)=-1}$. We have $\Fvw=(\underline{R}\to B_{s_1}\psa{-1})$,  $\HH^0(\Fvw)=(\underline{R}\xrasim R)$, and $\HHC^0(\Fvw)=(\underline{\C}\xrasim \C)$ so the right-hand sides  of~\eqref{eq:PKRtop} and~\eqref{eq:PKRC} are zero:
\begin{equation*}
  [a^2]\PKR(\bhvw;a,q,t)=0 \quad\text{and}\quad \PKRC(\bvw;q,t)=0.
\end{equation*}
This is consistent with the fact that $\PKR(\bhvw)$ is a link invariant satisfying $\PKR(\unkn)=1$, therefore  by~\eqref{eq:PKRtop_dfn} we have $\PKRtop(\bhvw)=1$. For a computation of $\HH^1(\FR(\bvw))$, see~\cite{Mellit_torus}.
\end{example}

In the next two examples, we have $\cb>1$. We start with the case of the \emph{$2$-component unlink} $\unlk$. It is the closure of $\id\in \Braid_{S_2}$, but we consider the representative $\beta_{s_1,s_1}=\sigma_1\sigma_1^{-1}$ instead.
\begin{example}[$2$-component unlink]\label{ex:soe_unlk}
Let $n=2$, $v=s_1$, $w=s_1$, and thus $\ncyc(wv^{-1})=2$ and $\chibvw=0$. We have
\begin{equation*}
  \FR(\bvw)=(B_{s_1}\to (\und{R\oplus (B_{s_1}\otimes B_{s_1}\psa{-1})})\to B_{s_1}\psa{-1}).
\end{equation*}
We apply the well-known (see, e.g., \cite[Example~3.12]{GKS_lec_notes}) Soergel bimodule isomorphism $B_{s_1}\otimes B_{s_1}\cong B_{s_1}\psa{1}\oplus B_{s_1}$, sending $1\otimes1\otimes1\mapsto (0,1\otimes1)$ and $1\otimes y_1\otimes1\mapsto (1\otimes1,0)$. Next, we use Gaussian elimination~\cite[Exercise~19.12]{EMTW} to obtain
$\FR(\bvw)\cong \FR(\id)$ in $\KSBim$, in agreement with \cref{prop:Rou}. We have $R=\C[y_1]$, $\HH^0(R)\cong R$, and $\HHC^0(R)\cong \C$. (More generally, recall from \cref{rmk:diag} that the $R$-module $\HH^0(B)$ is always free.) Therefore the only nonzero terms are $\HHX 0p_{v,w}\cong \C$ for $p=0,1,2,\dots$, and $\HHXC 00_{v,w}\cong \C$. We find
\begin{equation*}
  \PKRtop(\bhvw;q,t)=(1-t)(1+t+t^2+\cdots)=1 \quad\text{and}\quad \PKRC(\bvw;q,t)=1.
\end{equation*}
\end{example}

\begin{table}
\def\lcw{1.7cm}
\def\Hcw{0.65cm}
\makebox[1.0\textwidth]{
\scalebox{0.77}{
  \begin{tabular}{cc}

  \begin{tabular}{|C{\lcw}|C{\Hcw}C{\Hcw}C{\Hcw}C{\Hcw}C{\Hcw}|}\hline
                               $H^k$ & $H^0$ & $H^1$ & $H^2$ & $H^3$ & $H^4$ \\\hline
                               $k-p=0$ & $1$ & $3$ & $4$ & $3$ & $1$ \\\hline
                             \end{tabular}
&

  \begin{tabular}{|C{\lcw}|C{\Hcw}C{\Hcw}C{\Hcw}C{\Hcw}C{\Hcw}C{\Hcw}C{\Hcw}|}\hline
                               $H^k$ & $H^0$ & $H^1$ & $H^2$ & $H^3$ & $H^4$ & $H^5$ & $H^6$ \\\hline
                               $k-p=0$ & $1$ & $4$ & $7$ & $8$ & $7$ & $4$ & $1$ \\\hline
                             \end{tabular}

\\ 
$H^{k,(p,p)}(\Pio_{2,4})$ & $H^{k,(p,p)}(\Pio_{2,5})$ 
\\ 
\\

 \begin{tabular}{|C{\lcw}|C{\Hcw}C{\Hcw}C{\Hcw}C{\Hcw}C{\Hcw}|}\hline
                               $H^k$ & $H^{-4}$ & $H^{-3}$ & $H^{-2}$ & $H^{-1}$ & $H^0$ \\\hline
                               $k+p=0$ & $1$ & $3$ & $4$ & $3$ & $1$ \\\hline
                             \end{tabular}
&

 \begin{tabular}{|C{\lcw}|C{\Hcw}C{\Hcw}C{\Hcw}C{\Hcw}C{\Hcw}C{\Hcw}C{\Hcw}|}\hline
                               $H^k$ & $H^{-6}$ & $H^{-5}$ & $H^{-4}$ & $H^{-3}$ & $H^{-2}$ & $H^{-1}$ & $H^0$ \\\hline
                               $k+p=0$ & $1$ & $4$ & $7$ & $8$ & $7$ & $4$ & $1$ \\\hline
                             \end{tabular}
\\
$H^{k,(p)}\HHC^0(\FR(\beta_{f_{2,4}}))$ & $H^{k,(p)}\HHC^0(\FR(\beta_{f_{2,5}}))$ 
  \end{tabular}
}
}
\vspace{0.1in}
  \caption{\label{tab:MHT_vs_HHHC} Comparing the mixed Hodge tables of $\Pio_{2,4}$ and $\Pio_{2,5}$ (top) to $\HHHC^0$ of the associated links (bottom).}
\end{table}

The \emph{Hopf link} $\betah=\hopflink$ consists of two linked unknots. It is isotopic to $\betah_{f_{2,4}}$, as well as to the closure of $(\sigma_1)^2\in\Braid_{S_2}$.
\begin{example}[Hopf link]\label{ex:KR_Hopf}
Let $n=2$, $\beta=(\sigma_1)^2$, and thus $\cb=2$ and $\chib=1$. We have 
\begin{equation*}
\FR(\beta)=(B_{s_1}\otimes B_{s_1}\to B_{s_1}\oplus B_{s_1}\to \underline{R}).
\end{equation*}
Using Gaussian elimination as in \cref{ex:soe_unlk}, we obtain
\begin{equation*}%
  \FR(\beta)\cong (B_{s_1}\psa{1}\to B_{s_1}\to \underline{R}).
\end{equation*}
Here, the first map sends $1\otimes1\mapsto y_1\otimes1-1\otimes y_1$, and the second map sends $1\otimes1\mapsto1$. 
Taking $R$-invariants (cf. \cref{rmk:diag}), we find 
\begin{equation}\label{eq:Hopf_mats}
  \HH^0(\FR(\beta))=\left(  R\psa{2} \xrightarrow{0} R\psa{1} \xrightarrow{2y_1} \underline{R} \right).
\end{equation}
We get $\HHB 00(\FR(\beta))\cong\C$ and $H^{-2}(\HH^0(\FR(\beta)))\cong R\psa{2}$. In other words, 
\begin{equation}\label{eq:Hopf_HHB}
  \HHB 00(\FR(\beta))\cong\C, \quad%
  \HHB {-2}p (\FR(\beta))\cong \C \quad\text{for $p=2,3,4,\dots$\;.}
\end{equation}
 Sending $y_1\to 0$ in~\eqref{eq:Hopf_mats}, we find 
\begin{equation*}
\HHBC kp(\FR(\beta)) \cong 
\begin{cases}
  \C, &\text{if $(k,p)\in\{(0,0),(-1,1),(-2,2)\}$;}\\
0,&\text{otherwise.}
\end{cases}
\end{equation*}
 Thus,
\begin{align*}
  \PKRtop(\betah;q,t)&=(1-t)\qq\tti \left(1+t/q\left(1+t+t^2+\cdots\right)\right)\\
                     &=\qq\tti-\qq\tt+\qqi\tt,\\
  \PKRC(\beta;q,t)&= \qq\tti \left(1+\qqi\tt+t/q\right)=\qq\tti+1+\qqi\tt.
\end{align*}

\begin{remark}\label{rmk:f24_HHHC}
Since $\PKRtop(\betah;q,t)$ is a link invariant and $\betah\cong \betah_{2,4}$ (with $\chi(\betah)= \chi(\betah_{2,4})$), we see that $\HHH^0(\FR(\beta))\cong\HHH^0(\FR(\beta_{2,4}))$ as bigraded vector spaces. Using an elaborate computation, one can also check that $\PKRC(\beta;q,t)=\PKRC(\beta_{2,4};q,t)$. However, observe that~\eqref{eq:PKRC} has $\left(1+\qqi\tt\right)^{n-\cb}$ in the denominator, where $n=2$ for $\beta$ and $n=4$ for $\beta_{2,4}$. Thus, $\HHHC^0(\FR(\beta_{2,4}))$ differs from $\HHHC^0(\FR(\beta))$ by ``multiplication by $\left(1+\qqi\tt\right)^2$,'' and the actual bigraded dimensions of $\HHHC^0(\FR(\beta_{2,4}))$ are given in \tabref{tab:MHT_vs_HHHC}(bottom left).
\end{remark} 
\end{example}

\begin{remark}
  We have a resolution of $\C$ by free $R$-modules: $0\to R\xrightarrow{y_1} R\to \C\to0$. Thus, $\Tor^R_\bul(\C,\C)=(\C\xrightarrow{0}\C)$. Noting that $\Tor^R_\bul(\C,R)=\C$, we see that $\HHHC^0(\FR(\beta))\cong \Tor^R_\bul(\C, \HHH^0(\FR(\beta)))$. We conjecture that this holds more generally for all links; see~\eqref{eq:Torconj}.
\end{remark}

As we explained in \cref{rmk:torus_knot}, for $k=2$ and $n=5$, $\betah_{\fkn}$ is the $(2,3)$-torus knot, which is isotopic to the trefoil knot: $\betah_{\fkn}\cong \trefoil$. It can be alternatively obtained as the closure of the braid $(\sigma_1)^3\in\Braid_2$.
\begin{example}[Trefoil knot]\label{ex:soe_trefoil}
Let $n=2$, $\beta=(\sigma_1)^3$, and thus $\cb=1$ and $\chib=1$. We have 
\begin{equation*}
\FR(\beta)=(B_{s_1}^{\otimes3}\to 3B_{s_1}^{\otimes2}\to 3B_{s_1} \to \underline{R}).
\end{equation*}
Here $3B_{s_1}^{\otimes2}$ denotes the direct sum of three copies of $B_{s_1}\otimes B_{s_1}$, etc. Applying Gaussian elimination as in \cref{ex:soe_unlk}, 
we arrive at a simplified complex
\begin{equation*}
\FR(\beta)\cong (B_{s_1}\psa{2}\to B_{s_1}\psa{1}\to B_{s_1} \to \underline{R})
\end{equation*}
with the three maps given by $1\otimes1\mapsto y_1\otimes1+1\otimes y_1$, $1\otimes1\mapsto y_1\otimes1-1\otimes y_1$, and $1\otimes1\mapsto1$, respectively. Taking $R$-invariants, we find  
\begin{align*}
  \HH^0(\FR(\beta))=\left(  R\psa{3} \xrightarrow{2y_1} R\psa{2} \xrightarrow{0} R\psa{1} \xrightarrow{2y_1} \underline{R} \right).
\end{align*}
We find that the only nonzero terms are 
\begin{equation*}%
  \HHB 00(\FR(\beta))\cong\HHB {-2}2(\FR(\beta))\cong\C.
\end{equation*}
 Sending $y_1\to 0$, we also compute $\HHBC kp(\FR(\beta))$, which leads to
\begin{equation*}
  \PKRtop(\betah;q,t)=\PKRC(\beta;q,t)=\qq\tti+\qqi \tt.
\end{equation*}
The corresponding $q,t$-Catalan number is $\Cat_{2,3}(q,t)=q+t=\qq\tt\cdot  \PKRtop(\betah;q,t)$, in agreement with~\eqref{eq:main_KR}.
\begin{remark}
We have $\betah\cong \bhvw$ for $v=\id,w=f_{2,5}\in S_5$ and $\chib=\chibvw=1$, thus, $\PKRtop(\betah)=\PKRtop(\bhvw)$.
Similarly to \cref{rmk:f24_HHHC}, we may compute that $\PKRC(\beta;q,t)=\PKRC(\beta_{2,5};q,t)$, and thus $\HHHC^0(\FR(\beta_{2,5}))$ is given in \tabref{tab:MHT_vs_HHHC}(bottom right).
\end{remark}
\end{example}

\section{Cohomology of positroid and Richardson varieties}\label{sec:cohom-pos-rich}
We briefly review background on positroid varieties, Richardson varieties and the various versions of cohomology that we will be using.

\subsection{Positroid varieties}\label{sec:positroid-varieties}
Recall from \cref{sec:positr-vari-grassm} that the Grassmannian $\Gr(k,n)$ is identified with the space of $k\times n$ matrices modulo row operations. Given a $k\times n$ matrix $A$, we let $\RowSpan(A)\in\Gr(k,n)$ denote its row span and $A_1,A_2,\dots, A_n$ be its columns. We extend this to a sequence $(A_j)_{j\in\Z}$ by requiring
\begin{equation*}
  A_{j+n}=A_j \quad\text{for all $j\in\Z$.}
\end{equation*}

\begin{definition}[\cite{KLS}]
A bijection $\ft:\Z\to\Z$ is called a \emph{$(k,n)$-bounded affine permutation} if it satisfies
\begin{itemize}
\item $\ft(j+n)=\ft(j) + n$ for all $j\in\Z$,
\item $\sum_{j=1}^{n} (\ft(j)-j)=kn$, and
\item $j\leq \ft(j)\leq j+n$ for all $j\in\Z$.
\end{itemize}
\end{definition}
\noindent Alternatively, the second condition can be replaced with $k=\#\{j\in[n]\mid f(j)>n\}$. 

We let $\Bkn$ denote the (finite) set of $(k,n)$-bounded affine permutations. For a full rank $k\times n$ matrix $A$, we let $\ft_A:\Z\to\Z$ be given by
\begin{equation}\label{eq:f_dfn}
  \ft_A(i)=\min\{j\geq i\mid A_i\in \Span \left(A_{i+1},A_{i+2},\dots,A_j\right)\} \quad\text{for $i\in\Z$.}
\end{equation}
For example, if $A_i$ is a zero column then $\ft_A(i)=i$, and if $A_i$ is not in the span of other columns then $\ft_A(i)=i+n$. It is known~\cite{KLS} that $\ft_A$ is a $(k,n)$-bounded affine permutation which depends only on the row span of $A$. 
 The \emph{positroid stratification} of $\Gr(k,n)$ is given by
\begin{equation*}
  \Gr(k,n)=\bigsqcup_{\ft\in \Bkn} \Pio_{\ft}, \quad\text{where}\quad \Pio_{\ft}:=\{\RowSpan(A)\in\Gr(k,n)\mid \ft_A=\ft\}.
\end{equation*}

We extend any permutation $u\in S_n$ to a bijection $\ut:\Z\to\Z$ satisfying $\ut(j+n)=\ut(j)+n$ for all $n$. We introduce a $(k,n)$-bounded affine permutation $\taukn:\Z\to\Z$, determined by 
\begin{equation*}
  \taukn(j)=
  \begin{cases}
    j+n, &\text{if $1\leq j\leq k$;}\\
    j,&\text{if $k+1\leq j\leq n$.}
  \end{cases}
\end{equation*}

Recall that $\Qkn:=\{(v,w)\in S_n\times S_n\mid v\leq w\text{ and $w$ is $k$-Grassmannian}\}$. The following result explains the bijection $(v,w)\mapsto \fvw$ introduced in \cref{prop:v_w}.
\begin{proposition}[{\cite[Proposition~3.15]{KLS}}] \label{prop:v_w_aff}
For every $\ft\in\Bkn$, there exists a unique pair $(v, w)\in\Qkn$ such that $\ft=\wtl\circ \taukn \circ \vt^{-1}$.
\end{proposition}
\noindent Here, $\circ$ denotes the usual composition of bijections $\Z\to\Z$. We thus define $\fvw:=\wtl\circ \taukn \circ \vt^{-1}$.  Furthermore, we have the following relationship between positroid and Richardson varieties.
\begin{proposition}[{\cite[Theorem~5.9]{KLS}}] \label{prop:KLS}
Let $G=\PGL_n(\C)$. For each $f=\fvw\in\Bkn$, the natural projection map $\Fl(n) \to\Gr(k,n)$ restricts to an isomorphism $\Rich_v^w\cong \Pio_f$. Thus, open positroid varieties are special cases of open Richardson varieties. 
\end{proposition}

\subsection{Torus action and Richardson varieties}\label{sec:torus-acti-rich}
The goal of this section is to prove \cref{prop:free}. We start by generalizing one direction to Richardson varieties of arbitrary type; the type $A$ specialization is discussed below. Let $G$ be a complex semisimple algebraic group of adjoint type and of rank $r$, and let $\Gsc$ denote the simply-connected group of the same Dynkin type.  We use the notation $\Tsc,\Bsc,\Usc$ for the corresponding subgroups of $\Gsc$.

Let us give a convenient well-known description (see~\cite[Theorem~2.3]{BGY}, \cite[Lemma~2.2]{Lec}, or~\cite[Lemma~3.1]{GL}) of $\Rich_v^w$ as an explicit affine variety. For $v\in W=N_{\Gsc}(\Tsc)/\Tsc$, let $\dv\in \Gsc$ denote an arbitrary fixed representative. For $v\leq w$, denote 
\begin{equation}\label{eq:str_dfn}
\strvw:=  \dv \Usc_-\cap \Usc_-\dv \cap \Bsc w \Bsc.
\end{equation}
Observe that the set $\Bsc w \Bsc$ does not depend on the choice of a representative for $w$.  

\begin{lemma}\label{lemma:strawberry}
The map $g\mapsto g\Bsc/\Bsc$ provides an isomorphism 
\begin{equation}\label{eq:strawberry}
\pushQED{\qed} 
 \strvw  \xrasim \Rich_v^w.
\qedhere
\popQED
\end{equation}
\end{lemma}

For $u \in W$, let $\mu(u)$ denote the dimension of the eigenspace with eigenvalue $1$ for $u$ acting on $\h^\ast$.  We say that $u \in W$ is \emph{elliptic} if $\mu(u)=0$.

Let $P$ denote the weight lattice of the root system of $G$ (that is, the character lattice of $\Gsc$) and $Q \subset P$ denote the root lattice (that is, the character lattice of $G$).  We let $\omega_1,\ldots,\omega_r \in P$ denote the fundamental weights. 
For $u \in W$, we note that $(u-\id)P \subseteq Q$.

For $\gamma,\delta\in P$, let $\Delta_{\gamma,\delta}$ denote the corresponding \emph{generalized minor}~\cite{BeZe,FZ_dB} for $\Gsc$. 
The condition that $g \in \Bsc w \Bsc$ implies that
\begin{equation}\label{eq:Delta_neq_0}
  \Delta_{w\omega_i,\omega_i}(g)\neq 0 \quad\text{for $i=1,2,\dots,r$.}
\end{equation}
Indeed, by~\cite[Definition~1.4]{FZ_dB}, we have $\Delta_{w\omega_i,\omega_i}(g)=\Delta_{\omega_i,\omega_i}(\lline{w^{-1}}g)$ for a certain representative $\lline{w^{-1}}\in G$ of $w^{-1}$, so~\eqref{eq:Delta_neq_0} follows from~\cite[Proposition~2.9 and Corollary~2.5]{FZ_dB}.

\begin{proposition}\label{prop:elliptic_free}
Suppose that $vw^{-1}$ is elliptic and $(vw^{-1}-\id)P = Q$.  Then $T$ acts freely on $\Rich_v^w$
and we have a $T$-equivariant isomorphism
\begin{equation*}
  \Rich_v^w\cong (\Rich_v^w/T)\times T.
\end{equation*}
\end{proposition}
\begin{proof}
Define 
\begin{equation*}
\Ngauge:=\{x \in \strvw\mid  \Delta_{\omega_i,w \omega_i}(g)=1 \quad\text{for $i=1,2,\dots,r$.}\}
\end{equation*}
We will show that $\Ngauge\cong \Rich_v^w/T$ and $\Rich_v^w\cong \Ngauge \times T$. 

The action of $\Tsc$ on $\Rich_v^w$ corresponds to the action $t \cdot g := t g \dv^{-1}t^{-1}\dv$ for $t \in \Tsc$ and $g \in \strvw$.    Let $\chi_{\omega}=\Delta_{\omega,\omega}: \Tsc \to \C$ denote the character corresponding to a weight $\omega \in P$.  Then 
\begin{equation*}%
  \Delta_{ w\omega_i, \omega_i}(t g \dv^{-1}t^{-1}\dv) = \chi_{w\omega_i}(t) \chi_{\omega_i}(\dv^{-1}t^{-1}\dv) \Delta_{ w\omega_i, \omega_i}( g).
\end{equation*}
Since $\chi_{\omega_i}(\dv^{-1}t^{-1}\dv) = \chi_{v\omega_i}(t^{-1})$, the weight of this generalized minor is $(w - v)\omega_i=-v(\id-v^{-1}w)\omega_i$.  The condition $\mu(vw^{-1})=0$ implies that $\id -v^{-1}w$ is invertible.  The condition $(vw^{-1}-\id)P = Q$ implies that $(\id - v^{-1}w)\omega_1, (\id - v^{-1}w)\omega_2,\ldots, (\id - v^{-1}w)\omega_r$ form a $\Z$-basis of $Q$.  Since $Q$ is the character lattice of $T$, it follows that the action of $T$ on $\strvw$ is free and that the functions $\Delta_{w\omega_i,\omega_i}$, $i=1,2,\dots,r$, can be simultaneously set to $1$ by a unique element $t \in T$.  It follows that $\Ngauge\cong \Rich_v^w/T$ and $\Rich_v^w\cong \Ngauge \times T$. 
\end{proof}

\begin{remark}
It would be interesting to classify elliptic elements $u\in W$ satisfying the condition $(u-\id)P = Q$, which depends only on the conjugacy class of $u$.  It is not satisfied for all elliptic elements. For example, in type $D_4$, the longest element $w_0$ acts by $-\id$ on $\h^\ast$, but we have $2P\subsetneq Q$ since $\frac{\alpha_i}2\notin P$ for any $i\in I$. 
\end{remark}

\begin{proposition}
Suppose that $c \in W$ is a Coxeter element.  Then $c$ is elliptic and satisfies $(c-\id)P = Q$.
\end{proposition}
\begin{proof}
We may assume that $c$ is a standard Coxeter element.  That is, $c = s_1 s_2 \cdots s_r$ where $s_i$ are simple generators corresponding to positive simple roots $\alpha_1,\ldots,\alpha_r$.  Define roots $\beta_1 = \alpha_1$, $\beta_2 = s_1 \alpha_2$, $\ldots, \beta_r = s_1s_2\cdots s_{r-1}\alpha_r$. By~\cite[Lemma~10.2]{LP}, $c$ is elliptic.  By \cite[Proposition 10.5]{LP} we have $(\id - c)\omega_i = \beta_i \in \alpha_i + \sum_{j < i} \Z\alpha_j$.  It follows that $(c-\id)P = \bigoplus_{i=1}^r \Z \alpha_i = Q$. 
\end{proof}

Let us now consider the type $A$ case, where $G=\PGL_n(\C)$ and $\Gsc=\SL_n(\C)$. For a permutation $u\in S_n$, any of the following conditions are equivalent: (i) $u$ is a single cycle, (ii) $u$ is an elliptic element, and (iii) $u$ is a Coxeter element. Each generalized minor $\Delta_{\omega_i,u\omega_i}:\Gsc\to\C$ is the usual matrix minor with row set $\{1,2,\dots,i\}$ and column set $\{u(1),u(2),\dots,u(i)\}$. The representative $\dv$ in~\eqref{eq:str_dfn} may be chosen to be a signed permutation matrix of $v$. In the next result, $\ncyc(\cdot)$ denotes the number of cycles of a permutation; cf.~\eqref{eq:ncyc_dfn} (as opposed the Coxeter element $c$ considered above).

\begin{corollary}\label{cor:Rich_T_action}
For all $v\leq w\in S_n$ such that $\ncyc(wv^{-1})=1$, the $T$-action on $\Rich_v^w$ is free and we have a $T$-equivariant isomorphism
\begin{equation*}
  \Rich_v^w\cong (\Rich_v^w/T)\times T.
\end{equation*}
\end{corollary}

\begin{remark}\label{rmk:necklace}
When $\Rich_v^w\cong \Pio_f$ (see \cref{prop:KLS}), the functions $\Delta_{[1,i],w[1,i]}$ on $\strvw\cong\Rich_v^w$ coincide with the Pl\"ucker coordinates on $\Pio_f$ corresponding to the \emph{Grassmann necklace} of $f$; see the proof of \cite[Lemma~4.7]{GL}. In particular, for $f=\fkn$, these are the cyclically consecutive maximal minors as in~\eqref{eq:Piokn_dfn}.
\end{remark}

\begin{proof}[Proof of \cref{prop:free}]
If $\ncyc(\fmod)=1$ then $T$ acts freely on $\Pio_f$ by \cref{cor:Rich_T_action}.  We prove the converse. For $f\in\Bkn$, let us construct a particular representative $\Xmin_f\in \Pio_f$. If $f(i)=i$ for some $i\in\Z$ then the corresponding column is zero, so we may assume $f(i)\neq i$ for all $i$. Let us write $\fmod$ in cycle notation: 
\begin{equation*}%
\fmod=(a_1a_2\cdots a_{m_1})(a_{m_1+1}a_{m_1+2}\cdots a_{m_2})\cdots(a_{m_r+1}\cdots a_n)  
\end{equation*}
 so that the minimal index of each cycle comes first. We label these indices left to right: set $\la(a_1):=1$, and for $i=1,2,\dots,n-1$, set $\la(a_{i+1})=\la(a_i)$ if $a_i<a_{i+1}$ and they belong to the same cycle, and $\la(a_{i+1})=\la(a_i)+1$ otherwise. It is easy to check that $\la(a_n)=k$. The element $\Xmin_f$ is the row span of the $k\times n$ matrix $M=(m_{i,j})$ whose only nonzero entries are $m_{\la(a_i),a_i}=1$. One checks using~\eqref{eq:f_dfn} that $\Xmin_f\in\Pio_f$. Furthermore, rescaling all columns that belong to a single cycle of $\fmod$ by the same value preserves the element $\Xmin_f$. Therefore when $\ncyc(\fmod)>1$, the $T$-action on $\Pio_f$ is not free.
\end{proof}

\subsection{Mixed Hodge structure}\label{sec:mixed-hodge-struct}
We follow the conventions of~\cite{LS}; see~\cite{Del71,PS} for further background. The results of~\cite{LS} apply to \emph{cluster varieties}. It was shown in~\cite{GL} that open positroid varieties $\Pio_f$, $f\in\Bkn$ are cluster varieties. By~\cite[Lemma~3.6]{GL}, setting the functions $\Delta_{ w\omega_i, \omega_i}$ to $1$ as we did in the proof of \cref{prop:elliptic_free} corresponds to setting the frozen variables to $1$ in the cluster structure on $\Pio_f$. Thus, for $f\in\Bknc$, $\Pio_f/T$ is a cluster variety with no frozen variables.

Consider a  smooth complex algebraic variety $Y$ of dimension $d$. By~\cite[Lemma-Definition~3.4]{PS}, the cohomology $H^k(Y,\C)$ and the compactly supported cohomology $H^k_c(Y,\C)$ are endowed with a \emph{Deligne splitting} 
\begin{equation*}
  H^k(Y,\C)=\bigoplus_{p,q\in\Z} H^{k,(p,q)}(Y,\C) \quad\text{and}\quad   H_c^k(Y,\C)=\bigoplus_{p,q\in\Z} H_c^{k,(p,q)}(Y,\C).
\end{equation*}
This splitting is functorial and satisfies the \Poincare duality~\cite[Theorem~6.23]{PS}:
\begin{equation}\label{eq:Poincare}
  H^{k,(p,q)}(Y,\C)\cong H^{2d-k,(d-p,d-q)}_c(Y,\C) \quad\text{for all $k,p,q\in\Z$}.
\end{equation}

We say that the cohomology of $Y$ is of  \emph{Hodge--Tate type} if $H^{k,(p,q)}(Y,\C)=0$ whenever $p\neq q$.  The notions of mixed Hodge structure and Hodge--Tate type also apply to equivariant cohomology.

\begin{definition}
Let $Y$ be a $d$-dimensional complex variety whose cohomology is of Hodge--Tate type. Define its \emph{mixed Hodge polynomial} $\Poinc(Y; q,t)\in\N[q^{\frac12},t^{\frac12}]$ by 
\begin{equation}\label{eq:Poinc_MHT_dfn}
  \Poinc(Y;q,t):=\sum_{k,p\in \Z}  q^{p-\frac k2} t^{\frac{d-k}2} \dim H^{k,(p,p)}(Y,\C).
\end{equation}
\end{definition}

\noindent We have $H^{k,(p,p)}(Y,\C)=0$ for $p>k$. By convention, we set $H^{k,(r,r)}(Y,\C):=0$ for $r\notin\Z$.

It is convenient to record the dimensions of the spaces $H^{k,(p,p)}(Y,\C)$ in a \emph{mixed Hodge table}: the columns are labeled by $H^0,H^1,\dots,H^d$, while the rows are labeled by $k-p=0,1,2,\dots$\;. Thus, an entry in a column labeled by $H^k$ and in a row labeled by $k-p$ encodes the dimension of $H^{k,(p,p)}(Y,\C)$.  Examples of mixed Hodge tables are given in Tables~\ref{tab:E8},~\ref{tab:MHT_vs_HHHC} and~\cite[Section~6]{LS}. For instance, we see from~\eqref{eq:Poinc_MHT_dfn} that the two rows of \cref{tab:E8} yield
\begin{equation*}
  \Poinc(\Pio_{3,8}/T;q,t)=\left( q^{4} + q^{3} t + q^{2} t^{2} + q t^{3} +  t^{4} \right) + \left( q^{2} t + q t^{2} \right),
\end{equation*}
confirming the computation in \cref{ex:E8}.

We say that a polynomial $\Poinc(q,t)\in\N[q^{\frac12},t^{\frac12}]$ is \emph{$q,t$-symmetric} if $\Poinc(q,t)=\Poinc(t,q)$. We say that $\Poinc$ is \emph{$q,t$-unimodal} if for each $a,b\in \frac12\Z$, the coefficients $([q^{a-k}t^{b+k}]\Poinc)_{k\in\Z}$ 
 form a unimodal sequence.\footnote{This property is sometimes called \emph{parity unimodality} since the terms with integer degrees and with half-integer degrees are required to form separate unimodal sequences.} Recall that  \cref{thm:Koszul}, proved in \cref{sec:Koszul}, states that the polynomial $\Poinc(\Rich_v^w;q,t)$ is $q,t$-symmetric for all $v\leq w\in W$.

A special case of \cref{thm:Koszul} for positroid varieties follows from the \emph{curious Lefschetz theorem} proved in~\cite[Theorem~8.3]{LS}. We say that a $2d$-dimensional complex algebraic variety $Y$ of Hodge--Tate type satisfies the curious Lefschetz theorem if there is a class $\gamma \in H^{2,(2,2)}(Y,\C)$ inducing isomorphisms
\begin{equation}\label{eq:CL}
\abxcup  \gamma^{d-p}: H^{p+s,(p,p)}(Y,\C) \simeq H^{2d-p+s,(2d-p,2d-p)}(Y,\C). 
\end{equation}
As explained in~\cite{LS}, for cluster varieties, one can choose $\gamma$ to be the Gekhtman--Shapiro--Vainshtein form~\cite{GSV}.  In the case that $\Pio_f$ is odd-dimensional, the product $\Pio_f \times \C^\ast$ will satisfy~\eqref{eq:CL}, and using the K\"unneth theorem, unimodality and symmetry for $\Poincf q,t)$ can be deduced. 

The curious Lefschetz theorem for positroid varieties relies on the fact that they admit cluster structures~\cite{GL}, which is not yet available for Richardson varieties. Another consequence of the curious Lefschetz property is that $\Poinc(\Pio_f;q,t)$ is $q,t$-unimodal. The question of whether $\Poinc(\Rich_v^w;q,t)$ is $q,t$-unimodal for arbitrary $v\leq w\in W$ remains open.

\subsection{Proof of Theorem~\ref{thm:master_kn} from Theorem~\ref{thm:main}}\label{ssec:introproof} 
Let $G=\PGL_n(\C)$ and $f=\fvw\in\Bknc$, thus, we have $\Rich_v^w=\Pio_f$. We set
\begin{equation*}
  \lvw=\ell(w)-\ell(v)=\dim(\Rich_v^w)=\dim(\Pio_f),\quad\df:=\dim(\Pit_f)=\lvw-n+1.  
\end{equation*}
 Since $\ncyc(f)=1$, we have $\df=2\chi(\beta_f)$ by~\eqref{eq:chi_dfn}, and the $T$-action on $\Pio_f$ is free by \cref{prop:free}. In this case, we have
\begin{equation}\label{eq:Pioquotient}
H^{k+n-1,(p,p)}_{T,c}(\Pio_f)\cong H^{k,(p,p)}_c(\Pit_f)\cong H^{2\df-k,(\df-p,\df-p)}(\Pit_f), %
\end{equation}
where %
the first isomorphism is Lemma~\ref{lem:HTfree} below, and the second isomorphism is the \Poincare duality~\eqref{eq:Poincare}. %

Therefore~\eqref{eq:grading:intro} yields 
\begin{equation}\label{eq:main_KR}
\Potf q,t)=   \left(\qq\tt\right)^{\chi(\beta_f)} \PK(\betah_f;q,t).
\end{equation}
This is a special case of~\eqref{eq:PKRtop=PKRC}, which is proven in a similar way in \cref{sec:ordinary-cohomology}. In the case $f=\fkn$ and $\gcd(k,n)=1$, \cref{rmk:torus_knot} implies that $\betah_f$ is a torus knot, for which the right-hand side of~\eqref{eq:main_KR} was shown by Mellit~\cite{Mellit_torus} to coincide with $\Cat_{k,n-k}(q,t)$. Thus, \cref{thm:master_kn} follows from \cref{thm:main}. \qed

\subsection{Mixed Hodge structures of open Richardson varieties}
\begin{theorem}\label{thm:mixed_Tate_type}
  For any $v\leq w\in W$, the cohomology $H^\bul(\Rich_v^w,\C)$, the compactly supported cohomology $H^\bul_c(\Rich_v^w,\C)$, and the compactly supported $T$-equivariant cohomology $H^\bul_{T,c}(\Rich_v^w,\C)$ of $\Rich_v^w$ are of Hodge--Tate type.
\end{theorem}
We have omitted equivariant cohomology $H^\bul_{T}(\Rich_v^w,\C)$ from Theorem~\ref{thm:mixed_Tate_type} because the statement of \Poincare duality in the equivariant setting is considerably more complicated than for ordinary cohomology. 

\begin{lemma}\label{lem:GysinHT}
 Let $Y$ be a complex algebraic variety, $U \subset Y$ an open subvariety, and $Z := Y  \setminus U$.  Suppose that the compactly supported
cohomologies of $U$ and $Z$ are of Hodge--Tate type.  Then the same is true
for $Y$.

The same statement holds for compactly supported $T$-equivariant cohomology with the assumption that $U, Z$ are $T$-stable.
\end{lemma}
\begin{proof}
We have a Gysin long exact sequence for the triple $(Y,Z,U)$ (see for example \cite[(B-15)]{PS})
\begin{equation}\label{eq:Gysin}
\cdots \to H^k_c(U,\C) \to H^k_c(Y,\C) \to H^k_c(Z,\C) \to H^{k+1}_c(U,\C) \to\cdots \;.
\end{equation}
the maps of which respect the mixed Hodge structure~\cite[Theorem~4.1]{dCM}.

Taking the $(p,q)$ piece of the Deligne splitting, we have
$$
\cdots \to H^{k,(p,q)}_c(U,\C) \to H^{k,(p,q)}_c(Y,\C) \to H^{k,(p,q)}_c(Z,\C) \to\cdots \;. 
$$
By assumption when $p \neq q$, we have $H^{k,(p,q)}_c(U,\C)  =  0 = H^{k,(p,q)}_c(Z,\C)$.  Thus, $H^{k,(p,q)}_c(Y,\C)=0$.  The same proof applies in compactly supported equivariant cohomology.
\end{proof}

\begin{proof}[Proof of Theorem~\ref{thm:mixed_Tate_type}]
Since $\Rich_v^w$ is smooth, by \eqref{eq:Poincare} it suffices to show that the compactly supported cohomology, and the compactly supported equivariant cohomology are of Hodge--Tate type.

We will prove the statement by induction on $\lvw= \dim(\Rich_v^w)$.  The statement clearly holds if $\lvw = 0$, for then $\Rich_v^w$ is a point.  We will use a recursion for the varieties $\Rich_v^w$ from \cite{RSW}.   
By \cite[Lemma 4.3.1 and Proof of Proposition 4.3.6]{RSW},
 for any open Richardson $Y = \Rich_v^w$ with $w > v$, we can find a decomposition $Y = U \sqcup Z$, where $U \subset Y$ is open and $Z \subset Y$ is closed, and we have isomorphisms
\begin{equation}
\label{eq:Riso}
Z \simeq Y' \times \C \qquad U \simeq Y'' \times \C^\ast
\end{equation}
where $Y'$, $Y''$ are open Richardson varieties of lower dimension, and $Z$ is possibly empty.  Furthermore, $U,Z$ are $T$-stable and the isomorphisms \eqref{eq:Riso} are torus-equivariant, for certain linear actions of $T$ on $\C, \C^\ast$.  By the K\"unneth formula, the Hodge--Tate type property is preserved under products.  By the inductive hypothesis, the compactly supported (equivariant) cohomology of $U$ and of $Z$ (when non-empty) are therefore of Hodge--Tate type.  It follows that the same statement holds for $Y$.
\end{proof}

The following corollary of Theorem~\ref{thm:mixed_Tate_type} also follows from combining \cite[Theorem~8.3]{LS} and~\cite{GL}.\begin{corollary}
 For all $f\in\Bkn$, the cohomology of $\Pio_f$ is of Hodge--Tate type.  If $\ncyc(\fmod)=1$ then the cohomology of $\Pio_f/T$ is also of Hodge--Tate type. 
\end{corollary}
\begin{proof}
For the second statement, we note that $T$ acts freely on $\Pio_f$ by \cref{cor:Rich_T_action}.  Thus, $H^\bul_T(\Pio_f,\C) \simeq H^\bul(\Pio_f/T,\C)$ is of Hodge--Tate type by Theorem~\ref{thm:mixed_Tate_type}.
\end{proof}

When computing examples in the next section, we shall repeatedly use the following result.
\begin{lemma}\label{lem:HTfree}
Suppose that $T$ acts freely on a complex algebraic variety $Y$, and let $\dimT := \dim_\C(T)$.  Then we have an isomorphism of mixed Hodge structures
\begin{equation*}
  H^{k+\dimT,(p,q)}_{T,c}(Y,\C)\cong   H^{k,(p,q)}_{c}(Y/T,\C) \quad\text{for all $k,p,q\in\Z$,}
\end{equation*}
and the action of $H^\bul_T(\pt,\C)$ on $H^\bul_{T,c}(Y,\C)$ is trivial (that is, factors through the map $H^\bul_T(\pt,\C) \to H^0_T(\pt,\C) \simeq \C$).
\end{lemma}
\begin{proof}
Suppose $T = (\C^\ast)^\dimT$ and let $E_m := (\C^m \setminus \{0\})^\dimT$ be a finite-dimensional approximation to a contractible space $ET$ where $T$ acts freely.  Then by definition
\begin{align*}
H^k_{T,c}(Y)
&= \lim_{m\to\infty} H^k_c(Y \times_T E_m) \\
&= \lim_{m\to\infty} H^k_c(Y/T \times E_m) \\
&= \lim_{m\to\infty} \bigoplus_j H^{k-j}_c(Y/T) \otimes H^{j}_c(E_m) & \mbox{by the K\"unneth formula} \\ 
&= H^{k-\dimT}_c(Y/T).
\end{align*}
We have used the fact that $H^\bul_c(\C^m \setminus \{0\})$ is $\C$ in dimensions 1 and $2m$ and vanishes in other dimensions.  Since $H^{1}_c(\C^m \setminus \{0\})$ is of type $(0,0)$, the isomorphism is compatible with mixed Hodge structures.
\end{proof}

\subsection{Examples}\label{sec:Htc_examples}
We compute the compactly supported torus-equivariant cohomology of some open Richardson varieties, corresponding to the examples computed in \cref{sec:Soergel_examples}. As in \cref{sec:Soergel_examples}, we label each example by the link associated to the Richardson variety. Using \cref{tab:examples}, one can compare the below examples with the ones in \cref{sec:Soergel_examples} and observe that the computations agree with our main results, \cref{thm:ordinary,thm:main}.

Recall that we say that the action of $R=\Symh$ is \emph{trivial} if $\h^\ast$ acts by zero. We identify $R$ with $H^\bul_T(\pt,\C)$.
\begin{example}[Unknot-I]
Let $n=1$ and $v=w=\id$, and thus $\lvw=0$. Then $\Rich_v^w\cong T\cong \pt$, and the $T$-action is free. The only nonzero terms in the cohomology are
\begin{equation}\label{eq:H(pt)}
  H^{0,(0,0)}(\pt)\cong   H^{0,(0,0)}_c(\pt)\cong   H^{0,(0,0)}_{T,c}(\pt)\cong \C.
\end{equation}
As in \cref{ex:soe_unkn_1}, the $R=\C$-module structure on $H^\bul_{T,c}(\pt)$ is trivial (since $T$ acts freely).
\end{example}

\begin{example}[Unknot-II]
Let $n=2$, $v=\id$, and $w=s_1$, and thus $\lvw=1$. Then $\Rich_v^w\cong \Pio_{1,2}\subset\Gr(1,2)$ and we have $\Rich_v^w\cong T\cong \C^\ast$. The $T$-action is free and $\Rich_v^w/T\cong \pt$. Using \Poincare duality~\eqref{eq:Poincare} and \cref{lem:HTfree}, we find that the only nonzero terms are
\begin{equation*}
  H^{0,(0,0)}(\C^\ast)\cong   H^{1,(1,1)}(\C^\ast)\cong \C,\quad   H_c^{2,(1,1)}(\C^\ast)\cong   H_c^{1,(0,0)}(\C^\ast)\cong \C, \quad\text{and}
\end{equation*}
\begin{equation}\label{eq:Htc_Unkn2}
  H^{1,(0,0)}_{\C^\ast,c}(\C^\ast)\cong \C.
\end{equation}
As in \cref{ex:soe_unkn_2}, the $R=\C[y_1]$-module structure on $H^\bul_{T,c}(\Rich_v^w)$ is trivial (since $T$ acts freely).
\end{example}

\begin{example}[$2$-component unlink]
Let $n=2$, $v=s_1$, and $w=s_1$, and thus $\lvw=0$. We have $\Rich_v^w\cong\pt$ and $T\cong\C^\ast$. The $T$-action is not free. We have already computed $H^\bul(\pt)$ and $H^\bul_{T,c}(\pt)$ in~\eqref{eq:H(pt)}. The nonzero terms of $H^\bul_{T,c}(\Rich_v^w)\cong H^\bul_{\C^\ast,c}(\pt)$ are given by
\begin{equation*}
  H_{T,c}^{2p,(p,p)}(\pt)\cong \C \quad\text{for $p=0,1,2,\dots$\;.}
\end{equation*}
As in \cref{ex:soe_unlk}, we have $H^\bul_{T,c}(\Rich_v^w)\cong H^\bul_{T,c}(\pt) \cong R$ as an $R$-module.
\end{example}

\begin{example}[Hopf link]\label{example:hard}
Let $n=4$, $v=\id$, $w=f_{2,4}=s_2s_1s_3s_2$, and thus $\lvw=4$. We have $Y:=\Rich_v^w \cong \Pio_{2,4} \subset \Gr(2,4)$, an open positroid variety of dimension $4$. It is isomorphic to 
\begin{equation*}
 Y\cong  \left\{ \begin{pmatrix}
      1 & 0 & a & b\\
      0 & 1 & c & d
    \end{pmatrix}\middle| (a,b,c,d)\in\C^4: a\neq 0, d\neq 0,ad-bc\neq 0 \right\}.
\end{equation*}
 The action of $T\cong(\C^\ast)^3$ on $Y$ is not free: $T$ acts by rescaling rows and columns in such a way that the first two columns form the identity matrix. There exists a two-dimensional torus $T' \subset T$ that acts freely on $Y$: for example, one can always rescale columns $3$ and $4$ uniquely (since $a,d\neq 0$) to force the minors $\Delta_{2,3}$ and $\Delta_{1,4}$ to be equal to $1$:
\begin{equation*}
 Y/T'\cong  \left\{ \begin{pmatrix}
      1 & 0 & -1 & -y\\
      0 & 1 & x & 1
    \end{pmatrix}\middle| (x,y)\in\C^2: xy\neq 1 \right\}.
\end{equation*}
The quotient $Y/T'$ can be identified with the two-dimensional $A_1$-cluster variety (with one frozen variable; cf. \cite[Section~6.1]{LS}). We denote it by 
$$
U:=Y/T' \simeq \{(x,y) \in \C^2 \mid xy \neq 1\} \subset \C^2.
$$
The action of $T/T'$ can be identified with the action of $\C^\ast$ on $U$ with $\lambda \cdot(x,y) = (\lambda x, \lambda^{-1} y)$ for $\lambda \in \C^\ast$ and $(x,y)\in \C^2$. Forgetting this torus action, it was shown in~\cite[Corollary~7.2]{LS} that $U$ is homotopy equivalent\footnote{We caution the reader that neither the compactly supported cohomology nor the mixed Hodge structure are preserved by homotopy equivalences.} to a pinched torus:

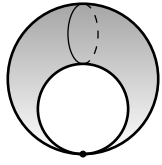
\tikzset{
    partial ellipse/.style args={#1:#2:#3}{
        insert path={+ (#1:#3) arc (#1:#2:#3)}
    }
}
\def\lw{1pt}
\def\lww{0.5}
\def\wid{0.4}
\begin{center}
\begin{tikzpicture}[scale=1]
 \draw[shade, even odd rule, opacity=0.7] (0,0) circle (1) (0,-\wid) circle (1-\wid);
\draw[line width=\lw] (0,0) circle (1);
\draw[line width=\lw] (0,-\wid) circle (1-\wid);
\draw[line width=\lww] (0,1-\wid) [partial ellipse=90:270:0.2 and 0.4];
\draw[line width=\lww,dashed] (0,1-\wid) [partial ellipse=-90:90:0.2 and 0.4];
\node[scale=0.2,draw,circle,fill=black](A) at (0,-1) {};
\end{tikzpicture}
\end{center}
Therefore the Betti numbers of $U$ are $(1,1,1)$, and moreover we have~\cite[Section~6.1]{LS}
\begin{equation*}
  H^{0,(0,0)}(U)\cong   H^{1,(1,1)}(U)\cong   H^{2,(2,2)}(U)\cong \C.
\end{equation*}
We have $Y\cong U\times (\C^\ast)^2$, which corresponds to multiplying the mixed Hodge polynomial of $U$ by $\left(\qq+\tt\right)^2$. The resulting mixed Hodge table of $Y$, whose sole row contains the coefficients of the polynomial $\left(q+\qq\tt+t\right)\cdot \left(\qq+\tt\right)^2$, is given in \tabref{tab:MHT_vs_HHHC}(top left).

Let us return to computing the $\C^\ast$-equivariant cohomology of $U$. Denote $W:= \C^2$ and let $Z := W \setminus U$ be the hyperbola $\{(x,y) \mid xy = 1\}$.  We first compute the compactly supported equivariant cohomologies $H^\bul_{\C^\ast,c}(Y)$ and $H^\bul_{\C^\ast,c}(Z)$.  Set $R:=H^\bul_{\C^\ast}(\pt)$.

 First, suppose that $\C^\ast$ acts on $\C^m$ linearly in any way.  Then it follows directly from the definitions that 
$H^\bul_{\C^\ast,c}(\C^m)$ is a free $R$-module with generator in degree $2m$. Specifically, all nonzero terms are given by
\begin{equation}\label{eq:lem:Cm}
  H^{2(m+k),(m+k,m+k)}_{\C^\ast,c}(\C^m)\cong \C \quad\text{for $k=0,1,2,\dots$\;.}
\end{equation}

Next, observe that the variety $Z$ is $\C^\ast$-equivariantly isomorphic to the variety $\C^\ast$ on which $\C^\ast$ acts freely. By~\eqref{eq:Htc_Unkn2},
\begin{equation}\label{eq:lem:Z}
  H^{1,(0,0)}_{\C^\ast,c}(Z)\cong \C.
\end{equation}

Now we compute $H^\bul_{\C^\ast,c}(U)$.  The Gysin sequence \eqref{eq:Gysin} for the triple $(W,Z,U)$ gives
\begin{align*}
0 & \to H^0_{\C^\ast,c}(U) \to H^0_{\C^\ast,c}(W) \to H^0_{\C^\ast,c}(Z) \\
&\to H^1_{\C^\ast,c}(U) \to H^1_{\C^\ast,c}(W) \to H^1_{\C^\ast,c}(Z) \to\cdots\;.
\end{align*}
Applying~\eqref{eq:lem:Cm}--\eqref{eq:lem:Z}, we get
\begin{align*}
0 &\to  H^0_{\C^\ast,c}(U) \to 0 \to  0 \to  H^1_{\C^\ast,c}(U) \to 0 \to  \C \to  H^2_{\C^\ast,c}(U) \to 0 \to  0 \\
&\to  H^3_{\C^\ast,c}(U) \to 0 \to  0 \to  H^4_{\C^\ast,c}(U) \to \C \to  0 \to H^5_{\C^\ast,c}(U) \to 0 \to  0 \to\cdots\;.
\end{align*}
We conclude that the nonzero terms of $H^\bul_{\C^\ast,c}(U)$ are given by
\begin{equation*}
  H_{\C^\ast,c}^{2,(0,0)}(U)\cong \C \quad\text{and}\quad H_{\C^\ast,c}^{4+2k,(2+k,2+k)}(U)\cong \C \quad\text{for $k=0,1,2,\dots$\;.}
\end{equation*}
By the same computation as in the proof of \cref{lem:HTfree}, we have $H^{k+2,(p,q)}_{T,c}(Y)\cong H^{k,(p,q)}_{\C^\ast,c}(U)$.  Thus, the nonzero terms of $H^\bul_{T,c}(Y)$ are
\begin{equation}\label{eq:Hopf_Htc}
  H_{T,c}^{4,(0,0)}(Y)\cong \C \quad\text{and}\quad H_{T,c}^{6+2k,(2+k,2+k)}(Y)\cong \C \quad\text{for $k=0,1,2,\dots$\;.}
\end{equation}
Recall that $\lvw=4$. In view of~\eqref{eq:grading:intro}, the dimensions in~\eqref{eq:Hopf_Htc} match perfectly with those computed in~\eqref{eq:Hopf_HHB} from the Soergel bimodule perspective.
\begin{remark}
We observe that the $R=\C[y_1]$-module structure on $H_{T,c}^\bul(Y)$ also agrees with that computed in~\eqref{eq:Hopf_HHB}. More generally, for $W=S_n$, \cref{cor:Rich_T_action} can be extended to arbitrary open Richardson varieties $\Rich_v^w$: we have a subtorus $T'\subset T$ of dimension $n-\ncyc(\beta)$ acting freely on $\Rich_v^w$, and $H^\bul_{T,c}(\Rich_v^w)\cong H^\bul_{T/T',c}(\Rich_v^w/T')$. Recall from \cref{sec:link_cpts_vs_Koszul} that we may therefore view both sides of~\eqref{eq:main} as graded modules over a polynomial ring $R$ in $\ncyc(\beta)-1$ variables, and we expect that these $R$-modules are isomorphic under the grading change~\eqref{eq:grading:intro}.
\end{remark}
\end{example}

It would be interesting to combine Example~\ref{example:hard} with \cite{LS} to obtain a description of the equivariant cohomology of more general cluster varieties.

\begin{example}[Trefoil knot]
Let $n=5$, $v=\id$, $w=f_{2,5}=s_3s_2s_1s_4s_3s_2$, and thus $\lvw=6$. We have $\Rich_v^w\cong \Pio_{2,5}$, on which the torus $T\cong (\C^\ast)^5$ acts freely. As explained in \cref{rmk:necklace}, the quotient is obtained by fixing the cyclically consecutive maximal minors to $1$. An explicit parametrization can be chosen as follows:
\begin{align*}
  \Pio_{2,5}/T\cong   &\left\{ \begin{pmatrix}
      1 & 0 & -1 & -y& \frac{1+y}{1-xy}\\
      0 & 1 & x & xy-1 & 1
    \end{pmatrix}\middle| (x,y)\in\C^2: xy\neq 1 \right\}\\
& \qquad\qquad\qquad\qquad\qquad\qquad \sqcup \left\{ \begin{pmatrix}
      1 & 0 & -1 & 1& z\\
      0 & 1 & -1 & 0 & 1
    \end{pmatrix}\middle | z\in \C \right\}.
\end{align*}
Observe that the point count therefore equals $(q^2-q+1)+q=q^2+1=q\cdot \Cat_{2,3}(q,1/q)$. The variety $\Pio_{2,5}/T$ is a two-dimensional cluster variety of type $A_2$ with no frozen variables. Its cohomology was computed in \cite[Section~6.2]{LS}: the nonzero terms are
\begin{equation*}
  H^{0,(0,0)}(\Pio_{2,5}/T)\cong   H^{2,(2,2)}(\Pio_{2,5}/T)\cong \C.
\end{equation*}
Multiplying the mixed Hodge polynomial by $\left(\qq+\tt\right)^4$, we see that the mixed Hodge table of $\Pio_{2,5}$ is given in \tabref{tab:MHT_vs_HHHC}(top right). By~\eqref{eq:Poincare} and \cref{lem:HTfree}, we find
\begin{equation*}
  H^{2,(0,0)}_{c}(\Pio_{2,5})\cong   H^{4,(2,2)}_{c}(\Pio_{2,5})\cong\C \quad\text{and}\quad  H^{6,(0,0)}_{T,c}(\Pio_{2,5})\cong   H^{8,(2,2)}_{T,c}(\Pio_{2,5})\cong \C.
\end{equation*}
As in \cref{ex:soe_trefoil}, the $R$-module structure on $H^\bul_{T,c}(\Rich_v^w)$ is trivial (since $T$ acts freely).
\end{example}

We now give three examples of Richardson and positroid knots with non-vanishing odd cohomology, as promised in \cref{sec:notes:odd_cohom_van}. Here by a \emph{positroid knot} we mean a knot of the form $\betah_f$ for $f\in\Bknc$.
\begin{example}[Odd cohomology-I]\label{ex:odd_cohom1}
Let $n=5$, $v=s_3$, and $w=s_2 s_3 s_4 s_3 s_1 s_2 s_1$, and thus $\lvw=6$ and $\ncyc(wv^{-1})=1$. The Richardson knot $\bvw$ is the \emph{$3$-twist knot}, listed as $5_2$ in Rolfsen's table~\cite{Rolfsen}. By~\eqref{eq:homfly_vw}, the point count is given by
\begin{equation*}%
\#(\Rich_v^w/T)(\F_q)=q^2-q+1.
\end{equation*}
The appearance of $-q$ implies that the cohomology of $\Rich_v^w/T$ \emph{cannot} be concentrated in even degrees.
\end{example}

The following two examples of \emph{positroid} knots were discovered by David Speyer jointly with the second author.
\begin{example}[Odd cohomology-II]\label{ex:odd_cohom2}
Let $k=7$, $n=14$, and let $f=\fvw\in\Bknc$ and $v\leq w\in S_n$ be given by 
\begin{align*}
f&=[3, 8, 9, 16, 7, 14, 15, 20, 12, 18, 13, 24, 19, 25],\\
v&=(2, 1, 8, 4, 6, 3, 11, 9, 10, 5, 7, 13, 14, 12),\\
w&=(8, 9, 10, 1, 11, 12, 2, 3, 13, 4, 14, 5, 6, 7).
\end{align*}
Here $v$ and $w$ are given in one-line notation, and $f=[f(1),f(2),\dots,f(n)]$ is given in \emph{window notation}. The permutation $w$ is $7$-Grassmannian, and we have $\ncyc(\fmod)=1$, $\ell(v)=19$, $\ell(w)=40$, and $\df=8$. The mixed Hodge table of $\Pit_f$ is given by
\begin{center}
\begin{tabular}{|c|ccccccccc|}\hline
$H^k$ & $H^0$ & $H^1$ & $H^2$ & $H^3$ & $H^4$ & $H^5$ & $H^6$ &$H^7$ & $H^8$ \\\hline
$k-p=0$&  $1$ &  $0$ &  $1$ &  $0$ &  $1$ &  $0$ &  $1$ &  $0$ &  $1$ \\
$k-p=1$&      &      &      &      &  $1$ &  $1$ &  $1$ & & \\\hline
\end{tabular}
\end{center}
In particular, $H^5$ is non-vanishing.
\end{example}

\begin{example}[Odd cohomology-III]\label{ex:odd_cohom3}
Let $k=7$, $n=14$, and let $f=\fvw\in\Bknc$ and $v\leq w\in S_n$ be given by 
\begin{align*}
f&=[7, 4, 15, 13, 11, 8, 19, 16, 14, 12, 23, 20, 17, 24],\\
v&=(1, 3, 6, 9, 2, 5, 8, 12, 4, 7, 11, 14, 10, 13),\\
w&=(8, 9, 1, 10, 11, 12, 2, 3, 13, 14, 4, 5, 6, 7).
\end{align*}
Similarly, $w$ is $7$-Grassmannian, $\ncyc(\fmod)=1$, $\ell(v)=19$, $\ell(w)=40$, and $\df=8$.
One can easily compute (e.g. using \cref{thm:homfly}) that the point count is given by
\begin{equation*}%
  \#\Pit_f(\F_q)=q^8+q^6+3q^5-q^4+3q^3+q^2+1.
\end{equation*}
Similarly to \cref{ex:odd_cohom1}, this polynomial has a negative term, and therefore the odd cohomology does not vanish.
\end{example}

\begin{remark}\label{rmk:alg}
Ivan Cherednik has suggested to us that one might expect odd cohomology vanishing for algebraic knots; see  \cite{Piontkowski} and \cite[Section~3.4]{Cherednik_Grob}. 
The knot $5_2$ in \cref{ex:odd_cohom1} is not algebraic. We thank the anonymous referee for pointing out to us that the positroid knot in \cref{ex:odd_cohom2} is also not algebraic.
\end{remark}

\section{Equivariant derived categories of flag varieties}\label{sec:sheaf-cohomology}
Rather than working with mixed Hodge structures, our proof will be mostly stated in the language of weights~\cite{DelWeilII} on \etale cohomology.\footnote{As pointed out to us by Wolfgang Soergel, our results could also be formulated using the language of equivariant mixed Tate motives~\cite{SVW}.} We assume that the reader is familiar with derived categories of $\ell$-adic sheaves in the equivariant setting; see e.g.~\cite{BBD,BL,LO} for relevant background. 

\subsection{Conventions}
\def\Yqb{Y_{\Fqb}}
Fix a prime power $q$.  We shall consider schemes over the finite field $\F_q$ and its algebraic closure $\Fqb$.  For an $\F_q$-scheme $Y$, let $ Y_{\Fqb}:=Y\times_{\Spec(\F_q)} \Spec(\Fqb)$ denote the base change to an $\Fqb$-scheme.  We have the Frobenius automorphism $\Fr: \Yqb \to \Yqb$ whose fixed points are exactly the points of $\Yqb$ defined over $\F_q$. 

Fix a prime number $\ell$ different from the characteristic of $\F_q$ and let $\k:=\Qlb$.  Let $H$ be an algebraic group acting on an $\F_q$-variety $Y$.  We consider the bounded derived category $\DbcH(Y,\k)$  of mixed $H$-constructible (that is, constructible along $H$-orbits) $\k$-sheaves on $Y$, as well as the corresponding category $\DbcH(\Yqb,\k)$ of $H$-constructible $\k$-sheaves on $\Yqb$; see~\cite{BBD}. We also let $\DBH$ and $\DCH$ denote the corresponding $H$-equivariant bounded derived categories as in~\cite{BL} (see also~\cite{WW_below} for a discussion of mixed derived categories in the equivariant setting).   There are functors $\For: \DCH \to \DbcH(\Yqb,\k)$ and $\For: \DBH \to \DbcH(Y,\k)$ forgetting the equivariant structure.  There are also functors $\omega: \DbcH(Y,\k) \to \DbcH(\Yqb,\k)$ and $\omega:\DBH \to \DCH$ obtained by extension of scalars from $\Fq$ to $\Fqb$.  

The language of algebraic stacks \cite{LO} allows us to switch between ordinary and equivariant derived categories at our convenience.  For example, $\Db([Y/H],\k) \simeq \Db_H(Y,\k)$.

We denote by $[m]$ the cohomological shift $m$ steps to the left in a derived category as in \cref{sec:rouquier-complexes}.  For $\Fcal,\Gcal \in \DBH$  and $k \in \Z$, let 
\begin{equation}
\label{eq:Extdef}
  \ExtDC^k(\Fcal,\Gcal) = \Ext^k_{Y}(\Fcal,\Gcal):=\Hom_{\DCH}(\omega\Fcal,\omega\Gcal[k]) \quad\text{for $k\in\Z$}.
\end{equation}
The space $\Hom_{\DCH}(\omega\Fcal,\omega\Gcal[k])$ has a natural action of the Frobenius $\Fr$.
Therefore $\ExtDC^\bul(\Fcal,\Gcal)$ is a graded $H^\bul_H(\pt,\k)$-module equipped with an action of $\Fr$, or in other words, an $(H^\bul_H(\pt,\k),\Fr)$-module.  The actual extension groups in $\DBH$ are denoted by $\ext^k(\Fcal,\Gcal)$, and are related to $\ExtDC^k(\Fcal,\Gcal)$ by the exact sequence \cite[(5.1.2.5)]{BBD} 
\begin{equation}\label{eq:extExt}
0 \to \ExtDC^{i-1}(\Fcal,\Gcal)_{\Fr} \to \ext^i(\Fcal,\Gcal) \to \ExtDC^i(\Fcal,\Gcal)^\Fr \to 0,
\end{equation}
where  $\ExtDC^i(\Fcal,\Gcal)^\Fr$ and $\ExtDC^{i-1}(\Fcal,\Gcal)_{\Fr}$ denote Frobenius invariants an coinvariants, respectively.
We denote by $\hom(\Fcal,\Gcal)= \ext^0(\Fcal,\Gcal) = \Hom_{\DBH}(\Fcal,\Gcal)$ the Hom groups in $\DBH$.

We fix an isomorphism $\k\cong \C$ and denote by $|\la|$ the norm of $\la\in\k$ considered as an element of $\C$.  If $M$ is an $\Fr$-module, then the {\it weights} of $\Fr$ on $M$ are the real numbers $2 \log(\la)/\log(q)$ for $\la$ an eigenvalue of the action of $\Fr$.  All weights we consider will be integers: the cohomology sheaves of an object $\Fcal \in \DbcH(Y,\k)$ are required to have punctual integer weights; see \cite[5.1.5]{BBD}.
We fix a square root $(1/2)$ of the Tate twist, and for $\Fcal\in\DBH$ and $\wt\in\Z$, we denote by $\Fcal(\wt/2)$ the corresponding Tate twist of $\Fcal$.  

Recall \cite[5.1]{BBD} that $\Fcal \in \DBH$ has weights $\leq \wt$ if for each $i$ the sheaf $H^i(\Fcal)$ has mixed punctual weights $\leq \wt+i$.  We say that $\Fcal \in \DBH$ has weights $\geq \wt$ if the Verdier dual $\VERD\Fcal$ has mixed punctual weights $\leq -\wt$.  Finally, we say that $\Fcal \in \DBH$ is {\it pure of weight $\wt$} if it has weights $\leq \wt$ and weights $\geq \wt$.  If $\Fcal$ is pure of weight $\wt$ then $\Fcal[1]$ is pure of weight $\wt+1$, while $\Fcal(1/2)$ is pure of weight $\wt -1$.

For an integer $\wt \in \Z$, we denote by $\ExtDC^{k,(\wt/2)}(\Fcal,\Gcal) \subset \ExtDC^{k}(\Fcal,\Gcal)$ the generalized eigenspace for $\Fr$ of weight $\wt$. Thus, 
\begin{equation}\label{eq:Ext_bigr_dfn}
  \ExtDC(\Fcal,\Gcal)=\bigoplus_{k,\wt\in\Z} \ExtDC^{k,(\wt/2)}(\Fcal,\Gcal).
\end{equation}
For all $\Fcal,\Gcal\in\DBH$ and $k,k',\wt,\wt'\in\Z$, we have
\begin{equation}\label{eq:Ext_shifts}
  \begin{aligned}
  \ExtDC^{k,(\wt/2)}(\Fcal[-k'](-\wt'/2),\Gcal)&\cong \ExtDC^{k,(\wt/2)}(\Fcal,\Gcal[k'](\wt'/2))\\ &\cong \ExtDC^{k+k',((\wt+\wt')/2)}(\Fcal,\Gcal).
\end{aligned}
\end{equation}

\subsection{Equivariant cohomology}
For $\Fcal \in \DBH$, the equivariant hypercohomology $\H^\bullet_H(\Fcal) = \H^\bullet_H(Y,\Fcal)$ is defined by
\begin{equation}\label{eq:HBExt}
\H^\bullet_H(\Fcal) := \ExtDC^\bullet(\ku_Y, \Fcal).
\end{equation}
In particular, we have 
$
\H_H^{k,(\wt/2)}(\Fcal[k'](\wt'/2)) = \H_H^{k+k',((\wt+\wt')/2)}(\Fcal).
$

Let $\pi:Y \to \pt:=\Spec(\F_q)$ be the projection to a point.  By definition, the $H$-equivariant cohomology $H^\bullet_H(Y,\k)$ and the compactly supported $H$-equivariant cohomology $H^\bullet_{H,c}(Y,\k)$ of $Y$ is given by
\begin{equation}\label{eq:equivcompact}
H^\bullet_H(Y,\k) := \H^\bullet_H(Y, \ku_Y) = \H^\bullet_H(\pt, \pi_* \ku_Y) \quad \text{and} \quad H^\bullet_{H,c}(Y,\k) :=   \H^\bullet_H(\pt, \pi_! \ku_Y).
\end{equation}
Both $H^\bullet_H(Y,\k)$ and $H^\bullet_{H,c}(Y,\k)$ are graded $(H^\bul_H(\pt,\k), \Fr)$-modules.

\subsection{Flag varieties}
We fix a semisimple algebraic group $G$, split over $\F_q$, and a maximal torus and a Borel subgroup $T \subset B\subset G$.  Let $\Xq=G/B$ be the flag variety over $\F_q$ and let
 $\Xqb:=\Xq\times_{\Spec(\F_q)} \Spec(\Fqb)$ be obtained by extending scalars. 
The variety $\Xq$ is stratified by $B$-orbits $\Xw:=BwB/B$ (known as \emph{Schubert cells}):
\begin{equation*}
  \Xq=\bigsqcup_{w\in W} \Xqw.
\end{equation*}
Let $R = H^\bul_B(\pt,\k) \simeq \k[\h]$.

\begin{remark}\label{rmk:kC}
We switch from working over $R=\Symh$ to $R=\k[\h]$. The results in \cref{sec:KR_hom} do not depend on the field as long as it is of characteristic zero. Therefore on the Soergel bimodule side, one can freely switch between working over $\C$ and over $\k$.
\end{remark}

For $w\in W$, we let $i_w:\Xw\to \Xq$ be the inclusion map. Introduce the \emph{standard} and \emph{costandard sheaves} $\Delta_w,\Nabla_w\in\DB$ defined by
\begin{equation*}
  \Delta_w:=i_{w,!} \kuqw[\ell(w)](\ell(w)/2) \quad\text{and}\quad  \Nabla_w:=i_{w,*} \kuqw[\ell(w)](\ell(w)/2).
\end{equation*}
Here, $\kuqw$ denotes the constant sheaf on $\Xqw$. The \emph{intersection cohomology sheaves} $\IC_w\in\DB$ are defined using the intermediate extension functor $i_{w,!*}$:
\begin{equation*}
  \IC_w:=i_{w,!*}\kuqw[\ell(w)](\ell(w)/2).
\end{equation*}

Since we have $\DB \simeq \Db_{B \times B}(G,\k)$, the equivariant cohomology $\H^\bul_B(\Fcal)$ is a graded $(R \otimes R,\Fr)$-module.   Furthermore, there is a restriction functor ${\Res_{T,B}: \Db_B(\Xq,\k) \to \Db_T(\Xq,\k)}$ (see~\cite{BL}), and we also have the hypercohomology functor ${\H_T^\bul: \Db_T(\Xq,\k) \to H^\bul_T(\pt) - \mod}$.
It is well known that $H^\bul_B(\pt,\k)  \simeq R \simeq H^\bul_T(\pt,\k)$, and furthermore we have
\begin{equation}\label{eq:ResTB}
  \Res_{T,B}: \Db_B(\pt,\k)\cong \Db_T(\pt,\k) \quad\text{and}\quad \H^\bul_B(\Xq,\Fcal)\cong \H^\bul_T(\Xq,\Res_{T,B}\Fcal)
\end{equation}
for $\Fcal \in \Db_B(\Xq,\k)$.

\subsection{Equivariant cohomology of open Richardson varieties}
We split the proof of our main result into two parts. We will focus on the equivariant case (\cref{thm:main}).  The proof of its non-equivariant version (\cref{thm:ordinary}) will follow as a byproduct, and will be discussed in \cref{sec:ordinary-cohomology}.

Recall from \cref{sec:main-result-equivariant} that our goal is to show the following result.
\begin{theorem}\label{thm:iso_RR_Fr}
For all $v\leq w\in W$, we have an isomorphism of bigraded $R$-modules
\begin{equation*}
\HHH^0(\FR_{v,w}) \cong H^{\bullet}_{T,c}(\Rich_v^w).
\end{equation*}
For all $k,p\in\Z$, it restricts to an isomorphism
\begin{equation*}
\HHX {k}p_{v,w}  \cong  H^{\lvw+2p+k,(p,p)}_{T,c}(\Rich_v^w)
\end{equation*}
of vector spaces.
\end{theorem}

We will accomplish this in two steps. The first one is an equivariant version of~\cite[Proposition~4.2.1]{RSW}.
\begin{proposition}\label{prop:RSW}
For all $v\leq w\in W$, we have an isomorphism of bigraded $R$-modules
\begin{equation*}
  H^{\bullet}_{T,c}(\Rich_v^w) \cong \ExtDC^\bul(\Delta_v,\Delta_w).
\end{equation*}
For all $m,\wt\in\Z$, it restricts to an isomorphism
\begin{equation*}
H^{\lvw+m,(\wt/2,\wt/2)}_{T,c}(\Rich_v^w)\cong \ExtDC^{m,((\wt-\lvw)/2)}(\Delta_v,\Delta_w)
\end{equation*}
of vector spaces. (In particular, both sides are zero for odd $\wt$.)
\end{proposition}

The second step passes through the \emph{mixed equivariant derived category} of~\cite{AR1,AR2} and involves the \emph{degrading functor} of~\cite{Bei} (see also~\cite{Rid}).

\begin{proposition}\label{prop:degr}
For all $v\leq w\in W$, we have an isomorphism of bigraded $R$-modules
\begin{equation*}
\HHH^0(\FR_{v,w}) \cong \ExtDC^\bul(\Delta_v,\Delta_w).
\end{equation*}
For all $k,\wt\in\Z$, it restricts to an isomorphism
\begin{equation} \label{eq:HHExt}
\HHX {k}{\wt/2}_{v,w}  \cong  \ExtDC^{\wt+k,((\wt-\lvw)/2)}(\Delta_v,\Delta_w)
\end{equation}
of vector spaces. (In particular, both sides are zero for odd $\wt$.)
\end{proposition}

\subsection{From weights to the mixed Hodge numbers}
We briefly explain the standard relation between the mixed Hodge structure of the complex variety $Y_\C = (\Rich_v^w)_{\C}$ and the \etale cohomology of the variety $Y_{\Fqb} = (\Rich_v^w)_{\Fqb}$.  First, by the comparison theorem (\cite[Theorem~21.1]{Milne}) between Betti cohomology and \etale cohomology of $Y_\C$, we have an isomorphism $H^k_{ Betti}(Y_\C,\k) \simeq H^k_{et}(Y_\C,\k)$ which preserves the weight filtration of both sides.  Next, we find a discrete valuation ring $S \subset \C$ with residue field $\Fqb$ and construct the Richardson variety $Y_S$ over $S$.  Then $Y_\C$ and $Y_{\Fqb}$ are obtained from $Y_S$ via base change, and we obtain isomorphisms between the \etale cohomologies $H^\bul(Y_\C,\k) \simeq H^\bul(Y_S,\k) \simeq H^\bul(Y_{\Fqb},\k)$, compatibly with the weight filtrations; see \cite[Section 20]{Milne}.  For $H^\bul(Y_{\Fqb},\k)$, the weight filtration is obtained by taking sums of generalized eigenspaces of the Frobenius $\Fr$.  Finally, the cohomology of $Y_\C$ is of Hodge--Tate type, so the weight filtration is simply given by $W^{2r}(H^k(Y_\C,\C)) = \bigoplus_{p \leq r} H^{k,(p,p)}(Y_\C,\C)$.  (All these statements hold also equivariantly.) To sum up:

\begin{proposition}\label{prop:kC}
For all $v\leq w\in W$ and $k,p\in\Z$, we have the equalities 
\begin{align*}
\dim_\C H^{k,(p,p)}((\Rich_v^w)_{\C},\C) &= \dim_\k H^{k,(p)}((\Rich_v^w)_{\Fqb},\k), \\
\dim_\C H^{k,(p,p)}_{T,c}((\Rich_v^w)_{\C},\C) &= \dim_\k H^{k,(p)}_{T,c}((\Rich_v^w)_{\Fqb},\k)
\end{align*}
where $H^{k,(p)}_{T,c}((\Rich_v^w)_{\Fqb},\k)=\H^{k,(p)}_T(\pt,\pi_!\ku_{\Rich_v^w})=\Ext^{k,(p)}(\ku_{\pt},\pi_!\ku_{\Rich_v^w})$ as in~\eqref{eq:equivcompact}.
\end{proposition}

See also \cite[Remark 7.1.4]{RSW} where a comparison between derived categories of flag varieties is given.

\section{Proof of Proposition~\ref{prop:RSW}}\label{sec:proof-RSW}
\def\Gm{{\mathbb G_m}}
\def\A{{\mathbb A}}

We follow the steps in the proof of~\cite[Proposition~4.2.1]{RSW}.  Using~\eqref{eq:Ext_shifts} and the adjunction $(i_{v,!},i_v^!)$, we find
\begin{align*}
\ExtDC^{m,((\wt-\lvw)/2)}(\Delta_v,\Delta_w)&\cong \ExtDC^{m,((\wt-\lvw)/2)}(i_{v,!}\kuv,i_{w,!}\kuw[\lvw](\lvw/2))\\
&\cong \ExtDC^{m+\lvw,(\wt/2)}(i_{v,!}\kuv,i_{w,!}\kuw) \\ &\cong \ExtDC^{m+\lvw,(\wt/2)}(\kuv,i_v^!i_{w,!}\kuw).
\end{align*}
Note that $i_v^!i_{w,!}\kuw\in \Db_B(\Xv,\k)$.  By \eqref{eq:HBExt} and \eqref{eq:ResTB}, we have 
$$\ExtDC^\bul(\kuv,i_v^!i_{w,!}\kuw) \cong \H^\bul_B(\Xv,i_v^!i_{w,!}\kuw) \cong \H^\bul_T(\Xv,i_v^!i_{w,!}\kuw).$$

We now switch to working with $T$-equivariant derived categories.  First, we state a $T$-equivariant version of~\cite[Proposition~1]{Soe89}.
Let $Z$ be a $T$-variety and let $q: X \hookrightarrow Z$ be an inclusion of a closed $T$-subvariety.  An action $\Gm \times Z \to Z$ \emph{contracts $Z$ to $X$} if there is a commutative diagram
$$\displaystyle 
\xymatrix{
\Gm \times Z \ar[r]^-{act} \ar[d]_-{j} & Z \ar@{=}[d] \\
\A^1 \times Z \ar[r]^-{act_0} & Z \\
Z \ar[u]^-{i} \ar[r]^-{p} & X \ar[u]_-{q}
}$$
where $i,j$ are the obvious maps and $p \circ q = \id$, and all arrows are $T$-equivariant, with $T$ acting trivially on $\Gm$ and $\A^1$.  Let $\pi: \Gm \times Z \to Z$ be the projection and suppose that $\Fcal \in \Db_T(Z)$ is $\Gm$-equivariant, i.e. satisfies $act^*(\Fcal) \simeq \pi^*(\Fcal) \in \Db_T(\Gm \times Z)$.  There is a natural morphism $q^! \to p_!$ of functors (obtained by composing the adjunction morphism $q_!q^!\to \id$ with $p_!$) and we have the following.
\begin{proposition}[{cf. \cite[Proposition~1]{Soe89}}] \label{prop:Soe}
The map $q^!\Fcal \to p_! \Fcal$ is an isomorphism.
\end{proposition}

Let $\amap:\Xv\to\pt$ and $b:\Rich_v^w\to\pt$ be the projections (cf. \figref{fig:RSW}(right)). Our next goal is to manipulate the object $i_v^!i_{w,!}\kuw$, in order to establish the following result.
\begin{lemma}\label{lemma:RSW}
We have $i_v^!i_{w,!}\kuw\cong a^*b_!\ku_{\Rich_v^w}$ in $\Db_T(\Xv)$. 
\end{lemma}
\begin{proof}
Recall that $\Xw:=(BwB)/B$. Denote $\Xvm:=(B_-vB)/B$, and thus $\Rich_v^w=\Xw\cap \Xvm$. Let $\Xwcl:=\bigsqcup_{u\leq w}\Xu$ be the closure of $\Xw$, and denote $\Rsemi_v^w:=\Xwcl\cap \Xvm$. A diagram of the various inclusions between the spaces $\Xw$, $\Xv$, $\Xwcl$, and $\Xq=G/B$ is given in \figref{fig:RSW}(left). With this notation, we have isomorphisms
\begin{equation*}
  i_v^!i_{w,!}\kuw\cong i_{v,w}^!\ibar_w^!\ibar_{w,!}j_{w,!} \kuw\cong i_{v,w}^!j_{w,!} \kuw.
\end{equation*}
The first isomorphism follows from the usual composition rules for sheaf operations; see e.g.~\cite[Section~1.4.2]{BL}. The second isomorphism follows from $\ibar_w^!\ibar_{w,!} \simeq \id$.

\begin{figure}
\setlength{\tabcolsep}{6pt}
\scalebox{0.85}{
\begin{tabular}{ccc}
$\displaystyle 
\xymatrix{
\Xw \ar@{^{(}->}[r]^-{j_w} \ar@{^{(}->}@/^1.5pc/[rr]^-{i_w} & \Xwcl \ar@{^{(}->}[r]^-{\ibar_w} & X \\
& \Xv \ar@{^{(}->}[u]^-{i_{v,w}} \ar@{^{(}->}[ur]_-{i_v} &
}$ & 

$\displaystyle 
\xymatrix{
\Xw \ar@{}[dr]|{\square} \ar@{^{(}->}[r]^-{j_w} & \Xwcl \\
\Rich_v^w \times \Xv \ar@{^{(}->}[u]^-{j'} \ar@{^{(}->}[r]^-{j_w'} & \Rsemi_v^w \times \Xv \ar@{^{(}->}[u]^-{j} \ar@<1ex>[d]^-{\projvw} \\
& \Xv \ar@{->}[u]^-{k} \ar@/_3pc/[uu]_-{i_{v,w}}
}$ &

$\displaystyle 
\xymatrix{
\Rich_v^w \times \Xv \ar[r]^-{\qmap} \ar[d]_-{\projvw \circ j_w'} \ar@{}[dr]|{\square} & \Rich_v^w \ar[d]^-{b} \\
\Xv \ar[r]^-{a} & \pt
}$
\end{tabular}
}
  \caption{\label{fig:RSW} Three commutative diagrams from~\cite{RSW}.}
\end{figure}
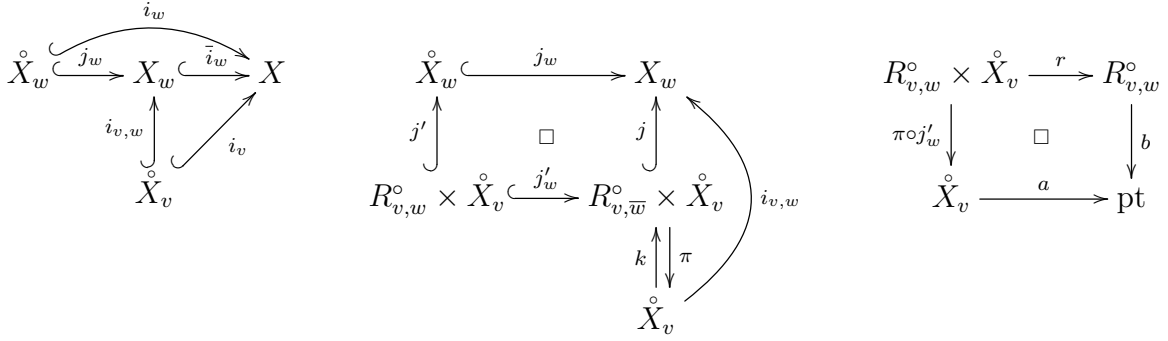

Consider the commutative diagram in \figref{fig:RSW}(middle). The map $k$ is given by $k(zB/B) = (vB/B,zB/B)$ and $\proj$ is the obvious projection map. The map $j'_w$ has two components: the first one is the inclusion $\Rich_v^w\hookrightarrow \Rsemi_v^w$, and the second one is the identity map $\Xv\xrasim\Xv$. The map $j_w:\Xw\hookrightarrow \Xwcl$ is the inclusion map as above. It remains to define the maps $j$ and $j'$. They have been considered in~\cite[Section~1.4]{KL2}; see for example~\cite{FS,KWY} for further details. Observe that $\Rich_v^w\subset\Rsemi_v^w\subset \Xvm$. The maps $j,j'$ are the restrictions of a map $\BA:\Xvm\times \Xv \hookrightarrow \Xq$ defined as follows. Recall that $U\subset B$ and $U_-\subset B_-$ are the unipotent radicals of $B$ and $B_-$. Any element of $\Xvm$ can be written uniquely as $xvB/B$ for an element $x\in U_-\cap vU_-v^{-1}$. (These objects do not depend on the choice of representative $\dv$ of $v$.) Similarly, any element of $\Xv$ can be written uniquely as $yvB/B$ for an element $y\in U\cap vU_-v^{-1}$. We then define 
\begin{equation*}
  \BA(xvB/B,yvB/B):=yxvB/B.
\end{equation*}
It is not hard to see (using \emph{Gaussian decomposition}) that the map $\BA$ is injective and yields an isomorphism $\BA:\Xvm\times \Xv \xrasim vB_-B/B$. Moreover, if $xvB/B\in \Xu$ for some $u\in W$, then we have $yxvB/B\in \Xu$ since $y\in U\subset B$. Thus, $\BA$ restricts to an inclusion $\Rich_v^u\times \Xv\hookrightarrow \Xu$  for each $u\in W$. The map $j'$ is this inclusion for the special case $u:=w$.  The map $j$ is the restriction of $\BA$ to the union of $\Rich_v^u\times \Xv$ over all $u\in W$ satisfying $v\leq u\leq w$.

The torus $T$ acts on each space in each commutative diagram in \cref{fig:RSW}. The action on the direct products $\Rich_v^w\times \Xv$ and $\Rsemi_v^w\times \Xv$ is given by $t\cdot (aB/B,bB/B)=(taB/B,tbB/B)$ for $t\in T$. Notice that conjugation by $t$ preserves each of the subgroups $U$, $U_-$, and $vU_-v^{-1}$. Therefore for $x\in U_-\cap vU_-v^{-1}$, we have $txvB/B=txt^{-1}vB/B$, where $txt^{-1}\in  U_-\cap vU_-v^{-1}$, and similarly for $y\in U\cap vU_-v^{-1}$. Thus, the map $\BA$ is $T$-equivariant:
\begin{equation*}
  \BA(txvB/B,tyvB/B)=tyxvB/B.
\end{equation*}
 We conclude that all maps in \cref{fig:RSW} are $T$-equivariant.

The maps $j,j',j_w,j'_w$ in \figref{fig:RSW}(middle) form a Cartesian square. We get the following isomorphisms:
\begin{equation*}
   i_{v,w}^!j_{w,!} \kuw\cong k^!j^!j_{w,!} \kuw\cong k^!j^*j_{w,!} \kuw \cong   k^!j'_{w,!} j'^* \kuw \cong k^!j'_{w,!} \ku_{\Rich_v^w\times \Xv}.
\end{equation*}
The first and the last isomorphisms are trivial.  The second isomorphism follows from the fact that $j$ is an open embedding~\cite[Section~1.4.5]{BL}. The third isomorphism is the base change theorem~\cite[Section~1.4.6]{BL}. 

We now apply \cref{prop:Soe} with $q =k$, and $p = \proj$, and $\Fcal = j'_{w,!} \ku_{\Rich_v^w\times \Xv}$. The $\Gm$-action is the composition of the $T$-action with the cocharacter $\chi_\rho^\vee:\Gm\to T$, where $\rho$ is a strictly dominant coweight satisfying $\<\rho,\alpha\>>0$ for any positive root $\alpha$. Since $\rho$ is strictly dominant, this $\Gm$-action extends to an $\A^1$-action by the same argument as in~\cite[Section~8.2]{GKL3}. The $\A^1$-action is obviously compatible with the $T$-action, and thus \cref{prop:Soe} applies.

%
We obtain
\begin{equation*}
  k^!j'_{w,!} \ku_{\Rich_v^w\times \Xv} \cong \projvw_!j'_{w,!} \ku_{\Rich_v^w\times \Xv}.
\end{equation*}

Applying base change to the Cartesian diagram in \figref{fig:RSW}(right), we get
\begin{equation*}
  \projvw_!j'_{w,!} \ku_{\Rich_v^w\times \Xv} \cong   (\projvw\circ j'_{w})_! \ku_{\Rich_v^w\times \Xv} \cong  (\projvw\circ j'_{w})_! \qmap^*\ku_{\Rich_v^w}  \cong a^*b_!\ku_{\Rich_v^w}. \qedhere
\end{equation*}
\end{proof}

So far we have constructed an isomorphism
\begin{equation}\label{eq:RSW_aab}
  \begin{aligned}
 \ExtDC^{m,((\wt-\lvw)/2)}(\Delta_v,\Delta_w) &\cong  \H_T^{m+\lvw,(\wt/2)}(\Xv,a^*b_!\ku_{\Rich_v^w}) \\ &\cong \H_T^{m+\lvw,(\wt/2)}(\pt,a_*a^*b_!\ku_{\Rich_v^w}).
\end{aligned}
\end{equation}
Since $\Xv$ is an affine space that is $T$-equivariantly homotopy-equivalent to a point (cf. \cite[Lemma 1]{Soe89}), the adjunction $(a^*,a_*)$ induces an isomorphism of functors $a_* a^* \to \id$, and thus 
\begin{equation}\label{eq:RSW_b}
  \H_T^{m+\lvw,(\wt/2)}(\pt,a_*a^*b_!\ku_{\Rich_v^w})\cong \H_T^{m+\lvw,(\wt/2)}(\pt,b_!\ku_{\Rich_v^w}),
\end{equation} 
which equals to $H_{T,c}^{m+\lvw,(\wt/2,\wt/2)}(\Rich_v^w)$ by~\eqref{eq:equivcompact}.  All the isomorphisms are natural (coming from sheaf operations) and thus \eqref{eq:RSW_b} is compatible with the action of $R$ on both sides.
This completes the proof of \cref{prop:RSW}. \qed

\section{Proof of Proposition~\ref{prop:degr}}\label{sec:sheafproof}

Recall from \cref{rmk:kC} that we switch to working with Soergel bimodules over $\k$, so that e.g. $R=\k[\h]$. By definition, given a (graded) Soergel bimodule $B=\bigoplus_r B^r$, $\Fr$ acts diagonally on each $B^r$ by multiplication by $q^r$. 

Equivariant derived categories are identified with categories of dg-modules in the work of Bernstein-Lunts \cite{BL}.  By using the formalism of Yun \cite[Appendix B]{BY}, we avoid explicit mention of dg-modules in the situation of interest to us.

An $(R, \Fr)$-module is an $R$-module $M$ equipped with an action of $\Z\Fr$, such that $\Fr(r \cdot x) = \Fr(r) \cdot \Fr(x)$ for $x\in M$ and $r\in R$.  The twist functor $\{m/2\}$ sends a module $M$ to the module $M\{m/2\}$ where the action of $\Fr$ has been multiplied by $q^{m/2}$.
Let $\Dperf(R , \Fr)$ denote the full triangulated subcategory of the derived category of $(R, \Fr)$-modules generated by half-integer twists of $R$.  According to \cite[Corollary~B.4.1]{BY}, we have an equivalence of triangulated categories $\Db_{B}(\pt, \k) \simeq \Dperf(R, \Fr)$.  Similarly, we define $\Dperf(R \otimes R, \Fr)$ and have an equivalence $\Db_{B \times B}(\pt, \k) \simeq \Dperf(R \otimes R, \Fr)$.

Recall that $\Hom_{\SBim}$ includes only bimodule morphisms of degree zero. We let 
\begin{equation*}
  \Hom_{R\otimes R}(B,B'):= \bigoplus_{\wt\in\Z} \Hom_{\SBim}(B,B'\psa{-\wt/2})
\end{equation*} 
denote the space of morphisms of arbitrary degree. Thus, $\Hom_{R\otimes R}(B,B')$ is an $(R,\Fr)$-module. Given a complex $C^\bul\in\KSBim$, we regard $\Hom_{R \otimes R}(R, C^\bul)$ (obtained by applying $\Hom_{R\otimes R}(R,-)$ term-wise) as a complex of $(R,\Fr)$-modules, treated as an element of $\Dperf(R, \Fr)$.

Now, for $\Fcal,\Gcal \in \DB$, let $\sRHom(\Fcal,\Gcal)\in \DB$ denote the internal derived hom.   With $\pi:G \to \pt$, we define
$$
\RHom(\Fcal,\Gcal)=\RHom_{X}(\Fcal,\Gcal):= \pi_* \sRHom(\Fcal,\Gcal) \in \Db_{B}(\pt,\k).
$$
Thus, $\H^\bullet_B(\RHom(\Fcal,\Gcal)) = \Ext^\bullet(\Fcal,\Gcal)$.  We shall establish the following strengthening of Proposition~\ref{prop:degr}.
\begin{proposition}\label{prop:RHom}
We have $$\RHom_{X}(\Delta_v(-\ell(v)/2),\Delta_w(-\ell(w)/2)) \simeq \Hom_{R \otimes R}(R, \FR_{v,w})$$ inside $\Dperf( R, \Fr)$.
\end{proposition}
 Proposition~\ref{prop:degr} follows 
from Proposition~\ref{prop:RHom} by taking cohomology $\H_B: \Db_{B}(\pt,\k) \to (R_{\rm gr}, \Fr)-\mod$, where $(R_{\rm gr}, \Fr)-\mod$ denotes (cohomologically) graded $R$-modules equipped with a $\Fr$-action.  The cohomological degree $r+k$ on the right-hand side of \eqref{eq:HHExt} appears since the functor $\H: \Dperf(R,\Fr) \to (R_{\rm gr},\Fr)-\mod$ sends the sum of the two gradings to the cohomological one; see \cite[Corollary B.4.1(1)]{BY}. 

\subsection{Realization functors}
We record two results on realization functors taken from \cite{AMRW2}.  See also \cite{Bei, Rid}. For a definition of a filtered version of a triangulated category, see~\cite[Definition~A.1]{Bei}.
\begin{proposition}[{\cite[Proposition~2.2]{AMRW2}}]\label{prop:AMRW1}
Let $\Tcal$ be a triangulated category that admits a filtered version $\tilde \Tcal$ and let $\Acal \subset \Tcal$ be a full additive subcategory that admits no negative self-extensions.  Then there is a functor of triangulated categories
$$
\real: \Kb \Acal \to \Tcal
$$
whose restriction to $\Acal$ is the inclusion functor.
\end{proposition}

\begin{proposition}[{\cite[Proposition~2.3]{AMRW2}}]\label{prop:AMRW2}
Let $\Tcal_1$ and $\Tcal_2$ be triangulated categories admitting a filtered version, and let $\Acal_1 \subset \Tcal_1$ and $\Acal_2 \subset \Tcal_2$ be two full additive subcategories admitting no negative self-extensions.  Let $F: \Tcal_1 \to \Tcal_2$ be a triangulated functor that restricts to an additive functor $F_0:\Acal_1 \to \Acal_2$.  If $F$ lifts to a functor $\tilde F: \tilde \Tcal_1 \to \tilde \Tcal_2$, then the following diagram commutes up to natural isomorphism:
\[ \begin{tikzcd}
\Kb\Acal_1 \arrow{r}{\real} \arrow[swap]{d}{\Kb F_0} & \Tcal_1 \arrow{d}{F} \\%
\Kb\Acal_2 \arrow{r}{\real}& \Tcal_2
\end{tikzcd}
\]
\end{proposition}

\subsection{Semisimple complexes}
Let $\Semis_B(X) \subset \Db_B(X,\k)$ denote the additive subcategory generated by semisimple complexes pure of weight 0.  Thus, an object of $\Semis_B(X)$ is a direct sum of the twisted intersection cohomology sheaves $\IC_w[n](n/2)$ for $w \in W$ and $n \in \Z$.

Recall from \cref{sec:soergel-bimodules} that $S_w \subset \BS_\uw$ denotes the indecomposable Soergel bimodule indexed by $w \in W$.

\begin{lemma}\label{lem:ICpure}
For $\Fcal, \Gcal \in \Semis_B(X)$, the $\Ext$-group $\ExtDC^i(\Fcal, \Gcal)$ is pure of weight $i$ for all $i\in\Z$.  
\end{lemma}
\begin{proof}
Follows from \cite[Lemma 3.1.5]{BY}.
\end{proof}

\begin{proposition}\label{prop:SemiSBim}
The hypercohomology functor induces an equivalence of additive categories 
$$
\H_B: \Semis_B(X) \to \SBim,
$$
enriched over $R\otimes R$, and sending $\IC_w$ to the shifted Soergel bimodule $S_w\psa{-\ell(w)}$ and the twist $[n](n/2)$ to the change of grading $\psa{-n/2}$. 
\end{proposition}
\begin{proof}
By a well-known result of Soergel~\cite{Soe07} (see also~\cite[Section~1.3]{EW} or~\cite[Section~16.1]{EMTW}), we have $\H_B(\IC_w) \simeq S_w\psa{-\ell(w)/2}$ as a graded $R \otimes R$-module.  By Lemma~\ref{lem:ICpure}, the cohomological and weight gradings on $\H_B(\IC_w)$ agree, and furthermore $\IC_w[n](n/2)$ is sent to $S_w\{-(\ell(w)+n)/2\}$.  The result can then be deduced from \cite[Proposition 3.1.6]{BY}.
\end{proof}

The following result is well-known; see~\cite{ARs} and \cite[Lemma B.1.1]{BY}.
\begin{lemma}\label{lem:Extgroups}
For $\Fcal, \Gcal \in \Semis_B(X)$, the action of $\Fr$ on $\ExtDC^\bullet(\Fcal, \Gcal)$ is semisimple.  Furthermore, we have 
\begin{equation}\label{eq:homSemis}
\hom_{\Db_B(X,\k)}(\Fcal,\Gcal) \simeq \ExtDC^0(\Fcal, \Gcal) \simeq \ExtDC^{0,(0)}(\Fcal, \Gcal)
\end{equation}
and $\ext^i_{\Db_B(X,\k)}(\Fcal,\Gcal) = 0$ for $i < 0$.
\end{lemma}

\subsection{The mixed derived category}
Following \cite{AR2}, we define the {\it mixed derived category} 
$$
\Dmix_B(X):= \Kb(\Semis_B(X))
$$
to be the homotopy category of cochain complexes in $\Semis_B(X)$.  Define the {\it Tate twist} of $\Dmix_B(X)$ by
$$
\langle n \rangle:= \psa{-n/2}[-n],
$$
where $[n]: \Kb(\Semis_B(X)) \to \Kb(\Semis_B(X))$ is the cohomological shift functor, and $\psa{n/2}: \Semis_B(X) \to \Semis_B(X)$ is the autoequivalence $\Fcal \mapsto \Fcal[-n](-n/2)$.

By Proposition~\ref{prop:SemiSBim}, we have an equivalence of triangulated categories
$$
\Kb \H_B: \Dmix_B(X) \to \KSBim.
$$

\begin{theorem}
There exists a triangulated realization functor
$$
\real: \Dmix_B(X) \to \Db_B(X,\k),
$$
restricting to the inclusion on $\Semis_B(X)$, sending 
$$[n] \to [n], \qquad \langle n \rangle \to (n/2), \qquad \psa{n/2} \to [-n](-n/2),
$$
such that the composition
$$
\kappa:= \omega \circ \real
$$
satisfies
\begin{align}
\label{eq:eqgrading1}
\Hom_{\Dmix_B(X)}(\Fcal,\Gcal \langle n \rangle) &\simeq \ExtDC^{0,(n/2)}(\real \Fcal, \real \Gcal), \\
\label{eq:eqgrading2}
\bigoplus_{n \in \Z} \Hom_{\Dmix_B(X)}(\Fcal,\Gcal \langle n \rangle) &\simeq \ExtDC^{0}(\real \Fcal, \real \Gcal) = \Hom_{\Db_B(X,\k)}(\kappa \Fcal, \kappa \Gcal).
\end{align}
Furthermore, all functors are compatible with the $R \otimes R$ action on the corresponding $\ExtDC^\bullet$-groups.
\end{theorem}
\begin{proof}
By Lemma~\ref{lem:Extgroups}, $\Semis_B(G/B)$ has no negative self-extensions.  Since $ \Db_B(X,\k)$ admits a filtered version (see \cite{Bei}), we may apply Proposition~\ref{prop:AMRW1} to obtain a realization functor 
$$
\real: \Dmix_B(X) \to \Db_B(X,\k)
$$
that restricts to the inclusion on $\Semis_B(G/B)$.  For $\Fcal,\Gcal \in \Semis_B(X)$, the isomorphism~\eqref{eq:eqgrading2} follows from \eqref{eq:homSemis} while~\eqref{eq:eqgrading1} follows from \eqref{eq:homSemis} and~\eqref{eq:eqgrading2}, since $\langle n \rangle:  \Dmix_B(X) \to \Dmix_B(X)$ is sent to $(n/2): \Db_B(X,\k) \to \Db_B(X,\k)$ by the realization functor $\real$.  For $\Fcal,\Gcal \in \Kb \Semis_B(X)$, \eqref{eq:eqgrading1}--\eqref{eq:eqgrading2} are proven by double induction on the lengths of chain complexes representing $\Fcal, \Gcal$.  See \cite[Section 4.1]{Rid} for a detailed argument.
\end{proof}

\subsection{Standard objects and Rouquier complexes}

\begin{proposition}\label{prop:stand}
The composition $\real \circ \H_B^{-1}:\KSBim \to \Db_B(X,\k)$ takes $F^\bullet(w)$ to $\Delta_w(-\ell(w)/2)$ and $F^\bullet(v)^{-1}$ to $\nabla_v(\ell(v)/2)$.
\end{proposition}

Note that Proposition~\ref{prop:degr} follows nearly immediately from Proposition~\ref{prop:stand} and~\eqref{eq:eqgrading1}--\eqref{eq:eqgrading2}: since the realization functor sends $[k]\psa{-\wt/2}$ to $[k+\wt](\wt/2)$, we have 
\begin{align*}
\HHX {k}\wt/2_{v,w}  &\cong  \Hom_{\KSBim}(R, F^\bullet_{v,w} [k]\psa{-\wt/2})) \\
& \cong \ExtDC^{\wt+k,(\wt/2)}(\Delta_e, \Delta_w(-\ell(w)/2) \star \nabla_v(\ell(v)/2))
\\
& \cong\ExtDC^{\wt+k,(\wt/2)}(\Delta_v(-\ell(v)/2),\Delta_w(-\ell(w)/2)) \\
&\cong   \ExtDC^{\wt+k,((\wt-\ell_{v,w})/2)}(\Delta_v,\Delta_w).
\end{align*}
The second isomorphism above is obtained from the adjointness of the convolution functors $(-) \star  \Delta_v(-\ell(v/2))$ and $(-)  \star \nabla_v(\ell(v)/2))$ to be presently explained; see \cref{lem:RHomad} for a stronger statement, and see also~\eqref{eq:biadj1}--\eqref{eq:biadj2} for a similar statement on the Soergel bimodule side.

\begin{proof}[Proof of Proposition~\ref{prop:stand}]
We prove the claim for $F^\bullet(w)$.  The claim for $F^\bullet(v)^{-1}$ is similar.

There is a monoidal structure $\star: \Db_B(X,\k) \times  \Db_B(X,\k)  \to  \Db_B(X,\k) $ obtained by convolution; see for example \cite[Section 3.2]{BY} or \cite[Section 4.3]{AR2}.  By \cite[Proposition~3.2.5]{BY}, the additive subcategory $\Semis_B(X) \subset  \Db_B(X,\k)$ is preserved by convolution.   According to \cite[Proposition 3.2.1]{BY}, convolution $\star$ is sent by $\H_B$ to the tensor product operation on $\SBim$.  Note that the {\it derived} tensor product in \cite{BY} reduces to the tensor product on $\SBim$ since all Soergel bimodules are free as left (or right) $R$-modules.  

For a simple generator $s \in S$, %
let $\pi_s: X =G/B \to G/P_s$ denote the projection to the partial flag variety, where $P_s \supset B$ denotes a minimal parabolic subgroup, and let $\theta_s: \Db_B(X,\k)  \to  \Db_B(X,\k)$ denote the composition $\theta = \pi_s^* \pi_{s,*}$.  By \cite[Lemma~3.2.7]{BY} (see also \cite[Lemma 4.3]{AR2}), we have a natural isomorphism of functors $\theta_s \simeq (-) \star \IC_s[-1](-1/2)$, and $\theta_s$ restricts to an endofunctor $\theta_s:\Semis_B(X) \to \Semis_B(X)$.  It is well-known \cite[Korollar 2]{SoeGarben} that the equivalence $\H_B: \Semis_B(X) \to \SBim$ takes the functor $\theta_s$ to the functor $R \otimes_{R^s} (-)$. 

Now, there is a natural morphism of functors $ \theta_s \to \id$ arising from the adjunction of $\pi_s^*$ and $\pi_{s,*}$.  The map $B_s \to R$ in \eqref{eq:Rouquier_s} arises by an analogous adjunction; see \cite[Section 3]{Rou}.  Now, $\IC_e \simeq \Delta_e$ and $\H_B(\IC_e) = R$, and the morphism $\theta_s \to \id$ applied to $\IC_e$ fits into the distinguished triangle 
$$
 \IC_s [-1](-1/2) \to \IC_e \to \Delta_s(-1/2)
$$
in $\Db_B(X,\k)$; see for example \cite[(C.4)]{BY} or \cite[Lemma 4.1]{AR2}.  It follows that we have
\begin{equation}\label{eq:Delta12}
\real \circ \H_B^{-1}(\FR(s)) = \Delta_s(-1/2).
\end{equation}
See also \cite[Proposition 5.3]{Rou}.  This establishes Proposition~\ref{prop:stand} in the case $\ell(w)=1$.  We then obtain a natural isomorphism of functors
\begin{equation}\label{eq:DH}
 \left((-) \star \Delta_s(-1/2) \right) \circ \real \circ \H_B^{-1} \simeq \real \circ \H_B^{-1}\circ \left((-) \otimes F(s)\right)  
\end{equation}
from $\KSBim$ to $\Db_B(X,\k)$.

On the other hand, it is known (\cite[Lemma 3.2.2]{BY}) that if $w = s_1 s_2 \cdots s_l$ is a reduced decomposition then 
$$
\Delta_w \simeq \Delta_{s_1} \star \Delta_{s_2} \star \cdots \star \Delta_{s_l}
$$
is in $\Db_B(X, \k)$.  Combining \eqref{eq:Delta12} with \eqref{eq:DH}, we find
\begin{align*}
\real \circ \H_B^{-1}(\FR(w)) &\simeq \real \circ \H_B^{-1}(\FR(s_1)  \otimes \cdots \otimes \FR(s_\ell)) \\
& \simeq \Delta_{s_1}(-1/2)\star \cdots  \star \Delta_{s_\ell}(-1/2) \simeq \Delta_w(-\ell(w)/2).\qedhere
\end{align*}
\end{proof}

\subsection{Proof of Proposition~\ref{prop:RHom}}

\subsubsection{}
A functor $F: D^b(Y,\k) \to D^b(Z,\k)$ is called {\it geometric} (\cite[Definition 6.6]{ARs}) if there is a natural transformation
$$
\RHom_{Y}(\Fcal,\Gcal) \to \RHom_{Z}(F \Fcal, F\Gcal)
$$
for $\Fcal,\Gcal \in D^b(Y,\k)$.
We shall apply the notion of a geometric functor for $B$-equivariant derived categories.
\begin{lemma}\label{lem:geometric}
The endofunctors $(-)\star \Delta_s(-1/2), \; (-) \star \nabla_s(1/2): \Db_B(X,\k) \to \Db_B(X,\k)$ are geometric.
\end{lemma}
\begin{proof}
This is stated for the affine Grassmannian case in \cite[Proposition 12.2]{ARs}.  The same proof applies in the flag variety case.  
\end{proof}

\begin{lemma}\label{lem:RHomad} We have
$$
\RHom(\Delta_v,\Delta_w) = \RHom(\Delta_e, \Delta_w \star \nabla_v)
$$
inside $\Db_{B}(\pt,\k)$. 
\end{lemma}
\begin{proof}
By Lemma~\ref{lem:geometric}, for $\Fcal,\Gcal \in \Db_B(X,k)$. we have a map
\begin{equation}\label{eq:RHom1}
\RHom(\Fcal,\Gcal) \to \RHom( \Fcal \star \nabla_s, \Gcal \star \nabla_s)
\end{equation}
and a map
\begin{equation}\label{eq:RHom2}
\RHom( \Fcal \star \nabla_s,\Gcal \star \nabla_s) \to \RHom(\Fcal \star \nabla_s \star \Delta_s,\Gcal \star \nabla_s\star \Delta_s)
\end{equation}
inside $\Db_B(\pt,\k)$.  Convolution is associative and $\nabla_s \star \Delta_s \simeq \Delta_e$ (\cite[Proposition 4.4]{AR2}), and also $(-) \star \Delta_e$ is the identity functor.  So composing \eqref{eq:RHom1} and \eqref{eq:RHom2}, we get an automorphism $\RHom(\Fcal,\Gcal) \simeq \RHom(\Fcal,\Gcal)$ inside $\Db_B(\pt,\k)$.  It follows that $\RHom(\Fcal,\Gcal) \simeq \RHom(\Fcal \star \nabla_s,  \Gcal \star \nabla_s)$.  Choosing $\Fcal = \Delta_v$ and $\Gcal = \Delta_w$ and repeatedly applying this we obtain the required statement.
\end{proof}  
We remark that we could have defined $\RHom(\Delta_v,\Delta_w)$ and $\RHom(\Delta_e, \Delta_w \star \nabla_v)$ as objects in $\Db_{B \times B}(\pt,\k)$, but Lemma~\ref{lem:RHomad} would not hold.  The functors $(-) \star \Delta_s$ and $(-) \star \nabla_s$ ``commute" with only one of the $B$-actions.

\subsubsection{}
Let $\iota_e: B \hookrightarrow G$ be the inclusion and $\pi: G/B \to \{e\} = \pt$ be the projection. These maps are $B$-equivariant.  Now, the sheaf $\Delta_e$ is supported on a single point $\{e\} \subset G/B$.  Thus, for any $\Fcal \in \Db_B(X,\k)$, the object $\sRHom(\Delta_e, \Fcal) \in \Db_B(X,\k)$ is also supported on $\{e\}$.  The pullback $\iota_e^*$ and pushforward $\iota_{e,*}$ are inverse equivalences of categories between $D^b_B(\{e\},\k)$ and the full subcategory of $\Db_B(X,\k)$ consisting of objects whose cohomology sheaves are supported on $\{e\}$.  
Inside $\Db_B(\{e\},\k) = \Db_B(\pt,\k)$, we thus have
\begin{equation}\label{eq:RHomiota1}
  \begin{aligned}
 \RHom_X(\Delta_e, \Fcal) &\simeq \pi_*  \sRHom_X( \Delta_e, \Fcal) \\ &\simeq \iota_e^* \sRHom_{X}( \Delta_e, \Fcal) = \sRHom_{\{e\}}(\k, \iota_e^* \Fcal) = \iota_e^* \Fcal.
\end{aligned}
\end{equation}
\subsubsection{}

The equivalence $\Db_B(\pt,\k) \simeq \Dperf(R, \Fr)$ (resp., $\Db_{B \times B}(\pt,\k) \simeq \Dperf(R \otimes R, \Fr)$)  sends objects pure of weight 0 to objects in $\Dperf(R, \Fr)$ (resp., $\Dperf(R \otimes R, \Fr)$) concentrated in cohomological degree 0.  For $\Fcal \in \Semis_B(X)$, we thus view the Soergel bimodule $\H_B(\Fcal)$ as sitting inside $\Dperf(R \otimes R, \Fr)$ in cohomological degree 0, that is, as an object in $\Mod(R \otimes R,\Fr)$. 

\begin{lemma}\label{lem:pureagree}
For $\Fcal \in \Semis_B(G/B)$, we have an isomorphism
\[
\iota_e^* \Fcal \simeq \Hom_{R \otimes R}(R, \H_B(\Fcal))
\]
inside $\Mod(R,\Fr)$.
\end{lemma}
\begin{proof}
For $\Fcal \in \Semis_B(G/B)$, we have that $\iota_e^* \Fcal = \iota_e^! \Fcal$ is again pure of weight 0.  Thus, $\iota_e^* \Fcal \in \Db_B(\pt,\k)$ can be identified with an element of $\Mod(R, \Fr)$.
By \cite[Proposition 3.1.6]{BY} and \eqref{eq:RHomiota1}, we have 
\begin{equation*}%
  \Hom_{R \otimes R}(R, \H_B(\Fcal)) \simeq \H_B( \RHom(\Delta_e, \Fcal)) \simeq \H_B(\iota_e^*\Fcal).
\end{equation*}
  (Since $ \iota_e^* \Fcal$ is pure of weight 0, $\H_B(\iota_e^*\Fcal)$ is simply the corresponding object in $\Mod(R,\Fr)$.)  We conclude that $ \iota_e^* \Fcal \simeq \Hom_{R \otimes R}(R, \H_B(\Fcal))$ inside $\Mod(R, \Fr)$ and the result follows.\end{proof}

We remark that $\Hom_{R \otimes R}(R, \H_B(\Fcal))$ is free as an $R$-module; cf. \cite[Lemma 3.1.5]{BY}.  See also Remark~\ref{rmk:diag}.

\subsubsection{}
Recall that we have a realization functor $\real: \Kb\Semis_B(X) \to \Db_{B}(X,\k)$.
We now have two functors from $\Kb \Semis_B(X)$ to $\Db_{B }(\pt,\k)$.  The functor
\begin{equation}\label{eq:RHomiota2}
\iota_e^* \circ \real: \Kb \Semis_B(X) \to \Db_{B}(X,\k)\to \Db_{B }(\pt,\k)
\end{equation}
and the functor 
\begin{equation}\label{eq:HomRH}
\Kb\Hom_{R \otimes R}(R, -) \circ \Kb \H_B: \Kb \Semis_B(X) \to \KSBim \to \Db_{B}(\pt,\k).
\end{equation}
We explain the last functor $\Kb\Hom_{R \otimes R}(R, -): \KSBim \to \Db_B(\pt,\k)$.  Let $\Free(R, \Fr)$ denote the category of finitely generated free $R$-modules equipped with an action of $\Fr$.  The functor $\Hom_{R \otimes R}(R,-)$ takes $\SBim$ to $\Free(R,\Fr)$, and $\Kb\Hom_{R \otimes R}(R,-)$ takes $\KSBim$ to $\Kb \Free(R,\Fr)$.
We have an inclusion $\Free(R,\Fr) \to \Db_B(\pt,\k)$.  Applying Proposition~\ref{prop:AMRW1}, we obtain a triangulated functor $\real: \Kb(\Free(R,\Fr)) \to \Db_B(\pt,\k)$.  Composing $\Kb\Hom_{R \otimes R}(R,-)$ with $\real$ we obtain $\Kb\Hom_{R \otimes R}(R, -): \KSBim \to \Db_B(\pt,\k)$.

\def\G{L}
By Lemma~\ref{lem:pureagree}, the two triangulated functors \eqref{eq:RHomiota2} and \eqref{eq:HomRH} agree on the subcategory $\Semis_B(X) \subset \Kb \Semis_B(X)$, sending $\Semis_B(X)$ to $\Free(R,\Fr) \subset \Dperf(R,\Fr) \simeq  \Db_{B}(\pt,\k)$.  Denoting this restriction by $\G: \Semis_B(X) \to \Free(R,\Fr)$, we apply Proposition~\ref{prop:AMRW2} to deduce that both triangulated functors are isomorphic to $$\real \circ \Kb \G: \Kb \Semis_B(X) \to \Kb \Free(R,\Fr) \to \Dperf(R,\Fr) \simeq  \Db_{B}(\pt,\k).$$  Thus, 
\begin{equation}\label{eq:twofunctors}
\iota_e^* \circ \real \simeq \Kb\Hom_{R \otimes R}(R, -) \circ \Kb \H_B.
\end{equation}

\begin{remark}
The essential image of the realization functor $\Kb\Free(R,\Fr) \to \Db_{\perf}(R,\Fr)$ is a subcategory of $\Db_{\perf}(R,\Fr)$ equivalent to the infinitesimal extension of $\Kb\Free(R,\Fr)$, in the sense of \cite{ARs}.
\end{remark}

\subsubsection{Conclusion.}
By \eqref{eq:twofunctors} and Proposition~\ref{prop:stand}, we have $$\iota_e^*(\Delta_w(-\ell(w)/2) \star  \nabla_v(\ell(v)/2)) \simeq \Hom_{R \otimes R}(R, \FR_{v,w})$$ inside $\Db_{\perf}(R,\Fr)$. By \cref{lem:RHomad} and~\eqref{eq:RHomiota1}, we find $$\iota_e^*(\Delta_w(-\ell(w)/2) \star  \nabla_v(\ell(v)/2)) \simeq \RHom_{X}(\Delta_v(-\ell(v)/2),\Delta_w(-\ell(w)/2)).$$ This finishes the proof of \cref{prop:RHom}, as well as of \cref{prop:degr} and \cref{thm:main}.

\section{Ordinary cohomology, Koszul duality, and Verma modules}\label{sec:ord_cohom}
The goal of this section is to prove \cref{thm:ordinary,thm:Koszul,thm:Verma}.

\subsection{Ordinary cohomology}\label{sec:ordinary-cohomology}
We have a forgetful functor $\For: \Db_B(\pt,\k) \to \Db(\pt,\k)$, and a commutative diagram (see \cite{BL} and \cite[Proposition B.3.1]{BY})
\begin{equation}\label{eq:conereal2} \begin{tikzcd}
 \Db_B(\pt,\k) \arrow{r}{\For} \arrow[swap]{d}{\simeq} & \Db(\pt,\k) \arrow{d}{\simeq} \\%
\Db_{\perf}(R,\Fr) \arrow{r}{\otimes_R^L \k}&\Db(\k,\Fr).
\end{tikzcd}
\end{equation}
Here, $\Db(\k,\Fr)$ is the derived category of finite-dimensional $\k$-vector spaces equipped with an $\Fr$-action with integer weights.  Applying $\For$ to \cref{prop:RHom}, we obtain the following. 

\begin{proposition}\label{prop:RHomk}
We have 
$$\RHom_{\Db(X,\k)}(\DeltaC_v(-\ell(v)/2),\DeltaC_w(-\ell(w)/2)) \simeq \Hom_{R \otimes R}(R, \FR_{v,w}) \otimes_R \k$$ inside $\Db(\k,\Fr)$.\qed
\end{proposition}
Here, $\DeltaC_v = \For(\Delta_v) \in \Db_{(B)}(X,\k)$ denotes the ordinary standard object in the Borel-constructible derived category, and the derived tensor product $\otimes^L_R \k$ is replaced by the usual tensor product since $\Hom_{R \otimes R}(R, \FR_{v,w})$ is free as an $R$-module.  Taking the hypercohomology of both sides of Proposition~\ref{prop:RHomk}, we obtain the following. 

\begin{corollary}\label{cor:degrk}
For all $v\leq w\in W$, and all $k,\wt \in \Z$, we have an isomorphism
\begin{equation}\label{eq:HHH_to_Ext_ordinary}
\ExtDC^{\wt+k,((\wt-\lvw)/2)}(\DeltaC_v,\DeltaC_w) \cong  \HHXk {k}\wt/2_{v,w}  
\end{equation}
of $\k$-vector spaces. (In particular, both sides are zero for odd $\wt$.)
\end{corollary}
\noindent Here we set $\HHk^0(B):=\HH^0(B)\otimes_R\k$, similarly to~\eqref{eq:HHC}. For the shift in cohomological degree, see the discussion after \cref{prop:RHom}.

\begin{proof}[Proof of \cref{thm:ordinary}.]

The non-equivariant version of \cref{prop:RSW} is given in~\cite[Proposition~4.2.1]{RSW}. Similarly to~\eqref{eq:RSW_aab}--\eqref{eq:RSW_b}, we get
\begin{equation*}
  \Ext^{m,((\wt-\lvw)/2)}(\DeltaC_v,\DeltaC_w) \cong H^{m+\lvw,(\wt/2)}_c(\Rich_v^w,\k) \quad\text{for all $m,\wt\in\Z$.}
\end{equation*}
\Poincare duality~\eqref{eq:Poincare} allows one to translate the compactly supported cohomology into the ordinary cohomology:
\begin{equation}\label{eq:RSW_ordinary}
  \Ext^{m,((\wt-\lvw)/2)}(\DeltaC_v,\DeltaC_w) \cong H^{\lvw-m,(\lvw-\wt/2)}(\Rich_v^w,\k).
\end{equation}

Combining~\eqref{eq:HHH_to_Ext_ordinary}--\eqref{eq:RSW_ordinary} with Koszul duality~\eqref{eq:Koszul} proved in the next section, and switching from working over $\k$ to working over $\C$ via \cref{rmk:kC,prop:kC}, we get 
\begin{align*}
   \dim_\C \HHXC {k}\wt/2_{v,w} &= \dim_\C H^{\lvw-k-\wt,(\lvw-\wt/2,\lvw-\wt/2)}(\Rich_v^w,\C) \\
   &= \dim_\C H^{-k,(\wt/2,\wt/2)}(\Rich_v^w,\C).
\qedhere
\end{align*}
\end{proof}

\begin{proof}[Proof of~\eqref{eq:PKRtop=PKRC}]
By the K\"unneth formula and \cref{cor:Rich_T_action}, we have 
\begin{equation*}%
  H^\bul(\Rich_v^w) \cong H^\bul(\Rich_v^w/T) \otimes_\C H^\bul(T).
\end{equation*}
 The space $H^\bul(T)$ is $2^{n-1}$-dimensional, and the mixed Hodge polynomials are related as
\begin{equation}\label{eq:KRC_proof_0}
  \Poinc(\Rich_v^w;q,t)= \left(\qq+\tt\right)^{n-1}\cdot \Poinc(\Rich_v^w/T;q,t).
\end{equation}
This implies~\eqref{eq:mod_T_qtn}. Next, we claim
\begin{equation}\label{eq:main_KRC}
\Poinc(\Rich_v^w/T;q,t)= \left(\qq\tt\right)^{\chi(\bvw)}\PKRC(\bhvw;q,t).
\end{equation}
First, combining \cref{thm:ordinary} with Koszul duality~\eqref{eq:Koszul} (to be proved below in \cref{sec:Koszul}), we find 
\begin{equation}\label{eq:KRC_proof_1}
  H^{k,(p,p)}(\Rich_v^w,\C) \cong \HHXC {-\lvw-k+2p}{\lvw-p}_{v,w}.
\end{equation}
Setting $k':=-\lvw-k+2p$ and $p':=\lvw-p$, we find $p=\lvw-p'$ and $k=\lvw-k'-2p'$. Plugging this into~\eqref{eq:Poinc_MHT_dfn} and applying~\eqref{eq:KRC_proof_1}, we get
\begin{align*}%
  \Poinc(\Rich_v^w;q,t)&=\sum_{k,p\in \Z}  q^{p-\frac k2} t^{\frac{\lvw-k}2} \dim H^{k,(p,p)}(\Rich_v^w,\C)\\
 &=\sum_{k',p'\in \Z}  q^{\frac{\lvw+k'}2} t^{p'+\frac{k'}2} \HHXC {k'}{p'}_{v,w}.
\end{align*}
 On the other hand, rewriting~\eqref{eq:PKRC}, we see that the right-hand side of~\eqref{eq:main_KRC} is given by 
\begin{equation*}%
{\left(\qq+\tt\right)^{1-n}}\sum_{k',p'\in\Z} q^{\frac{\lvw+k'}2} t^{p'+\frac{k'}2} \dim \HHBC{k'}{p'} (\FR(\beta)).
\end{equation*}
Together with~\eqref{eq:KRC_proof_0}, this finishes the proof of~\eqref{eq:main_KRC}. Finally,~\eqref{eq:PKRtop=PKRC} follows by comparing~\eqref{eq:main_KRC} with (the Richardson version of)~\eqref{eq:main_KR}.
\end{proof}

Since the $R$-action on $H^\bul_{T,c}(\Rich_v^w)$ is trivial (that is, $\h^\ast$ acts by zero), by \cref{thm:main}, the $R$-action on $\HHH^0(\Fvw)$ is also trivial. Alternatively, for $W=S_n$ and any knot $\betah$, the $R$-action on $\HHH^0(\FR(\beta))$ is trivial by \cref{cor:Koszul_laction}.   
 It thus follows that we have an isomorphism of bigraded $\C$-modules
\begin{equation}\label{eq:Torconj}
  \HHHC^0(\FR(\beta))\cong \Tor^R_\bul(\C, \HHH^0(\FR(\beta))).
\end{equation}
We conjecture that~\eqref{eq:Torconj} holds for all $W$ and all $\beta \in \BraidW$.  This would follow from~\eqref{eq:conereal2} if $\HH^0(\FR(\beta))$ and $\HHH^0(\FR(\beta))$ were known to be equivalent in $\Db_{\perf}(R,\Fr)$.

\subsection{Koszul duality and $q,t$-symmetry}\label{sec:Koszul}
We prove \cref{thm:Koszul}. 
 By \cite[Equation~(5.2), Theorem~5.3.1, and Remark~5.3.2]{BY}, for $k,\wt\in\Z$ and $v\leq w\in W$, we get an isomorphism
\begin{equation*}
  \Ext^{k,(\wt/2)}(\DeltaC_v,\DeltaC_w)\cong \Ext^{k-\wt,(-\wt/2)}(\DeltaC_{v^{-1}},\DeltaC_{w^{-1}})
\end{equation*}
of vector spaces. By~\eqref{eq:RSW_ordinary} and \cref{prop:kC}, this implies
\begin{equation}\label{eq:CL_inverse}
H^{k,(\wt/2,\wt/2)}(\Rich_v^w,\C) \cong H^{\lvw+k-\wt,(\lvw-\wt/2,\lvw-\wt/2)}(\Rich_{v^{-1}}^{w^{-1}},\C).
\end{equation}
The only difference between~\eqref{eq:CL_inverse} and the desired result~\eqref{eq:Koszul} is the appearance of $v^{-1}$ and $w^{-1}$ on the right-hand side. In fact, it is not hard to see that the Richardson varieties $\Rich_v^w$ and $\Rich_{v^{-1}}^{w^{-1}}$ are isomorphic. Indeed, recall from \cref{lemma:strawberry} that we have an isomorphism $\strvw\cong \Rich_v^w$. The map $g\mapsto g^{-1}$ restricts to an isomorphism $\strvw\cong \strvwi$ (choosing $\dv^{-1}$ as the representative for $v^{-1}$). By~\eqref{eq:strawberry}, we get an isomorphism $\Rich_v^w\xrasim \Rich_{v^{-1}}^{w^{-1}}$.
\qed

\subsection{Extensions of Verma modules}\label{sec:Verma}
We prove \cref{thm:Verma}. First, we explain the bigrading on $\Ext^\bul(M_v,M_w)$. Out of the several equivalent descriptions listed in~\cite{BGS}, the most convenient one for us is given in~\cite[Section~4.4]{BGS}: the bigraded vector spaces $\Ext^\bul(M_v,M_w)$ and $\Ext^\bul(\DeltaC_v,\DeltaC_w)$ are isomorphic (after changing the coefficients from $\C$ to $\k$), and the bigrading on $\Ext^\bul(M_v,M_w)$ comes from the bigrading on $\Ext^\bul(\DeltaC_v,\DeltaC_w)$ via Frobenius weights~\eqref{eq:Ext_bigr_dfn}:
\begin{equation*}
  \Ext^{k,(\wt/2)}(M_v,M_w):= \Ext^{k,(\wt/2)}(\DeltaC_v,\DeltaC_w).
\end{equation*}
See also~\cite[Equation~(1.1.1)]{RSW}. 

 The result follows by combining~\eqref{eq:RSW_ordinary} with Koszul duality~\eqref{eq:Koszul}.\qed

\section{Catalan numbers associated to positroid varieties}\label{sec:Deogram}
Our results give an embedding of the rational $q,t$-Catalan numbers $\Cat_{k,n-k}(q,t)$ into a family of $q,t$-polynomials $\Poinc(\Pit_f;q,t)\in\N[\qq,\tt]$ (all of which are $q,t$-symmetric and $q,t$-unimodal), indexed by $f\in\Bknc$. The goal of this section is to give a combinatorial interpretation for a specialization of $\Poinc(\Pit_f;q,t)$.
\begin{definition}
For $f\in\Bknc$, define the \emph{$f$-Catalan number} $\Cat_f\in\Z$ as the specialization
\begin{equation*}
  \Cat_f:= \Poinc(\Pit_f;q,t) \big|_{\qq=1,\tt=-1}.
\end{equation*}
Alternatively, $\Cat_f$ is the $q=1$ specialization of the point count polynomial $\#\Pit_f(\F_q)$, and we also have $\Cat_f=\Ptopnoq_f(1)$, where the polynomial $\Ptop_f$ is defined in \cref{thm:homfly}.
\end{definition}
\noindent 
 In particular, $\Cat_{\fkn}=\Cat_{k,n-k}(1,1)=\#\Dyck_{k,(n-k)}$ is the usual rational Catalan number when $\gcd(k,n)=1$.

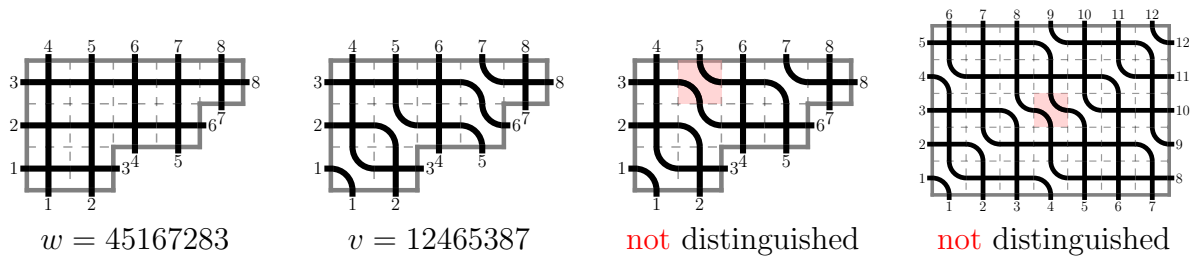
\begin{figure}

\makebox[1.0\textwidth]{
\def\sclbx{0.9}
\scalebox{0.95}{
		\begin{tabular}{cccc}
		\scalebox{\sclbx}{
	\begin{tikzpicture}[xscale=0.5,yscale=0.5]
	\draw[line width=0.5pt,dashed,opacity=\gridop] (0.00,1.00)--(2.00,1.00);\draw[line width=0.5pt,dashed,opacity=\gridop] (0.00,2.00)--(4.00,2.00);\draw[line width=0.5pt,dashed,opacity=\gridop] (0.00,3.00)--(5.00,3.00);\draw[line width=0.5pt,dashed,opacity=\gridop] (0.00,4.00)--(0.00,1.00);\draw[line width=0.5pt,dashed,opacity=\gridop] (1.00,4.00)--(1.00,1.00);\draw[line width=0.5pt,dashed,opacity=\gridop] (2.00,4.00)--(2.00,2.00);\draw[line width=0.5pt,dashed,opacity=\gridop] (3.00,4.00)--(3.00,2.00);\draw[line width=0.5pt,dashed,opacity=\gridop] (4.00,4.00)--(4.00,3.00);
	\godiagx{}{0/3,1/3,2/3,3/3,4/3,0/2,1/2,2/2,3/2,0/1,1/1}{}
	\draw[line width=1.5pt,black!50] (0.00,0.95)--(0.00,4.05);\draw[line width=1.5pt,black!50] (-0.05,4.00)--(5.05,4.00);\draw[line width=1.5pt,black!50] (-0.05,1.00)--(2.05,1.00);\draw[line width=1.5pt,black!50] (2.00,0.95)--(2.00,2.05);\draw[line width=2pt] (0.50,1.05)--(0.50,0.85);
					\node[anchor=north,scale=0.50,inner sep=1pt] (A) at (0.50,0.85) {$1$};
					\draw[line width=2pt] (1.50,1.05)--(1.50,0.85);
					\node[anchor=north,scale=0.50,inner sep=1pt] (A) at (1.50,0.85) {$2$};
					\draw[line width=2pt] (1.95,1.50)--(2.15,1.50);
				\node[anchor=west,scale=0.50,inner sep=1pt] (A) at (2.15,1.50) {$3$};
				\draw[line width=2pt] (0.05,1.50)--(-0.15,1.50);
				\node[anchor=east,scale=0.50,inner sep=1pt] (A) at (-0.15,1.50) {$1$};
				\draw[line width=1.5pt,black!50] (1.95,2.00)--(4.05,2.00);\draw[line width=1.5pt,black!50] (4.00,1.95)--(4.00,3.05);\draw[line width=2pt] (2.50,2.05)--(2.50,1.85);
					\node[anchor=north,scale=0.50,inner sep=1pt] (A) at (2.50,1.85) {$4$};
					\draw[line width=2pt] (3.50,2.05)--(3.50,1.85);
					\node[anchor=north,scale=0.50,inner sep=1pt] (A) at (3.50,1.85) {$5$};
					\draw[line width=2pt] (3.95,2.50)--(4.15,2.50);
				\node[anchor=west,scale=0.50,inner sep=1pt] (A) at (4.15,2.50) {$6$};
				\draw[line width=2pt] (0.05,2.50)--(-0.15,2.50);
				\node[anchor=east,scale=0.50,inner sep=1pt] (A) at (-0.15,2.50) {$2$};
				\draw[line width=1.5pt,black!50] (3.95,3.00)--(5.05,3.00);\draw[line width=1.5pt,black!50] (5.00,2.95)--(5.00,4.05);\draw[line width=2pt] (4.50,3.05)--(4.50,2.85);
					\node[anchor=north,scale=0.50,inner sep=1pt] (A) at (4.50,2.85) {$7$};
					\draw[line width=2pt] (4.95,3.50)--(5.15,3.50);
				\node[anchor=west,scale=0.50,inner sep=1pt] (A) at (5.15,3.50) {$8$};
				\draw[line width=2pt] (0.05,3.50)--(-0.15,3.50);
				\node[anchor=east,scale=0.50,inner sep=1pt] (A) at (-0.15,3.50) {$3$};
				\draw[line width=2pt] (0.50,3.95)--(0.50,4.15);
		\node[anchor=south,scale=0.50,inner sep=1pt] (A) at (0.50,4.15) {$4$};
		\draw[line width=2pt] (1.50,3.95)--(1.50,4.15);
		\node[anchor=south,scale=0.50,inner sep=1pt] (A) at (1.50,4.15) {$5$};
		\draw[line width=2pt] (2.50,3.95)--(2.50,4.15);
		\node[anchor=south,scale=0.50,inner sep=1pt] (A) at (2.50,4.15) {$6$};
		\draw[line width=2pt] (3.50,3.95)--(3.50,4.15);
		\node[anchor=south,scale=0.50,inner sep=1pt] (A) at (3.50,4.15) {$7$};
		\draw[line width=2pt] (4.50,3.95)--(4.50,4.15);
		\node[anchor=south,scale=0.50,inner sep=1pt] (A) at (4.50,4.15) {$8$};
		\end{tikzpicture}
			}
&
\scalebox{\sclbx}{
	\begin{tikzpicture}[xscale=0.5,yscale=0.5]
	\draw[line width=0.5pt,dashed,opacity=\gridop] (0.00,1.00)--(2.00,1.00);\draw[line width=0.5pt,dashed,opacity=\gridop] (0.00,2.00)--(4.00,2.00);\draw[line width=0.5pt,dashed,opacity=\gridop] (0.00,3.00)--(5.00,3.00);\draw[line width=0.5pt,dashed,opacity=\gridop] (0.00,4.00)--(0.00,1.00);\draw[line width=0.5pt,dashed,opacity=\gridop] (1.00,4.00)--(1.00,1.00);\draw[line width=0.5pt,dashed,opacity=\gridop] (2.00,4.00)--(2.00,2.00);\draw[line width=0.5pt,dashed,opacity=\gridop] (3.00,4.00)--(3.00,2.00);\draw[line width=0.5pt,dashed,opacity=\gridop] (4.00,4.00)--(4.00,3.00);
	\godiagx{3/3,1/2,3/2,0/1}{2/3,4/3,2/2,1/1}{0/3,1/3,0/2}
	\draw[line width=1.5pt,black!50] (0.00,0.95)--(0.00,4.05);\draw[line width=1.5pt,black!50] (-0.05,4.00)--(5.05,4.00);\draw[line width=1.5pt,black!50] (-0.05,1.00)--(2.05,1.00);\draw[line width=1.5pt,black!50] (2.00,0.95)--(2.00,2.05);\draw[line width=2pt] (0.50,1.05)--(0.50,0.85);
					\node[anchor=north,scale=0.50,inner sep=1pt] (A) at (0.50,0.85) {$1$};
					\draw[line width=2pt] (1.50,1.05)--(1.50,0.85);
					\node[anchor=north,scale=0.50,inner sep=1pt] (A) at (1.50,0.85) {$2$};
					\draw[line width=2pt] (1.95,1.50)--(2.15,1.50);
				\node[anchor=west,scale=0.50,inner sep=1pt] (A) at (2.15,1.50) {$3$};
				\draw[line width=2pt] (0.05,1.50)--(-0.15,1.50);
				\node[anchor=east,scale=0.50,inner sep=1pt] (A) at (-0.15,1.50) {$1$};
				\draw[line width=1.5pt,black!50] (1.95,2.00)--(4.05,2.00);\draw[line width=1.5pt,black!50] (4.00,1.95)--(4.00,3.05);\draw[line width=2pt] (2.50,2.05)--(2.50,1.85);
					\node[anchor=north,scale=0.50,inner sep=1pt] (A) at (2.50,1.85) {$4$};
					\draw[line width=2pt] (3.50,2.05)--(3.50,1.85);
					\node[anchor=north,scale=0.50,inner sep=1pt] (A) at (3.50,1.85) {$5$};
					\draw[line width=2pt] (3.95,2.50)--(4.15,2.50);
				\node[anchor=west,scale=0.50,inner sep=1pt] (A) at (4.15,2.50) {$6$};
				\draw[line width=2pt] (0.05,2.50)--(-0.15,2.50);
				\node[anchor=east,scale=0.50,inner sep=1pt] (A) at (-0.15,2.50) {$2$};
				\draw[line width=1.5pt,black!50] (3.95,3.00)--(5.05,3.00);\draw[line width=1.5pt,black!50] (5.00,2.95)--(5.00,4.05);\draw[line width=2pt] (4.50,3.05)--(4.50,2.85);
					\node[anchor=north,scale=0.50,inner sep=1pt] (A) at (4.50,2.85) {$7$};
					\draw[line width=2pt] (4.95,3.50)--(5.15,3.50);
				\node[anchor=west,scale=0.50,inner sep=1pt] (A) at (5.15,3.50) {$8$};
				\draw[line width=2pt] (0.05,3.50)--(-0.15,3.50);
				\node[anchor=east,scale=0.50,inner sep=1pt] (A) at (-0.15,3.50) {$3$};
				\draw[line width=2pt] (0.50,3.95)--(0.50,4.15);
		\node[anchor=south,scale=0.50,inner sep=1pt] (A) at (0.50,4.15) {$4$};
		\draw[line width=2pt] (1.50,3.95)--(1.50,4.15);
		\node[anchor=south,scale=0.50,inner sep=1pt] (A) at (1.50,4.15) {$5$};
		\draw[line width=2pt] (2.50,3.95)--(2.50,4.15);
		\node[anchor=south,scale=0.50,inner sep=1pt] (A) at (2.50,4.15) {$6$};
		\draw[line width=2pt] (3.50,3.95)--(3.50,4.15);
		\node[anchor=south,scale=0.50,inner sep=1pt] (A) at (3.50,4.15) {$7$};
		\draw[line width=2pt] (4.50,3.95)--(4.50,4.15);
		\node[anchor=south,scale=0.50,inner sep=1pt] (A) at (4.50,4.15) {$8$};
		\end{tikzpicture}
			}
&
\scalebox{\sclbx}{
	\begin{tikzpicture}[xscale=0.5,yscale=0.5]
	\draw[line width=0.5pt,dashed,opacity=\gridop] (0.00,1.00)--(2.00,1.00);\draw[line width=0.5pt,dashed,opacity=\gridop] (0.00,2.00)--(4.00,2.00);\draw[line width=0.5pt,dashed,opacity=\gridop] (0.00,3.00)--(5.00,3.00);\draw[line width=0.5pt,dashed,opacity=\gridop] (0.00,4.00)--(0.00,1.00);\draw[line width=0.5pt,dashed,opacity=\gridop] (1.00,4.00)--(1.00,1.00);\draw[line width=0.5pt,dashed,opacity=\gridop] (2.00,4.00)--(2.00,2.00);\draw[line width=0.5pt,dashed,opacity=\gridop] (3.00,4.00)--(3.00,2.00);\draw[line width=0.5pt,dashed,opacity=\gridop] (4.00,4.00)--(4.00,3.00);
	\godiagx{1/3,3/3,1/2,0/1}{2/3,4/3,2/2,3/2,1/1}{0/3,0/2}
	\draw[line width=1.5pt,black!50] (0.00,0.95)--(0.00,4.05);\draw[line width=1.5pt,black!50] (-0.05,4.00)--(5.05,4.00);\draw[line width=1.5pt,black!50] (-0.05,1.00)--(2.05,1.00);\draw[line width=1.5pt,black!50] (2.00,0.95)--(2.00,2.05);\draw[line width=2pt] (0.50,1.05)--(0.50,0.85);
					\node[anchor=north,scale=0.50,inner sep=1pt] (A) at (0.50,0.85) {$1$};
					\draw[line width=2pt] (1.50,1.05)--(1.50,0.85);
					\node[anchor=north,scale=0.50,inner sep=1pt] (A) at (1.50,0.85) {$2$};
					\draw[line width=2pt] (1.95,1.50)--(2.15,1.50);
				\node[anchor=west,scale=0.50,inner sep=1pt] (A) at (2.15,1.50) {$3$};
				\draw[line width=2pt] (0.05,1.50)--(-0.15,1.50);
				\node[anchor=east,scale=0.50,inner sep=1pt] (A) at (-0.15,1.50) {$1$};
				\draw[line width=1.5pt,black!50] (1.95,2.00)--(4.05,2.00);\draw[line width=1.5pt,black!50] (4.00,1.95)--(4.00,3.05);\draw[line width=2pt] (2.50,2.05)--(2.50,1.85);
					\node[anchor=north,scale=0.50,inner sep=1pt] (A) at (2.50,1.85) {$4$};
					\draw[line width=2pt] (3.50,2.05)--(3.50,1.85);
					\node[anchor=north,scale=0.50,inner sep=1pt] (A) at (3.50,1.85) {$5$};
					\draw[line width=2pt] (3.95,2.50)--(4.15,2.50);
				\node[anchor=west,scale=0.50,inner sep=1pt] (A) at (4.15,2.50) {$6$};
				\draw[line width=2pt] (0.05,2.50)--(-0.15,2.50);
				\node[anchor=east,scale=0.50,inner sep=1pt] (A) at (-0.15,2.50) {$2$};
				\draw[line width=1.5pt,black!50] (3.95,3.00)--(5.05,3.00);\draw[line width=1.5pt,black!50] (5.00,2.95)--(5.00,4.05);\draw[line width=2pt] (4.50,3.05)--(4.50,2.85);
					\node[anchor=north,scale=0.50,inner sep=1pt] (A) at (4.50,2.85) {$7$};
					\draw[line width=2pt] (4.95,3.50)--(5.15,3.50);
				\node[anchor=west,scale=0.50,inner sep=1pt] (A) at (5.15,3.50) {$8$};
				\draw[line width=2pt] (0.05,3.50)--(-0.15,3.50);
				\node[anchor=east,scale=0.50,inner sep=1pt] (A) at (-0.15,3.50) {$3$};
				\draw[line width=2pt] (0.50,3.95)--(0.50,4.15);
		\node[anchor=south,scale=0.50,inner sep=1pt] (A) at (0.50,4.15) {$4$};
		\draw[line width=2pt] (1.50,3.95)--(1.50,4.15);
		\node[anchor=south,scale=0.50,inner sep=1pt] (A) at (1.50,4.15) {$5$};
		\draw[line width=2pt] (2.50,3.95)--(2.50,4.15);
		\node[anchor=south,scale=0.50,inner sep=1pt] (A) at (2.50,4.15) {$6$};
		\draw[line width=2pt] (3.50,3.95)--(3.50,4.15);
		\node[anchor=south,scale=0.50,inner sep=1pt] (A) at (3.50,4.15) {$7$};
		\draw[line width=2pt] (4.50,3.95)--(4.50,4.15);
		\node[anchor=south,scale=0.50,inner sep=1pt] (A) at (4.50,4.15) {$8$};
		\drawelbowfill{1}{3}{red!40}\end{tikzpicture}
			}
&
\scalebox{0.9}{
	\begin{tikzpicture}[xscale=0.5,yscale=0.5]
	\draw[line width=0.5pt,dashed,opacity=\gridop] (0.00,1.00)--(7.00,1.00);\draw[line width=0.5pt,dashed,opacity=\gridop] (0.00,2.00)--(7.00,2.00);\draw[line width=0.5pt,dashed,opacity=\gridop] (0.00,3.00)--(7.00,3.00);\draw[line width=0.5pt,dashed,opacity=\gridop] (0.00,4.00)--(7.00,4.00);\draw[line width=0.5pt,dashed,opacity=\gridop] (0.00,5.00)--(7.00,5.00);\draw[line width=0.5pt,dashed,opacity=\gridop] (0.00,6.00)--(0.00,1.00);\draw[line width=0.5pt,dashed,opacity=\gridop] (1.00,6.00)--(1.00,1.00);\draw[line width=0.5pt,dashed,opacity=\gridop] (2.00,6.00)--(2.00,1.00);\draw[line width=0.5pt,dashed,opacity=\gridop] (3.00,6.00)--(3.00,1.00);\draw[line width=0.5pt,dashed,opacity=\gridop] (4.00,6.00)--(4.00,1.00);\draw[line width=0.5pt,dashed,opacity=\gridop] (5.00,6.00)--(5.00,1.00);\draw[line width=0.5pt,dashed,opacity=\gridop] (6.00,6.00)--(6.00,1.00);
	\godiagx{3/5,6/5,0/4,5/4,2/3,3/3,4/3,1/2,6/2,0/1,3/1}{0/5,1/5,2/5,4/5,5/5,1/4,2/4,3/4,4/4,6/4,0/3,1/3,5/3,6/3,0/2,2/2,3/2,4/2,5/2,1/1,2/1,4/1,5/1,6/1}{}
	\draw[line width=1.5pt,black!50] (0.00,0.95)--(0.00,6.05);\draw[line width=1.5pt,black!50] (-0.05,6.00)--(7.05,6.00);\draw[line width=1.5pt,black!50] (-0.05,1.00)--(7.05,1.00);\draw[line width=1.5pt,black!50] (7.00,0.95)--(7.00,2.05);\draw[line width=2pt] (0.50,1.05)--(0.50,0.85);
					\node[anchor=north,scale=0.50,inner sep=1pt] (A) at (0.50,0.85) {$1$};
					\draw[line width=2pt] (1.50,1.05)--(1.50,0.85);
					\node[anchor=north,scale=0.50,inner sep=1pt] (A) at (1.50,0.85) {$2$};
					\draw[line width=2pt] (2.50,1.05)--(2.50,0.85);
					\node[anchor=north,scale=0.50,inner sep=1pt] (A) at (2.50,0.85) {$3$};
					\draw[line width=2pt] (3.50,1.05)--(3.50,0.85);
					\node[anchor=north,scale=0.50,inner sep=1pt] (A) at (3.50,0.85) {$4$};
					\draw[line width=2pt] (4.50,1.05)--(4.50,0.85);
					\node[anchor=north,scale=0.50,inner sep=1pt] (A) at (4.50,0.85) {$5$};
					\draw[line width=2pt] (5.50,1.05)--(5.50,0.85);
					\node[anchor=north,scale=0.50,inner sep=1pt] (A) at (5.50,0.85) {$6$};
					\draw[line width=2pt] (6.50,1.05)--(6.50,0.85);
					\node[anchor=north,scale=0.50,inner sep=1pt] (A) at (6.50,0.85) {$7$};
					\draw[line width=2pt] (6.95,1.50)--(7.15,1.50);
				\node[anchor=west,scale=0.50,inner sep=1pt] (A) at (7.15,1.50) {$8$};
				\draw[line width=2pt] (0.05,1.50)--(-0.15,1.50);
				\node[anchor=east,scale=0.50,inner sep=1pt] (A) at (-0.15,1.50) {$1$};
				\draw[line width=1.5pt,black!50] (6.95,2.00)--(7.05,2.00);\draw[line width=1.5pt,black!50] (7.00,1.95)--(7.00,3.05);\draw[line width=2pt] (6.95,2.50)--(7.15,2.50);
				\node[anchor=west,scale=0.50,inner sep=1pt] (A) at (7.15,2.50) {$9$};
				\draw[line width=2pt] (0.05,2.50)--(-0.15,2.50);
				\node[anchor=east,scale=0.50,inner sep=1pt] (A) at (-0.15,2.50) {$2$};
				\draw[line width=1.5pt,black!50] (6.95,3.00)--(7.05,3.00);\draw[line width=1.5pt,black!50] (7.00,2.95)--(7.00,4.05);\draw[line width=2pt] (6.95,3.50)--(7.15,3.50);
				\node[anchor=west,scale=0.50,inner sep=1pt] (A) at (7.15,3.50) {$10$};
				\draw[line width=2pt] (0.05,3.50)--(-0.15,3.50);
				\node[anchor=east,scale=0.50,inner sep=1pt] (A) at (-0.15,3.50) {$3$};
				\draw[line width=1.5pt,black!50] (6.95,4.00)--(7.05,4.00);\draw[line width=1.5pt,black!50] (7.00,3.95)--(7.00,5.05);\draw[line width=2pt] (6.95,4.50)--(7.15,4.50);
				\node[anchor=west,scale=0.50,inner sep=1pt] (A) at (7.15,4.50) {$11$};
				\draw[line width=2pt] (0.05,4.50)--(-0.15,4.50);
				\node[anchor=east,scale=0.50,inner sep=1pt] (A) at (-0.15,4.50) {$4$};
				\draw[line width=1.5pt,black!50] (6.95,5.00)--(7.05,5.00);\draw[line width=1.5pt,black!50] (7.00,4.95)--(7.00,6.05);\draw[line width=2pt] (6.95,5.50)--(7.15,5.50);
				\node[anchor=west,scale=0.50,inner sep=1pt] (A) at (7.15,5.50) {$12$};
				\draw[line width=2pt] (0.05,5.50)--(-0.15,5.50);
				\node[anchor=east,scale=0.50,inner sep=1pt] (A) at (-0.15,5.50) {$5$};
				\draw[line width=2pt] (0.50,5.95)--(0.50,6.15);
		\node[anchor=south,scale=0.50,inner sep=1pt] (A) at (0.50,6.15) {$6$};
		\draw[line width=2pt] (1.50,5.95)--(1.50,6.15);
		\node[anchor=south,scale=0.50,inner sep=1pt] (A) at (1.50,6.15) {$7$};
		\draw[line width=2pt] (2.50,5.95)--(2.50,6.15);
		\node[anchor=south,scale=0.50,inner sep=1pt] (A) at (2.50,6.15) {$8$};
		\draw[line width=2pt] (3.50,5.95)--(3.50,6.15);
		\node[anchor=south,scale=0.50,inner sep=1pt] (A) at (3.50,6.15) {$9$};
		\draw[line width=2pt] (4.50,5.95)--(4.50,6.15);
		\node[anchor=south,scale=0.50,inner sep=1pt] (A) at (4.50,6.15) {$10$};
		\draw[line width=2pt] (5.50,5.95)--(5.50,6.15);
		\node[anchor=south,scale=0.50,inner sep=1pt] (A) at (5.50,6.15) {$11$};
		\draw[line width=2pt] (6.50,5.95)--(6.50,6.15);
		\node[anchor=south,scale=0.50,inner sep=1pt] (A) at (6.50,6.15) {$12$};
		\drawelbowfill{3}{3}{red!40}\end{tikzpicture}
			}
\\
$w=45167283$ & $v=12465387$ & \textcolor{red}{not} distinguished & \textcolor{red}{not} distinguished
\end{tabular}
}
		}

  \caption{\label{fig:Go_ex} For the two fillings on the left, we have $\la=(5,4,2)$ and $f=wv^{-1}=35148276$. The two fillings on the right do not satisfy the distinguished condition: the specific elbow violating the condition is shaded in red.}
\end{figure}

Recall from \cref{prop:v_w} that each $f=\fvw\in\Bkn$ corresponds to a pair $v\leq w\in S_n$ such that $w$ is $k$-Grassmannian. The set of $k$-Grassmannian permutations in $S_n$ is well known to be in bijection with the set of Young diagrams that fit inside a $k\times (n-k)$-rectangle. Let $\la$ be such a Young diagram. We are going to consider fillings of boxes of $\la$ with \emph{crossings}~\crossing and \emph{elbows}~\elbow. An example is given in \cref{fig:Go_ex}. Each such filling $\fil$  gives rise to a permutation $u_\fil$, obtained as follows. Consider paths labeled by $1,2,\dots,n$ entering from the southeast boundary of $\la$, where the labels increase in the northeast direction. The paths follow crossings and elbows until they exit through the northwest boundary of $\la$. Recording the positions of outgoing edges, one obtains the permutation $u_\fil$ (cf. \cref{fig:Go_ex}).
\begin{definition}\label{dfn:Go}
Let $\la$ be a Young diagram fitting in a $k\times(n-k)$-rectangle. A \emph{Deogram} (short for \emph{Deodhar diagram}\footnote{The terminology \emph{Deodhar diagram} is borrowed from~\cite{KW}.}) of shape $\la$ is a filling $\fil$ of the boxes of $\la$ with crossings and elbows satisfying the following \emph{distinguished condition}~\cite{Deodhar}: for any elbow in $\fil$, the label of its bottom-left path is less than the label of its top-right path. In other words, once two paths have crossed an odd number of times, they cannot form an elbow. See \cref{fig:Go_ex}.
\end{definition}
\noindent For example, any filling that consists either entirely of crossings or entirely of elbows satisfies the distinguished condition. Observe that when a Deogram $\fil$ of shape $\la$ consists entirely of crossings, the permutation $u_\fil=w$ indeed is $k$-Grassmannian: we have $w^{-1}(1)<w^{-1}(2)<\dots<w^{-1}(k)$ and $w^{-1}(k+1)<\dots<w^{-1}(n)$. We denote this correspondence by $\la_w:=\la$.
\begin{definition}
Let $f=\fvw\in\Bkn$. An \emph{$f$-Deogram} is a Deogram $\fil$ of shape $\la_w$ satisfying $u_\fil=v$.
 A \emph{maximal $f$-Deogram} is an $f$-Deogram with the maximal possible number of crossings among all $f$-Deograms.

 We denote by $\Go_f$ (resp., $\Gom_f$) the set of all (resp., maximal) $f$-Deograms.
\end{definition}
\begin{remark}\label{rmk:Deo_elbows}
It is easy to see that any $f$-Deogram must have at least $n-\ncyc(\fmod)$ elbows. One can also check that for each $f\in\Bkn$, there exist $f$-Deograms with exactly $n-\ncyc(\fmod)$ elbows.\footnote{The same statement does not hold for Richardson varieties: for $w=s_1s_2s_3s_2s_1$ and $v=s_2$ in $S_4$, there are no subexpressions for $v$ inside $w$ skipping exactly $n-\ncyc(wv^{-1})=2$ indices.}
\end{remark}

\begin{figure}
\centering
\makebox[1.0\textwidth]{
\setlength{\tabcolsep}{1pt}
\scalebox{0.77}{
\begin{tabular}{ccccccc}
\scalebox{0.77}{
	\begin{tikzpicture}[xscale=0.5,yscale=0.5]
	\draw[line width=0.5pt,dashed,opacity=\gridop] (0.00,1.00)--(5.00,1.00);\draw[line width=0.5pt,dashed,opacity=\gridop] (0.00,2.00)--(5.00,2.00);\draw[line width=0.5pt,dashed,opacity=\gridop] (0.00,3.00)--(5.00,3.00);\draw[line width=0.5pt,dashed,opacity=\gridop] (0.00,4.00)--(0.00,1.00);\draw[line width=0.5pt,dashed,opacity=\gridop] (1.00,4.00)--(1.00,1.00);\draw[line width=0.5pt,dashed,opacity=\gridop] (2.00,4.00)--(2.00,1.00);\draw[line width=0.5pt,dashed,opacity=\gridop] (3.00,4.00)--(3.00,1.00);\draw[line width=0.5pt,dashed,opacity=\gridop] (4.00,4.00)--(4.00,1.00);
	\godiagx{2/3,3/3,4/3,1/2,4/2,0/1,3/1}{2/2,1/1,2/1,4/1}{0/3,1/3,0/2,3/2}
	\draw[line width=1.5pt,black!50] (0.00,0.95)--(0.00,4.05);\draw[line width=1.5pt,black!50] (-0.05,4.00)--(5.05,4.00);\draw[line width=1.5pt,black!50] (-0.05,1.00)--(5.05,1.00);\draw[line width=1.5pt,black!50] (5.00,0.95)--(5.00,2.05);\draw[line width=2pt] (0.50,1.05)--(0.50,0.85);
					\node[anchor=north,scale=0.50,inner sep=1pt] (A) at (0.50,0.85) {$$};
					\draw[line width=2pt] (1.50,1.05)--(1.50,0.85);
					\node[anchor=north,scale=0.50,inner sep=1pt] (A) at (1.50,0.85) {$$};
					\draw[line width=2pt] (2.50,1.05)--(2.50,0.85);
					\node[anchor=north,scale=0.50,inner sep=1pt] (A) at (2.50,0.85) {$$};
					\draw[line width=2pt] (3.50,1.05)--(3.50,0.85);
					\node[anchor=north,scale=0.50,inner sep=1pt] (A) at (3.50,0.85) {$$};
					\draw[line width=2pt] (4.50,1.05)--(4.50,0.85);
					\node[anchor=north,scale=0.50,inner sep=1pt] (A) at (4.50,0.85) {$$};
					\draw[line width=2pt] (4.95,1.50)--(5.15,1.50);
				\node[anchor=west,scale=0.50,inner sep=1pt] (A) at (5.15,1.50) {$$};
				\draw[line width=2pt] (0.05,1.50)--(-0.15,1.50);
				\node[anchor=east,scale=0.50,inner sep=1pt] (A) at (-0.15,1.50) {$$};
				\draw[line width=1.5pt,black!50] (4.95,2.00)--(5.05,2.00);\draw[line width=1.5pt,black!50] (5.00,1.95)--(5.00,3.05);\draw[line width=2pt] (4.95,2.50)--(5.15,2.50);
				\node[anchor=west,scale=0.50,inner sep=1pt] (A) at (5.15,2.50) {$$};
				\draw[line width=2pt] (0.05,2.50)--(-0.15,2.50);
				\node[anchor=east,scale=0.50,inner sep=1pt] (A) at (-0.15,2.50) {$$};
				\draw[line width=1.5pt,black!50] (4.95,3.00)--(5.05,3.00);\draw[line width=1.5pt,black!50] (5.00,2.95)--(5.00,4.05);\draw[line width=2pt] (4.95,3.50)--(5.15,3.50);
				\node[anchor=west,scale=0.50,inner sep=1pt] (A) at (5.15,3.50) {$$};
				\draw[line width=2pt] (0.05,3.50)--(-0.15,3.50);
				\node[anchor=east,scale=0.50,inner sep=1pt] (A) at (-0.15,3.50) {$$};
				\draw[line width=2pt] (0.50,3.95)--(0.50,4.15);
		\node[anchor=south,scale=0.50,inner sep=1pt] (A) at (0.50,4.15) {$$};
		\draw[line width=2pt] (1.50,3.95)--(1.50,4.15);
		\node[anchor=south,scale=0.50,inner sep=1pt] (A) at (1.50,4.15) {$$};
		\draw[line width=2pt] (2.50,3.95)--(2.50,4.15);
		\node[anchor=south,scale=0.50,inner sep=1pt] (A) at (2.50,4.15) {$$};
		\draw[line width=2pt] (3.50,3.95)--(3.50,4.15);
		\node[anchor=south,scale=0.50,inner sep=1pt] (A) at (3.50,4.15) {$$};
		\draw[line width=2pt] (4.50,3.95)--(4.50,4.15);
		\node[anchor=south,scale=0.50,inner sep=1pt] (A) at (4.50,4.15) {$$};
		\end{tikzpicture}
			}
&
\scalebox{0.77}{
	\begin{tikzpicture}[xscale=0.5,yscale=0.5]
	\draw[line width=0.5pt,dashed,opacity=\gridop] (0.00,1.00)--(5.00,1.00);\draw[line width=0.5pt,dashed,opacity=\gridop] (0.00,2.00)--(5.00,2.00);\draw[line width=0.5pt,dashed,opacity=\gridop] (0.00,3.00)--(5.00,3.00);\draw[line width=0.5pt,dashed,opacity=\gridop] (0.00,4.00)--(0.00,1.00);\draw[line width=0.5pt,dashed,opacity=\gridop] (1.00,4.00)--(1.00,1.00);\draw[line width=0.5pt,dashed,opacity=\gridop] (2.00,4.00)--(2.00,1.00);\draw[line width=0.5pt,dashed,opacity=\gridop] (3.00,4.00)--(3.00,1.00);\draw[line width=0.5pt,dashed,opacity=\gridop] (4.00,4.00)--(4.00,1.00);
	\godiagx{2/3,4/3,1/2,3/2,0/1,3/1,4/1}{2/2,4/2,1/1,2/1}{0/3,1/3,3/3,0/2}
	\draw[line width=1.5pt,black!50] (0.00,0.95)--(0.00,4.05);\draw[line width=1.5pt,black!50] (-0.05,4.00)--(5.05,4.00);\draw[line width=1.5pt,black!50] (-0.05,1.00)--(5.05,1.00);\draw[line width=1.5pt,black!50] (5.00,0.95)--(5.00,2.05);\draw[line width=2pt] (0.50,1.05)--(0.50,0.85);
					\node[anchor=north,scale=0.50,inner sep=1pt] (A) at (0.50,0.85) {$$};
					\draw[line width=2pt] (1.50,1.05)--(1.50,0.85);
					\node[anchor=north,scale=0.50,inner sep=1pt] (A) at (1.50,0.85) {$$};
					\draw[line width=2pt] (2.50,1.05)--(2.50,0.85);
					\node[anchor=north,scale=0.50,inner sep=1pt] (A) at (2.50,0.85) {$$};
					\draw[line width=2pt] (3.50,1.05)--(3.50,0.85);
					\node[anchor=north,scale=0.50,inner sep=1pt] (A) at (3.50,0.85) {$$};
					\draw[line width=2pt] (4.50,1.05)--(4.50,0.85);
					\node[anchor=north,scale=0.50,inner sep=1pt] (A) at (4.50,0.85) {$$};
					\draw[line width=2pt] (4.95,1.50)--(5.15,1.50);
				\node[anchor=west,scale=0.50,inner sep=1pt] (A) at (5.15,1.50) {$$};
				\draw[line width=2pt] (0.05,1.50)--(-0.15,1.50);
				\node[anchor=east,scale=0.50,inner sep=1pt] (A) at (-0.15,1.50) {$$};
				\draw[line width=1.5pt,black!50] (4.95,2.00)--(5.05,2.00);\draw[line width=1.5pt,black!50] (5.00,1.95)--(5.00,3.05);\draw[line width=2pt] (4.95,2.50)--(5.15,2.50);
				\node[anchor=west,scale=0.50,inner sep=1pt] (A) at (5.15,2.50) {$$};
				\draw[line width=2pt] (0.05,2.50)--(-0.15,2.50);
				\node[anchor=east,scale=0.50,inner sep=1pt] (A) at (-0.15,2.50) {$$};
				\draw[line width=1.5pt,black!50] (4.95,3.00)--(5.05,3.00);\draw[line width=1.5pt,black!50] (5.00,2.95)--(5.00,4.05);\draw[line width=2pt] (4.95,3.50)--(5.15,3.50);
				\node[anchor=west,scale=0.50,inner sep=1pt] (A) at (5.15,3.50) {$$};
				\draw[line width=2pt] (0.05,3.50)--(-0.15,3.50);
				\node[anchor=east,scale=0.50,inner sep=1pt] (A) at (-0.15,3.50) {$$};
				\draw[line width=2pt] (0.50,3.95)--(0.50,4.15);
		\node[anchor=south,scale=0.50,inner sep=1pt] (A) at (0.50,4.15) {$$};
		\draw[line width=2pt] (1.50,3.95)--(1.50,4.15);
		\node[anchor=south,scale=0.50,inner sep=1pt] (A) at (1.50,4.15) {$$};
		\draw[line width=2pt] (2.50,3.95)--(2.50,4.15);
		\node[anchor=south,scale=0.50,inner sep=1pt] (A) at (2.50,4.15) {$$};
		\draw[line width=2pt] (3.50,3.95)--(3.50,4.15);
		\node[anchor=south,scale=0.50,inner sep=1pt] (A) at (3.50,4.15) {$$};
		\draw[line width=2pt] (4.50,3.95)--(4.50,4.15);
		\node[anchor=south,scale=0.50,inner sep=1pt] (A) at (4.50,4.15) {$$};
		\end{tikzpicture}
			}
&
\scalebox{0.77}{
	\begin{tikzpicture}[xscale=0.5,yscale=0.5]
	\draw[line width=0.5pt,dashed,opacity=\gridop] (0.00,1.00)--(5.00,1.00);\draw[line width=0.5pt,dashed,opacity=\gridop] (0.00,2.00)--(5.00,2.00);\draw[line width=0.5pt,dashed,opacity=\gridop] (0.00,3.00)--(5.00,3.00);\draw[line width=0.5pt,dashed,opacity=\gridop] (0.00,4.00)--(0.00,1.00);\draw[line width=0.5pt,dashed,opacity=\gridop] (1.00,4.00)--(1.00,1.00);\draw[line width=0.5pt,dashed,opacity=\gridop] (2.00,4.00)--(2.00,1.00);\draw[line width=0.5pt,dashed,opacity=\gridop] (3.00,4.00)--(3.00,1.00);\draw[line width=0.5pt,dashed,opacity=\gridop] (4.00,4.00)--(4.00,1.00);
	\godiagx{1/3,4/3,1/2,2/2,3/2,0/1,3/1}{4/2,1/1,2/1,4/1}{0/3,2/3,3/3,0/2}
	\draw[line width=1.5pt,black!50] (0.00,0.95)--(0.00,4.05);\draw[line width=1.5pt,black!50] (-0.05,4.00)--(5.05,4.00);\draw[line width=1.5pt,black!50] (-0.05,1.00)--(5.05,1.00);\draw[line width=1.5pt,black!50] (5.00,0.95)--(5.00,2.05);\draw[line width=2pt] (0.50,1.05)--(0.50,0.85);
					\node[anchor=north,scale=0.50,inner sep=1pt] (A) at (0.50,0.85) {$$};
					\draw[line width=2pt] (1.50,1.05)--(1.50,0.85);
					\node[anchor=north,scale=0.50,inner sep=1pt] (A) at (1.50,0.85) {$$};
					\draw[line width=2pt] (2.50,1.05)--(2.50,0.85);
					\node[anchor=north,scale=0.50,inner sep=1pt] (A) at (2.50,0.85) {$$};
					\draw[line width=2pt] (3.50,1.05)--(3.50,0.85);
					\node[anchor=north,scale=0.50,inner sep=1pt] (A) at (3.50,0.85) {$$};
					\draw[line width=2pt] (4.50,1.05)--(4.50,0.85);
					\node[anchor=north,scale=0.50,inner sep=1pt] (A) at (4.50,0.85) {$$};
					\draw[line width=2pt] (4.95,1.50)--(5.15,1.50);
				\node[anchor=west,scale=0.50,inner sep=1pt] (A) at (5.15,1.50) {$$};
				\draw[line width=2pt] (0.05,1.50)--(-0.15,1.50);
				\node[anchor=east,scale=0.50,inner sep=1pt] (A) at (-0.15,1.50) {$$};
				\draw[line width=1.5pt,black!50] (4.95,2.00)--(5.05,2.00);\draw[line width=1.5pt,black!50] (5.00,1.95)--(5.00,3.05);\draw[line width=2pt] (4.95,2.50)--(5.15,2.50);
				\node[anchor=west,scale=0.50,inner sep=1pt] (A) at (5.15,2.50) {$$};
				\draw[line width=2pt] (0.05,2.50)--(-0.15,2.50);
				\node[anchor=east,scale=0.50,inner sep=1pt] (A) at (-0.15,2.50) {$$};
				\draw[line width=1.5pt,black!50] (4.95,3.00)--(5.05,3.00);\draw[line width=1.5pt,black!50] (5.00,2.95)--(5.00,4.05);\draw[line width=2pt] (4.95,3.50)--(5.15,3.50);
				\node[anchor=west,scale=0.50,inner sep=1pt] (A) at (5.15,3.50) {$$};
				\draw[line width=2pt] (0.05,3.50)--(-0.15,3.50);
				\node[anchor=east,scale=0.50,inner sep=1pt] (A) at (-0.15,3.50) {$$};
				\draw[line width=2pt] (0.50,3.95)--(0.50,4.15);
		\node[anchor=south,scale=0.50,inner sep=1pt] (A) at (0.50,4.15) {$$};
		\draw[line width=2pt] (1.50,3.95)--(1.50,4.15);
		\node[anchor=south,scale=0.50,inner sep=1pt] (A) at (1.50,4.15) {$$};
		\draw[line width=2pt] (2.50,3.95)--(2.50,4.15);
		\node[anchor=south,scale=0.50,inner sep=1pt] (A) at (2.50,4.15) {$$};
		\draw[line width=2pt] (3.50,3.95)--(3.50,4.15);
		\node[anchor=south,scale=0.50,inner sep=1pt] (A) at (3.50,4.15) {$$};
		\draw[line width=2pt] (4.50,3.95)--(4.50,4.15);
		\node[anchor=south,scale=0.50,inner sep=1pt] (A) at (4.50,4.15) {$$};
		\end{tikzpicture}
			}
&
\scalebox{0.77}{
	\begin{tikzpicture}[xscale=0.5,yscale=0.5]
	\draw[line width=0.5pt,dashed,opacity=\gridop] (0.00,1.00)--(5.00,1.00);\draw[line width=0.5pt,dashed,opacity=\gridop] (0.00,2.00)--(5.00,2.00);\draw[line width=0.5pt,dashed,opacity=\gridop] (0.00,3.00)--(5.00,3.00);\draw[line width=0.5pt,dashed,opacity=\gridop] (0.00,4.00)--(0.00,1.00);\draw[line width=0.5pt,dashed,opacity=\gridop] (1.00,4.00)--(1.00,1.00);\draw[line width=0.5pt,dashed,opacity=\gridop] (2.00,4.00)--(2.00,1.00);\draw[line width=0.5pt,dashed,opacity=\gridop] (3.00,4.00)--(3.00,1.00);\draw[line width=0.5pt,dashed,opacity=\gridop] (4.00,4.00)--(4.00,1.00);
	\godiagx{0/3,3/3,4/3,1/2,4/2,0/1,2/1}{3/2,1/1,3/1,4/1}{1/3,2/3,0/2,2/2}
	\draw[line width=1.5pt,black!50] (0.00,0.95)--(0.00,4.05);\draw[line width=1.5pt,black!50] (-0.05,4.00)--(5.05,4.00);\draw[line width=1.5pt,black!50] (-0.05,1.00)--(5.05,1.00);\draw[line width=1.5pt,black!50] (5.00,0.95)--(5.00,2.05);\draw[line width=2pt] (0.50,1.05)--(0.50,0.85);
					\node[anchor=north,scale=0.50,inner sep=1pt] (A) at (0.50,0.85) {$$};
					\draw[line width=2pt] (1.50,1.05)--(1.50,0.85);
					\node[anchor=north,scale=0.50,inner sep=1pt] (A) at (1.50,0.85) {$$};
					\draw[line width=2pt] (2.50,1.05)--(2.50,0.85);
					\node[anchor=north,scale=0.50,inner sep=1pt] (A) at (2.50,0.85) {$$};
					\draw[line width=2pt] (3.50,1.05)--(3.50,0.85);
					\node[anchor=north,scale=0.50,inner sep=1pt] (A) at (3.50,0.85) {$$};
					\draw[line width=2pt] (4.50,1.05)--(4.50,0.85);
					\node[anchor=north,scale=0.50,inner sep=1pt] (A) at (4.50,0.85) {$$};
					\draw[line width=2pt] (4.95,1.50)--(5.15,1.50);
				\node[anchor=west,scale=0.50,inner sep=1pt] (A) at (5.15,1.50) {$$};
				\draw[line width=2pt] (0.05,1.50)--(-0.15,1.50);
				\node[anchor=east,scale=0.50,inner sep=1pt] (A) at (-0.15,1.50) {$$};
				\draw[line width=1.5pt,black!50] (4.95,2.00)--(5.05,2.00);\draw[line width=1.5pt,black!50] (5.00,1.95)--(5.00,3.05);\draw[line width=2pt] (4.95,2.50)--(5.15,2.50);
				\node[anchor=west,scale=0.50,inner sep=1pt] (A) at (5.15,2.50) {$$};
				\draw[line width=2pt] (0.05,2.50)--(-0.15,2.50);
				\node[anchor=east,scale=0.50,inner sep=1pt] (A) at (-0.15,2.50) {$$};
				\draw[line width=1.5pt,black!50] (4.95,3.00)--(5.05,3.00);\draw[line width=1.5pt,black!50] (5.00,2.95)--(5.00,4.05);\draw[line width=2pt] (4.95,3.50)--(5.15,3.50);
				\node[anchor=west,scale=0.50,inner sep=1pt] (A) at (5.15,3.50) {$$};
				\draw[line width=2pt] (0.05,3.50)--(-0.15,3.50);
				\node[anchor=east,scale=0.50,inner sep=1pt] (A) at (-0.15,3.50) {$$};
				\draw[line width=2pt] (0.50,3.95)--(0.50,4.15);
		\node[anchor=south,scale=0.50,inner sep=1pt] (A) at (0.50,4.15) {$$};
		\draw[line width=2pt] (1.50,3.95)--(1.50,4.15);
		\node[anchor=south,scale=0.50,inner sep=1pt] (A) at (1.50,4.15) {$$};
		\draw[line width=2pt] (2.50,3.95)--(2.50,4.15);
		\node[anchor=south,scale=0.50,inner sep=1pt] (A) at (2.50,4.15) {$$};
		\draw[line width=2pt] (3.50,3.95)--(3.50,4.15);
		\node[anchor=south,scale=0.50,inner sep=1pt] (A) at (3.50,4.15) {$$};
		\draw[line width=2pt] (4.50,3.95)--(4.50,4.15);
		\node[anchor=south,scale=0.50,inner sep=1pt] (A) at (4.50,4.15) {$$};
		\end{tikzpicture}
			}
&
\scalebox{0.77}{
	\begin{tikzpicture}[xscale=0.5,yscale=0.5]
	\draw[line width=0.5pt,dashed,opacity=\gridop] (0.00,1.00)--(5.00,1.00);\draw[line width=0.5pt,dashed,opacity=\gridop] (0.00,2.00)--(5.00,2.00);\draw[line width=0.5pt,dashed,opacity=\gridop] (0.00,3.00)--(5.00,3.00);\draw[line width=0.5pt,dashed,opacity=\gridop] (0.00,4.00)--(0.00,1.00);\draw[line width=0.5pt,dashed,opacity=\gridop] (1.00,4.00)--(1.00,1.00);\draw[line width=0.5pt,dashed,opacity=\gridop] (2.00,4.00)--(2.00,1.00);\draw[line width=0.5pt,dashed,opacity=\gridop] (3.00,4.00)--(3.00,1.00);\draw[line width=0.5pt,dashed,opacity=\gridop] (4.00,4.00)--(4.00,1.00);
	\godiagx{0/3,1/3,4/3,1/2,3/2,0/1,2/1}{4/2,1/1,3/1,4/1}{2/3,3/3,0/2,2/2}
	\draw[line width=1.5pt,black!50] (0.00,0.95)--(0.00,4.05);\draw[line width=1.5pt,black!50] (-0.05,4.00)--(5.05,4.00);\draw[line width=1.5pt,black!50] (-0.05,1.00)--(5.05,1.00);\draw[line width=1.5pt,black!50] (5.00,0.95)--(5.00,2.05);\draw[line width=2pt] (0.50,1.05)--(0.50,0.85);
					\node[anchor=north,scale=0.50,inner sep=1pt] (A) at (0.50,0.85) {$$};
					\draw[line width=2pt] (1.50,1.05)--(1.50,0.85);
					\node[anchor=north,scale=0.50,inner sep=1pt] (A) at (1.50,0.85) {$$};
					\draw[line width=2pt] (2.50,1.05)--(2.50,0.85);
					\node[anchor=north,scale=0.50,inner sep=1pt] (A) at (2.50,0.85) {$$};
					\draw[line width=2pt] (3.50,1.05)--(3.50,0.85);
					\node[anchor=north,scale=0.50,inner sep=1pt] (A) at (3.50,0.85) {$$};
					\draw[line width=2pt] (4.50,1.05)--(4.50,0.85);
					\node[anchor=north,scale=0.50,inner sep=1pt] (A) at (4.50,0.85) {$$};
					\draw[line width=2pt] (4.95,1.50)--(5.15,1.50);
				\node[anchor=west,scale=0.50,inner sep=1pt] (A) at (5.15,1.50) {$$};
				\draw[line width=2pt] (0.05,1.50)--(-0.15,1.50);
				\node[anchor=east,scale=0.50,inner sep=1pt] (A) at (-0.15,1.50) {$$};
				\draw[line width=1.5pt,black!50] (4.95,2.00)--(5.05,2.00);\draw[line width=1.5pt,black!50] (5.00,1.95)--(5.00,3.05);\draw[line width=2pt] (4.95,2.50)--(5.15,2.50);
				\node[anchor=west,scale=0.50,inner sep=1pt] (A) at (5.15,2.50) {$$};
				\draw[line width=2pt] (0.05,2.50)--(-0.15,2.50);
				\node[anchor=east,scale=0.50,inner sep=1pt] (A) at (-0.15,2.50) {$$};
				\draw[line width=1.5pt,black!50] (4.95,3.00)--(5.05,3.00);\draw[line width=1.5pt,black!50] (5.00,2.95)--(5.00,4.05);\draw[line width=2pt] (4.95,3.50)--(5.15,3.50);
				\node[anchor=west,scale=0.50,inner sep=1pt] (A) at (5.15,3.50) {$$};
				\draw[line width=2pt] (0.05,3.50)--(-0.15,3.50);
				\node[anchor=east,scale=0.50,inner sep=1pt] (A) at (-0.15,3.50) {$$};
				\draw[line width=2pt] (0.50,3.95)--(0.50,4.15);
		\node[anchor=south,scale=0.50,inner sep=1pt] (A) at (0.50,4.15) {$$};
		\draw[line width=2pt] (1.50,3.95)--(1.50,4.15);
		\node[anchor=south,scale=0.50,inner sep=1pt] (A) at (1.50,4.15) {$$};
		\draw[line width=2pt] (2.50,3.95)--(2.50,4.15);
		\node[anchor=south,scale=0.50,inner sep=1pt] (A) at (2.50,4.15) {$$};
		\draw[line width=2pt] (3.50,3.95)--(3.50,4.15);
		\node[anchor=south,scale=0.50,inner sep=1pt] (A) at (3.50,4.15) {$$};
		\draw[line width=2pt] (4.50,3.95)--(4.50,4.15);
		\node[anchor=south,scale=0.50,inner sep=1pt] (A) at (4.50,4.15) {$$};
		\end{tikzpicture}
			}
&
\scalebox{0.77}{
	\begin{tikzpicture}[xscale=0.5,yscale=0.5]
	\draw[line width=0.5pt,dashed,opacity=\gridop] (0.00,1.00)--(5.00,1.00);\draw[line width=0.5pt,dashed,opacity=\gridop] (0.00,2.00)--(5.00,2.00);\draw[line width=0.5pt,dashed,opacity=\gridop] (0.00,3.00)--(5.00,3.00);\draw[line width=0.5pt,dashed,opacity=\gridop] (0.00,4.00)--(0.00,1.00);\draw[line width=0.5pt,dashed,opacity=\gridop] (1.00,4.00)--(1.00,1.00);\draw[line width=0.5pt,dashed,opacity=\gridop] (2.00,4.00)--(2.00,1.00);\draw[line width=0.5pt,dashed,opacity=\gridop] (3.00,4.00)--(3.00,1.00);\draw[line width=0.5pt,dashed,opacity=\gridop] (4.00,4.00)--(4.00,1.00);
	\godiagx{2/3,4/3,0/2,3/2,0/1,1/1,4/1}{2/2,4/2,2/1,3/1}{0/3,1/3,3/3,1/2}
	\draw[line width=1.5pt,black!50] (0.00,0.95)--(0.00,4.05);\draw[line width=1.5pt,black!50] (-0.05,4.00)--(5.05,4.00);\draw[line width=1.5pt,black!50] (-0.05,1.00)--(5.05,1.00);\draw[line width=1.5pt,black!50] (5.00,0.95)--(5.00,2.05);\draw[line width=2pt] (0.50,1.05)--(0.50,0.85);
					\node[anchor=north,scale=0.50,inner sep=1pt] (A) at (0.50,0.85) {$$};
					\draw[line width=2pt] (1.50,1.05)--(1.50,0.85);
					\node[anchor=north,scale=0.50,inner sep=1pt] (A) at (1.50,0.85) {$$};
					\draw[line width=2pt] (2.50,1.05)--(2.50,0.85);
					\node[anchor=north,scale=0.50,inner sep=1pt] (A) at (2.50,0.85) {$$};
					\draw[line width=2pt] (3.50,1.05)--(3.50,0.85);
					\node[anchor=north,scale=0.50,inner sep=1pt] (A) at (3.50,0.85) {$$};
					\draw[line width=2pt] (4.50,1.05)--(4.50,0.85);
					\node[anchor=north,scale=0.50,inner sep=1pt] (A) at (4.50,0.85) {$$};
					\draw[line width=2pt] (4.95,1.50)--(5.15,1.50);
				\node[anchor=west,scale=0.50,inner sep=1pt] (A) at (5.15,1.50) {$$};
				\draw[line width=2pt] (0.05,1.50)--(-0.15,1.50);
				\node[anchor=east,scale=0.50,inner sep=1pt] (A) at (-0.15,1.50) {$$};
				\draw[line width=1.5pt,black!50] (4.95,2.00)--(5.05,2.00);\draw[line width=1.5pt,black!50] (5.00,1.95)--(5.00,3.05);\draw[line width=2pt] (4.95,2.50)--(5.15,2.50);
				\node[anchor=west,scale=0.50,inner sep=1pt] (A) at (5.15,2.50) {$$};
				\draw[line width=2pt] (0.05,2.50)--(-0.15,2.50);
				\node[anchor=east,scale=0.50,inner sep=1pt] (A) at (-0.15,2.50) {$$};
				\draw[line width=1.5pt,black!50] (4.95,3.00)--(5.05,3.00);\draw[line width=1.5pt,black!50] (5.00,2.95)--(5.00,4.05);\draw[line width=2pt] (4.95,3.50)--(5.15,3.50);
				\node[anchor=west,scale=0.50,inner sep=1pt] (A) at (5.15,3.50) {$$};
				\draw[line width=2pt] (0.05,3.50)--(-0.15,3.50);
				\node[anchor=east,scale=0.50,inner sep=1pt] (A) at (-0.15,3.50) {$$};
				\draw[line width=2pt] (0.50,3.95)--(0.50,4.15);
		\node[anchor=south,scale=0.50,inner sep=1pt] (A) at (0.50,4.15) {$$};
		\draw[line width=2pt] (1.50,3.95)--(1.50,4.15);
		\node[anchor=south,scale=0.50,inner sep=1pt] (A) at (1.50,4.15) {$$};
		\draw[line width=2pt] (2.50,3.95)--(2.50,4.15);
		\node[anchor=south,scale=0.50,inner sep=1pt] (A) at (2.50,4.15) {$$};
		\draw[line width=2pt] (3.50,3.95)--(3.50,4.15);
		\node[anchor=south,scale=0.50,inner sep=1pt] (A) at (3.50,4.15) {$$};
		\draw[line width=2pt] (4.50,3.95)--(4.50,4.15);
		\node[anchor=south,scale=0.50,inner sep=1pt] (A) at (4.50,4.15) {$$};
		\end{tikzpicture}
			}
&
\scalebox{0.77}{
	\begin{tikzpicture}[xscale=0.5,yscale=0.5]
	\draw[line width=0.5pt,dashed,opacity=\gridop] (0.00,1.00)--(5.00,1.00);\draw[line width=0.5pt,dashed,opacity=\gridop] (0.00,2.00)--(5.00,2.00);\draw[line width=0.5pt,dashed,opacity=\gridop] (0.00,3.00)--(5.00,3.00);\draw[line width=0.5pt,dashed,opacity=\gridop] (0.00,4.00)--(0.00,1.00);\draw[line width=0.5pt,dashed,opacity=\gridop] (1.00,4.00)--(1.00,1.00);\draw[line width=0.5pt,dashed,opacity=\gridop] (2.00,4.00)--(2.00,1.00);\draw[line width=0.5pt,dashed,opacity=\gridop] (3.00,4.00)--(3.00,1.00);\draw[line width=0.5pt,dashed,opacity=\gridop] (4.00,4.00)--(4.00,1.00);
	\godiagx{1/3,4/3,0/2,3/2,0/1,1/1,2/1}{1/2,4/2,3/1,4/1}{0/3,2/3,3/3,2/2}
	\draw[line width=1.5pt,black!50] (0.00,0.95)--(0.00,4.05);\draw[line width=1.5pt,black!50] (-0.05,4.00)--(5.05,4.00);\draw[line width=1.5pt,black!50] (-0.05,1.00)--(5.05,1.00);\draw[line width=1.5pt,black!50] (5.00,0.95)--(5.00,2.05);\draw[line width=2pt] (0.50,1.05)--(0.50,0.85);
					\node[anchor=north,scale=0.50,inner sep=1pt] (A) at (0.50,0.85) {$$};
					\draw[line width=2pt] (1.50,1.05)--(1.50,0.85);
					\node[anchor=north,scale=0.50,inner sep=1pt] (A) at (1.50,0.85) {$$};
					\draw[line width=2pt] (2.50,1.05)--(2.50,0.85);
					\node[anchor=north,scale=0.50,inner sep=1pt] (A) at (2.50,0.85) {$$};
					\draw[line width=2pt] (3.50,1.05)--(3.50,0.85);
					\node[anchor=north,scale=0.50,inner sep=1pt] (A) at (3.50,0.85) {$$};
					\draw[line width=2pt] (4.50,1.05)--(4.50,0.85);
					\node[anchor=north,scale=0.50,inner sep=1pt] (A) at (4.50,0.85) {$$};
					\draw[line width=2pt] (4.95,1.50)--(5.15,1.50);
				\node[anchor=west,scale=0.50,inner sep=1pt] (A) at (5.15,1.50) {$$};
				\draw[line width=2pt] (0.05,1.50)--(-0.15,1.50);
				\node[anchor=east,scale=0.50,inner sep=1pt] (A) at (-0.15,1.50) {$$};
				\draw[line width=1.5pt,black!50] (4.95,2.00)--(5.05,2.00);\draw[line width=1.5pt,black!50] (5.00,1.95)--(5.00,3.05);\draw[line width=2pt] (4.95,2.50)--(5.15,2.50);
				\node[anchor=west,scale=0.50,inner sep=1pt] (A) at (5.15,2.50) {$$};
				\draw[line width=2pt] (0.05,2.50)--(-0.15,2.50);
				\node[anchor=east,scale=0.50,inner sep=1pt] (A) at (-0.15,2.50) {$$};
				\draw[line width=1.5pt,black!50] (4.95,3.00)--(5.05,3.00);\draw[line width=1.5pt,black!50] (5.00,2.95)--(5.00,4.05);\draw[line width=2pt] (4.95,3.50)--(5.15,3.50);
				\node[anchor=west,scale=0.50,inner sep=1pt] (A) at (5.15,3.50) {$$};
				\draw[line width=2pt] (0.05,3.50)--(-0.15,3.50);
				\node[anchor=east,scale=0.50,inner sep=1pt] (A) at (-0.15,3.50) {$$};
				\draw[line width=2pt] (0.50,3.95)--(0.50,4.15);
		\node[anchor=south,scale=0.50,inner sep=1pt] (A) at (0.50,4.15) {$$};
		\draw[line width=2pt] (1.50,3.95)--(1.50,4.15);
		\node[anchor=south,scale=0.50,inner sep=1pt] (A) at (1.50,4.15) {$$};
		\draw[line width=2pt] (2.50,3.95)--(2.50,4.15);
		\node[anchor=south,scale=0.50,inner sep=1pt] (A) at (2.50,4.15) {$$};
		\draw[line width=2pt] (3.50,3.95)--(3.50,4.15);
		\node[anchor=south,scale=0.50,inner sep=1pt] (A) at (3.50,4.15) {$$};
		\draw[line width=2pt] (4.50,3.95)--(4.50,4.15);
		\node[anchor=south,scale=0.50,inner sep=1pt] (A) at (4.50,4.15) {$$};
		\end{tikzpicture}
			}\end{tabular}
}
}
  \caption{\label{fig:Go_3_8} The sets $\Gom_\fkn$ and $\Dyck_{k,n-k}$ have the same cardinality by \cref{prop:Go}. Compare with \cref{fig:Cab_ex}.}
\end{figure}
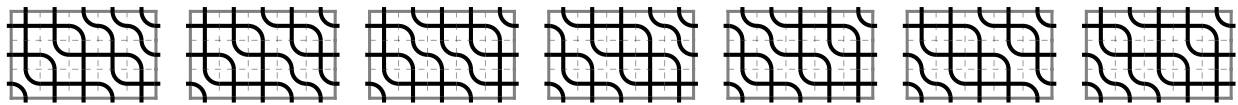

\begin{proposition}\label{prop:Go}
Let $f\in\Bknc$. Then $\Cat_f$ equals the number of maximal $f$-Deograms:
\begin{equation}\label{eq:Cat=Go}
  \Cat_f=\#\Gom_f.
\end{equation}
\end{proposition}
\begin{proof}
This is a simple consequence of the results of Deodhar~\cite{Deodhar}. Let $v,w$ be such that $f=\fvw$. By \cref{prop:KLS}, we have $\#\Pio_f(\F_q)=\#\Rich_v^w(\F_q)$. Deodhar expressed $\#\Rich_v^w(\F_q)$ as a certain sum over \emph{distinguished subexpressions} for $v$ inside a reduced word $w=s_{i_1}s_{i_2}\cdots s_{i_l}$, where $l=\ell(w)$. Here, a \emph{subexpression} for $v$ is a way to write $v$ as a product $\sv_{1}\sv_{2}\cdots \sv_{l}$, where $\sv_{j}\in\{s_{i_j},\id\}$ for all $j=1,2,\dots,l$. A subexpression is \emph{distinguished} if for all $j$ such that $\ell(\sv_{1}\cdots\sv_{j-1}s_{i_j})<\ell(\sv_{1}\cdots\sv_{j-1})$, we have $\sv_{j}=s_{i_j}$. Since $w$ is $k$-Grassmannian, the terms in the product $w=s_{i_1}s_{i_2}\cdots s_{i_l}$ correspond to the boxes of $\la_w$. Each Deogram $\fil\in\Go_f$ gives rise to a distinguished subexpression for $v$, so that  the indices $j$ such that $\sv_j=s_{i_j}$ correspond to the crossings in $\fil$. It is easy to see that this correspondence is bijective. Thus, the results of~\cite{Deodhar} imply that
\begin{equation}\label{eq:Deodhar}
  \#\Pio_f(\F_q)=\sum_{\fil\in\Go_f} (q-1)^{\elb(\fil)}q^{(\xing(\fil)-\ell(v))/2},
\end{equation}
where $\elb(\fil)$ and $\xing(\fil)$ denote the number of elbows and crossings in $\fil$. By \cref{rmk:Deo_elbows}, each maximal $f$-Deogram contributes $(q-1)^{n-1}q^{(\lvw-n+1)/2}$ to the right-hand side of~\eqref{eq:Deodhar}. (Note that $\xing(\fil)+\elb(\fil)=\ell(w)$ is constant.) It remains to note that $\#\Pit_f(\F_q)$ is obtained by dividing $\#\Pio_f(\F_q)$ by $\#T(\F_q)=(q-1)^{n-1}$, and that $\Cat_f$ is by definition the $q=1$ specialization of $\#\Pit_f(\F_q)$.
\end{proof}

Let us focus on the case $f=\fkn$ with $\gcd(k,n)=1$. Explicitly, a maximal $\fkn$-Deogram is a way of placing $n-1$ elbows in a $k\times(n-k)$ rectangle and filling the rest with crossings so that (i) the resulting permutation obtained by following the paths is the identity, and (ii) the distinguished condition in \cref{dfn:Go} is satisfied. By \cref{prop:Go}, the number $\#\Gom_\fkn$ of such objects equals the number $\#\Dyck_{k,(n-k)}$ of Dyck paths inside a $k\times (n-k)$-rectangle.
\begin{problem}\label{prob:catalan}
Give a bijective proof of \cref{prop:Go}. That is, find a bijection between $\Gom_\fkn$ and $\Dyck_{k,(n-k)}$ for the case $\gcd(k,n)=1$.
\end{problem}
\noindent For instance, \cref{fig:Cab_ex,fig:Go_3_8} both have $7$ objects in them, but it is unclear which objects correspond to which. It would be also interesting to understand the $\area$ and $\dinv$ statistics in the language of $\fkn$-Deograms.

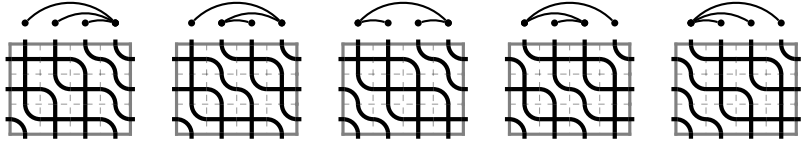
\begin{figure}
\centering
\makebox[1.0\textwidth]{
\setlength{\tabcolsep}{3pt}
\begin{tabular}{ccccc}
\scalebox{0.80}{
	\begin{tikzpicture}[xscale=0.5,yscale=0.5]
	\draw[line width=0.5pt,dashed,opacity=\gridop] (0.00,1.00)--(4.00,1.00);\draw[line width=0.5pt,dashed,opacity=\gridop] (0.00,2.00)--(4.00,2.00);\draw[line width=0.5pt,dashed,opacity=\gridop] (0.00,3.00)--(4.00,3.00);\draw[line width=0.5pt,dashed,opacity=\gridop] (0.00,4.00)--(0.00,1.00);\draw[line width=0.5pt,dashed,opacity=\gridop] (1.00,4.00)--(1.00,1.00);\draw[line width=0.5pt,dashed,opacity=\gridop] (2.00,4.00)--(2.00,1.00);\draw[line width=0.5pt,dashed,opacity=\gridop] (3.00,4.00)--(3.00,1.00);
	\godiagx{2/3,3/3,1/2,3/2,0/1,3/1}{2/2,1/1,2/1}{0/3,1/3,0/2}
	\draw[line width=1.5pt,black!50] (0.00,0.95)--(0.00,4.05);\draw[line width=1.5pt,black!50] (-0.05,4.00)--(4.05,4.00);\draw[line width=1.5pt,black!50] (-0.05,1.00)--(4.05,1.00);\draw[line width=1.5pt,black!50] (4.00,0.95)--(4.00,2.05);\draw[line width=2pt] (0.50,1.05)--(0.50,0.85);
					\node[anchor=north,scale=0.50,inner sep=1pt] (A) at (0.50,0.85) {$$};
					\draw[line width=2pt] (1.50,1.05)--(1.50,0.85);
					\node[anchor=north,scale=0.50,inner sep=1pt] (A) at (1.50,0.85) {$$};
					\draw[line width=2pt] (2.50,1.05)--(2.50,0.85);
					\node[anchor=north,scale=0.50,inner sep=1pt] (A) at (2.50,0.85) {$$};
					\draw[line width=2pt] (3.50,1.05)--(3.50,0.85);
					\node[anchor=north,scale=0.50,inner sep=1pt] (A) at (3.50,0.85) {$$};
					\draw[line width=2pt] (3.95,1.50)--(4.15,1.50);
				\node[anchor=west,scale=0.50,inner sep=1pt] (A) at (4.15,1.50) {$$};
				\draw[line width=2pt] (0.05,1.50)--(-0.15,1.50);
				\node[anchor=east,scale=0.50,inner sep=1pt] (A) at (-0.15,1.50) {$$};
				\draw[line width=1.5pt,black!50] (3.95,2.00)--(4.05,2.00);\draw[line width=1.5pt,black!50] (4.00,1.95)--(4.00,3.05);\draw[line width=2pt] (3.95,2.50)--(4.15,2.50);
				\node[anchor=west,scale=0.50,inner sep=1pt] (A) at (4.15,2.50) {$$};
				\draw[line width=2pt] (0.05,2.50)--(-0.15,2.50);
				\node[anchor=east,scale=0.50,inner sep=1pt] (A) at (-0.15,2.50) {$$};
				\draw[line width=1.5pt,black!50] (3.95,3.00)--(4.05,3.00);\draw[line width=1.5pt,black!50] (4.00,2.95)--(4.00,4.05);\draw[line width=2pt] (3.95,3.50)--(4.15,3.50);
				\node[anchor=west,scale=0.50,inner sep=1pt] (A) at (4.15,3.50) {$$};
				\draw[line width=2pt] (0.05,3.50)--(-0.15,3.50);
				\node[anchor=east,scale=0.50,inner sep=1pt] (A) at (-0.15,3.50) {$$};
				\draw[line width=2pt] (0.50,3.95)--(0.50,4.15);
		\node[anchor=south,scale=0.50,inner sep=1pt] (A) at (0.50,4.15) {$$};
		\draw[line width=2pt] (1.50,3.95)--(1.50,4.15);
		\node[anchor=south,scale=0.50,inner sep=1pt] (A) at (1.50,4.15) {$$};
		\draw[line width=2pt] (2.50,3.95)--(2.50,4.15);
		\node[anchor=south,scale=0.50,inner sep=1pt] (A) at (2.50,4.15) {$$};
		\draw[line width=2pt] (3.50,3.95)--(3.50,4.15);
		\node[anchor=south,scale=0.50,inner sep=1pt] (A) at (3.50,4.15) {$$};
		\drawtree{2/3/4/15,1/3/4/30,0/3/4/45}\end{tikzpicture}
			}
&
\scalebox{0.80}{
	\begin{tikzpicture}[xscale=0.5,yscale=0.5]
	\draw[line width=0.5pt,dashed,opacity=\gridop] (0.00,1.00)--(4.00,1.00);\draw[line width=0.5pt,dashed,opacity=\gridop] (0.00,2.00)--(4.00,2.00);\draw[line width=0.5pt,dashed,opacity=\gridop] (0.00,3.00)--(4.00,3.00);\draw[line width=0.5pt,dashed,opacity=\gridop] (0.00,4.00)--(0.00,1.00);\draw[line width=0.5pt,dashed,opacity=\gridop] (1.00,4.00)--(1.00,1.00);\draw[line width=0.5pt,dashed,opacity=\gridop] (2.00,4.00)--(2.00,1.00);\draw[line width=0.5pt,dashed,opacity=\gridop] (3.00,4.00)--(3.00,1.00);
	\godiagx{1/3,3/3,1/2,2/2,0/1,3/1}{3/2,1/1,2/1}{0/3,2/3,0/2}
	\draw[line width=1.5pt,black!50] (0.00,0.95)--(0.00,4.05);\draw[line width=1.5pt,black!50] (-0.05,4.00)--(4.05,4.00);\draw[line width=1.5pt,black!50] (-0.05,1.00)--(4.05,1.00);\draw[line width=1.5pt,black!50] (4.00,0.95)--(4.00,2.05);\draw[line width=2pt] (0.50,1.05)--(0.50,0.85);
					\node[anchor=north,scale=0.50,inner sep=1pt] (A) at (0.50,0.85) {$$};
					\draw[line width=2pt] (1.50,1.05)--(1.50,0.85);
					\node[anchor=north,scale=0.50,inner sep=1pt] (A) at (1.50,0.85) {$$};
					\draw[line width=2pt] (2.50,1.05)--(2.50,0.85);
					\node[anchor=north,scale=0.50,inner sep=1pt] (A) at (2.50,0.85) {$$};
					\draw[line width=2pt] (3.50,1.05)--(3.50,0.85);
					\node[anchor=north,scale=0.50,inner sep=1pt] (A) at (3.50,0.85) {$$};
					\draw[line width=2pt] (3.95,1.50)--(4.15,1.50);
				\node[anchor=west,scale=0.50,inner sep=1pt] (A) at (4.15,1.50) {$$};
				\draw[line width=2pt] (0.05,1.50)--(-0.15,1.50);
				\node[anchor=east,scale=0.50,inner sep=1pt] (A) at (-0.15,1.50) {$$};
				\draw[line width=1.5pt,black!50] (3.95,2.00)--(4.05,2.00);\draw[line width=1.5pt,black!50] (4.00,1.95)--(4.00,3.05);\draw[line width=2pt] (3.95,2.50)--(4.15,2.50);
				\node[anchor=west,scale=0.50,inner sep=1pt] (A) at (4.15,2.50) {$$};
				\draw[line width=2pt] (0.05,2.50)--(-0.15,2.50);
				\node[anchor=east,scale=0.50,inner sep=1pt] (A) at (-0.15,2.50) {$$};
				\draw[line width=1.5pt,black!50] (3.95,3.00)--(4.05,3.00);\draw[line width=1.5pt,black!50] (4.00,2.95)--(4.00,4.05);\draw[line width=2pt] (3.95,3.50)--(4.15,3.50);
				\node[anchor=west,scale=0.50,inner sep=1pt] (A) at (4.15,3.50) {$$};
				\draw[line width=2pt] (0.05,3.50)--(-0.15,3.50);
				\node[anchor=east,scale=0.50,inner sep=1pt] (A) at (-0.15,3.50) {$$};
				\draw[line width=2pt] (0.50,3.95)--(0.50,4.15);
		\node[anchor=south,scale=0.50,inner sep=1pt] (A) at (0.50,4.15) {$$};
		\draw[line width=2pt] (1.50,3.95)--(1.50,4.15);
		\node[anchor=south,scale=0.50,inner sep=1pt] (A) at (1.50,4.15) {$$};
		\draw[line width=2pt] (2.50,3.95)--(2.50,4.15);
		\node[anchor=south,scale=0.50,inner sep=1pt] (A) at (2.50,4.15) {$$};
		\draw[line width=2pt] (3.50,3.95)--(3.50,4.15);
		\node[anchor=south,scale=0.50,inner sep=1pt] (A) at (3.50,4.15) {$$};
		\drawtree{1/3/4/30,1/2/4/15,0/3/4/45}\end{tikzpicture}
			}
&
\scalebox{0.80}{
	\begin{tikzpicture}[xscale=0.5,yscale=0.5]
	\draw[line width=0.5pt,dashed,opacity=\gridop] (0.00,1.00)--(4.00,1.00);\draw[line width=0.5pt,dashed,opacity=\gridop] (0.00,2.00)--(4.00,2.00);\draw[line width=0.5pt,dashed,opacity=\gridop] (0.00,3.00)--(4.00,3.00);\draw[line width=0.5pt,dashed,opacity=\gridop] (0.00,4.00)--(0.00,1.00);\draw[line width=0.5pt,dashed,opacity=\gridop] (1.00,4.00)--(1.00,1.00);\draw[line width=0.5pt,dashed,opacity=\gridop] (2.00,4.00)--(2.00,1.00);\draw[line width=0.5pt,dashed,opacity=\gridop] (3.00,4.00)--(3.00,1.00);
	\godiagx{2/3,3/3,0/2,3/2,0/1,1/1}{2/2,2/1,3/1}{0/3,1/3,1/2}
	\draw[line width=1.5pt,black!50] (0.00,0.95)--(0.00,4.05);\draw[line width=1.5pt,black!50] (-0.05,4.00)--(4.05,4.00);\draw[line width=1.5pt,black!50] (-0.05,1.00)--(4.05,1.00);\draw[line width=1.5pt,black!50] (4.00,0.95)--(4.00,2.05);\draw[line width=2pt] (0.50,1.05)--(0.50,0.85);
					\node[anchor=north,scale=0.50,inner sep=1pt] (A) at (0.50,0.85) {$$};
					\draw[line width=2pt] (1.50,1.05)--(1.50,0.85);
					\node[anchor=north,scale=0.50,inner sep=1pt] (A) at (1.50,0.85) {$$};
					\draw[line width=2pt] (2.50,1.05)--(2.50,0.85);
					\node[anchor=north,scale=0.50,inner sep=1pt] (A) at (2.50,0.85) {$$};
					\draw[line width=2pt] (3.50,1.05)--(3.50,0.85);
					\node[anchor=north,scale=0.50,inner sep=1pt] (A) at (3.50,0.85) {$$};
					\draw[line width=2pt] (3.95,1.50)--(4.15,1.50);
				\node[anchor=west,scale=0.50,inner sep=1pt] (A) at (4.15,1.50) {$$};
				\draw[line width=2pt] (0.05,1.50)--(-0.15,1.50);
				\node[anchor=east,scale=0.50,inner sep=1pt] (A) at (-0.15,1.50) {$$};
				\draw[line width=1.5pt,black!50] (3.95,2.00)--(4.05,2.00);\draw[line width=1.5pt,black!50] (4.00,1.95)--(4.00,3.05);\draw[line width=2pt] (3.95,2.50)--(4.15,2.50);
				\node[anchor=west,scale=0.50,inner sep=1pt] (A) at (4.15,2.50) {$$};
				\draw[line width=2pt] (0.05,2.50)--(-0.15,2.50);
				\node[anchor=east,scale=0.50,inner sep=1pt] (A) at (-0.15,2.50) {$$};
				\draw[line width=1.5pt,black!50] (3.95,3.00)--(4.05,3.00);\draw[line width=1.5pt,black!50] (4.00,2.95)--(4.00,4.05);\draw[line width=2pt] (3.95,3.50)--(4.15,3.50);
				\node[anchor=west,scale=0.50,inner sep=1pt] (A) at (4.15,3.50) {$$};
				\draw[line width=2pt] (0.05,3.50)--(-0.15,3.50);
				\node[anchor=east,scale=0.50,inner sep=1pt] (A) at (-0.15,3.50) {$$};
				\draw[line width=2pt] (0.50,3.95)--(0.50,4.15);
		\node[anchor=south,scale=0.50,inner sep=1pt] (A) at (0.50,4.15) {$$};
		\draw[line width=2pt] (1.50,3.95)--(1.50,4.15);
		\node[anchor=south,scale=0.50,inner sep=1pt] (A) at (1.50,4.15) {$$};
		\draw[line width=2pt] (2.50,3.95)--(2.50,4.15);
		\node[anchor=south,scale=0.50,inner sep=1pt] (A) at (2.50,4.15) {$$};
		\draw[line width=2pt] (3.50,3.95)--(3.50,4.15);
		\node[anchor=south,scale=0.50,inner sep=1pt] (A) at (3.50,4.15) {$$};
		\drawtree{2/3/4/15,0/3/4/45,0/1/4/15}\end{tikzpicture}
			}
&
\scalebox{0.80}{
	\begin{tikzpicture}[xscale=0.5,yscale=0.5]
	\draw[line width=0.5pt,dashed,opacity=\gridop] (0.00,1.00)--(4.00,1.00);\draw[line width=0.5pt,dashed,opacity=\gridop] (0.00,2.00)--(4.00,2.00);\draw[line width=0.5pt,dashed,opacity=\gridop] (0.00,3.00)--(4.00,3.00);\draw[line width=0.5pt,dashed,opacity=\gridop] (0.00,4.00)--(0.00,1.00);\draw[line width=0.5pt,dashed,opacity=\gridop] (1.00,4.00)--(1.00,1.00);\draw[line width=0.5pt,dashed,opacity=\gridop] (2.00,4.00)--(2.00,1.00);\draw[line width=0.5pt,dashed,opacity=\gridop] (3.00,4.00)--(3.00,1.00);
	\godiagx{0/3,3/3,1/2,2/2,0/1,2/1}{3/2,1/1,3/1}{1/3,2/3,0/2}
	\draw[line width=1.5pt,black!50] (0.00,0.95)--(0.00,4.05);\draw[line width=1.5pt,black!50] (-0.05,4.00)--(4.05,4.00);\draw[line width=1.5pt,black!50] (-0.05,1.00)--(4.05,1.00);\draw[line width=1.5pt,black!50] (4.00,0.95)--(4.00,2.05);\draw[line width=2pt] (0.50,1.05)--(0.50,0.85);
					\node[anchor=north,scale=0.50,inner sep=1pt] (A) at (0.50,0.85) {$$};
					\draw[line width=2pt] (1.50,1.05)--(1.50,0.85);
					\node[anchor=north,scale=0.50,inner sep=1pt] (A) at (1.50,0.85) {$$};
					\draw[line width=2pt] (2.50,1.05)--(2.50,0.85);
					\node[anchor=north,scale=0.50,inner sep=1pt] (A) at (2.50,0.85) {$$};
					\draw[line width=2pt] (3.50,1.05)--(3.50,0.85);
					\node[anchor=north,scale=0.50,inner sep=1pt] (A) at (3.50,0.85) {$$};
					\draw[line width=2pt] (3.95,1.50)--(4.15,1.50);
				\node[anchor=west,scale=0.50,inner sep=1pt] (A) at (4.15,1.50) {$$};
				\draw[line width=2pt] (0.05,1.50)--(-0.15,1.50);
				\node[anchor=east,scale=0.50,inner sep=1pt] (A) at (-0.15,1.50) {$$};
				\draw[line width=1.5pt,black!50] (3.95,2.00)--(4.05,2.00);\draw[line width=1.5pt,black!50] (4.00,1.95)--(4.00,3.05);\draw[line width=2pt] (3.95,2.50)--(4.15,2.50);
				\node[anchor=west,scale=0.50,inner sep=1pt] (A) at (4.15,2.50) {$$};
				\draw[line width=2pt] (0.05,2.50)--(-0.15,2.50);
				\node[anchor=east,scale=0.50,inner sep=1pt] (A) at (-0.15,2.50) {$$};
				\draw[line width=1.5pt,black!50] (3.95,3.00)--(4.05,3.00);\draw[line width=1.5pt,black!50] (4.00,2.95)--(4.00,4.05);\draw[line width=2pt] (3.95,3.50)--(4.15,3.50);
				\node[anchor=west,scale=0.50,inner sep=1pt] (A) at (4.15,3.50) {$$};
				\draw[line width=2pt] (0.05,3.50)--(-0.15,3.50);
				\node[anchor=east,scale=0.50,inner sep=1pt] (A) at (-0.15,3.50) {$$};
				\draw[line width=2pt] (0.50,3.95)--(0.50,4.15);
		\node[anchor=south,scale=0.50,inner sep=1pt] (A) at (0.50,4.15) {$$};
		\draw[line width=2pt] (1.50,3.95)--(1.50,4.15);
		\node[anchor=south,scale=0.50,inner sep=1pt] (A) at (1.50,4.15) {$$};
		\draw[line width=2pt] (2.50,3.95)--(2.50,4.15);
		\node[anchor=south,scale=0.50,inner sep=1pt] (A) at (2.50,4.15) {$$};
		\draw[line width=2pt] (3.50,3.95)--(3.50,4.15);
		\node[anchor=south,scale=0.50,inner sep=1pt] (A) at (3.50,4.15) {$$};
		\drawtree{0/3/4/45,1/2/4/15,0/2/4/30}\end{tikzpicture}
			}
&
\scalebox{0.80}{
	\begin{tikzpicture}[xscale=0.5,yscale=0.5]
	\draw[line width=0.5pt,dashed,opacity=\gridop] (0.00,1.00)--(4.00,1.00);\draw[line width=0.5pt,dashed,opacity=\gridop] (0.00,2.00)--(4.00,2.00);\draw[line width=0.5pt,dashed,opacity=\gridop] (0.00,3.00)--(4.00,3.00);\draw[line width=0.5pt,dashed,opacity=\gridop] (0.00,4.00)--(0.00,1.00);\draw[line width=0.5pt,dashed,opacity=\gridop] (1.00,4.00)--(1.00,1.00);\draw[line width=0.5pt,dashed,opacity=\gridop] (2.00,4.00)--(2.00,1.00);\draw[line width=0.5pt,dashed,opacity=\gridop] (3.00,4.00)--(3.00,1.00);
	\godiagx{0/3,3/3,0/2,2/2,0/1,1/1}{3/2,2/1,3/1}{1/3,2/3,1/2}
	\draw[line width=1.5pt,black!50] (0.00,0.95)--(0.00,4.05);\draw[line width=1.5pt,black!50] (-0.05,4.00)--(4.05,4.00);\draw[line width=1.5pt,black!50] (-0.05,1.00)--(4.05,1.00);\draw[line width=1.5pt,black!50] (4.00,0.95)--(4.00,2.05);\draw[line width=2pt] (0.50,1.05)--(0.50,0.85);
					\node[anchor=north,scale=0.50,inner sep=1pt] (A) at (0.50,0.85) {$$};
					\draw[line width=2pt] (1.50,1.05)--(1.50,0.85);
					\node[anchor=north,scale=0.50,inner sep=1pt] (A) at (1.50,0.85) {$$};
					\draw[line width=2pt] (2.50,1.05)--(2.50,0.85);
					\node[anchor=north,scale=0.50,inner sep=1pt] (A) at (2.50,0.85) {$$};
					\draw[line width=2pt] (3.50,1.05)--(3.50,0.85);
					\node[anchor=north,scale=0.50,inner sep=1pt] (A) at (3.50,0.85) {$$};
					\draw[line width=2pt] (3.95,1.50)--(4.15,1.50);
				\node[anchor=west,scale=0.50,inner sep=1pt] (A) at (4.15,1.50) {$$};
				\draw[line width=2pt] (0.05,1.50)--(-0.15,1.50);
				\node[anchor=east,scale=0.50,inner sep=1pt] (A) at (-0.15,1.50) {$$};
				\draw[line width=1.5pt,black!50] (3.95,2.00)--(4.05,2.00);\draw[line width=1.5pt,black!50] (4.00,1.95)--(4.00,3.05);\draw[line width=2pt] (3.95,2.50)--(4.15,2.50);
				\node[anchor=west,scale=0.50,inner sep=1pt] (A) at (4.15,2.50) {$$};
				\draw[line width=2pt] (0.05,2.50)--(-0.15,2.50);
				\node[anchor=east,scale=0.50,inner sep=1pt] (A) at (-0.15,2.50) {$$};
				\draw[line width=1.5pt,black!50] (3.95,3.00)--(4.05,3.00);\draw[line width=1.5pt,black!50] (4.00,2.95)--(4.00,4.05);\draw[line width=2pt] (3.95,3.50)--(4.15,3.50);
				\node[anchor=west,scale=0.50,inner sep=1pt] (A) at (4.15,3.50) {$$};
				\draw[line width=2pt] (0.05,3.50)--(-0.15,3.50);
				\node[anchor=east,scale=0.50,inner sep=1pt] (A) at (-0.15,3.50) {$$};
				\draw[line width=2pt] (0.50,3.95)--(0.50,4.15);
		\node[anchor=south,scale=0.50,inner sep=1pt] (A) at (0.50,4.15) {$$};
		\draw[line width=2pt] (1.50,3.95)--(1.50,4.15);
		\node[anchor=south,scale=0.50,inner sep=1pt] (A) at (1.50,4.15) {$$};
		\draw[line width=2pt] (2.50,3.95)--(2.50,4.15);
		\node[anchor=south,scale=0.50,inner sep=1pt] (A) at (2.50,4.15) {$$};
		\draw[line width=2pt] (3.50,3.95)--(3.50,4.15);
		\node[anchor=south,scale=0.50,inner sep=1pt] (A) at (3.50,4.15) {$$};
		\drawtree{0/3/4/45,0/2/4/30,0/1/4/15}\end{tikzpicture}
			}\end{tabular}
}
  \caption{\label{fig:NC_alt} For $n=2k+1$, maximal $\fkn$-Deograms are in bijection with non-crossing alternating trees; see \cref{rmk:NC_alt}.}
\end{figure}

\begin{remark}\label{rmk:NC_alt}
For the case $n=2k+1$ of the standard Catalan numbers, the maximal $\fkn$-Deograms are easily seen to be in bijection with \emph{non-crossing alternating trees on $k+1$ vertices} (item~62 in~\cite{StaCat}). Explicitly, given $\fil\in\Gom_\fkn$, we assign a vertex to each of the $k+1$ columns of $\fil$. One can show that every row of $\fil$ must contain exactly two elbows, and connecting the two corresponding vertices by an edge for each of the $k$ rows, one obtains a non-crossing alternating tree. The case $k=3$, $n=7$ is illustrated in \cref{fig:NC_alt}.
\end{remark}
\begin{remark}
A recursive proof of \cref{prop:Go} for the case $n=dk\pm1$ ($d\geq 2$) was found by David Speyer (private communication). It appears that when $n=dk\pm1$, the distinguished condition is automatically satisfied for any maximal $\fkn$-Deogram.  However, this is not the case for instance when $k=5$ and $n=12$; see \figref{fig:Go_ex}(right). We were able to find a recursive proof of~\eqref{eq:Cat=Go} for arbitrary $k,n$. This and some other enumerative consequences of our results will appear in a separate paper~\cite{GL_cat_combin}.
\end{remark}
\begin{remark}
A probabilistic interpretation of $f$-Deograms and their weights in~\eqref{eq:Deodhar} in terms of \emph{the stochastic colored six-vertex model}~\cite{KMMO}  was recently discovered in~\cite{flip}. In particular, a result closely related  to \cref{thm:traces} appears in~\cite[Lemma~7.1 and Proposition~7.3]{flip}.
\end{remark}

\bibliographystyle{alpha}
\bibliography{qtcat}

\end{document}